\documentclass[a4paper,11pt]{amsart}

\providecommand{\bysame}{\leavevmode\hbox to3em{\hrulefill}\thinspace}
\providecommand{\MR}{\relax\ifhmode\unskip\space\fi MR }

\providecommand{\href}[2]{#2}

\usepackage{enumerate, booktabs}
\usepackage{amsmath, amsfonts, amssymb, amsthm, stmaryrd}
\usepackage[all]{xy}
\usepackage[height=22.5cm, width=15.5cm]{geometry}
\usepackage[active]{srcltx}

\numberwithin{equation}{section}

\theoremstyle{plain}
\newtheorem{theorem}{Theorem}[section]
\newtheorem{proposition}[theorem]{Proposition}
\newtheorem{lemma}[theorem]{Lemma}
\newtheorem{corollary}[theorem]{Corollary}
\newtheorem*{thm*}{Theorem}

\newtheorem*{prop*}{Proposition}
\theoremstyle{definition}

\theoremstyle{remark}
\newtheorem{rem}[theorem]{Remark}
\newtheorem*{remark}{Remark}
\newtheorem*{ex}{Example}

\newcommand{\mbb}[1]{\mathbb{#1}}

\newcommand{\ol}[1]{\overline{#1}}
\newcommand{\lie}[1]{{\mathfrak{#1}}}

\DeclareMathOperator{\im}{Im}
\DeclareMathOperator{\re}{Re}
\DeclareMathOperator{\id}{id}

\DeclareMathOperator{\Ad}{Ad}

\DeclareMathOperator{\End}{End}
\DeclareMathOperator{\Aut}{Aut}

\DeclareMathOperator{\tr}{Tr}

\DeclareMathOperator{\rk}{rk}

\DeclareMathOperator{\codim}{codim}

\DeclareMathOperator{\Irr}{Irr}

\newcommand{\p}{\mathfrak{p}}
\newcommand{\g}{\mathfrak{g}}

\newcommand{\so}{\mathfrak{so}}
\newcommand{\su}{\mathfrak{su}}
\newcommand{\lisp}{\mathfrak{sp}}
\newcommand{\liu}{\mathfrak{u}}

\newcommand{\RR}{\mathbb{R}}

\newcommand{\CC}{\mathbb{C}}

\newcommand{\HH}{\mathbb{H}}

\newcommand{\SU}{\mathrm{SU}}
\newcommand{\Sp}{\mathrm{Sp}}
\newcommand{\U}{\mathrm{U}}
\newcommand{\OO}{\mathrm{O}}
\newcommand{\GL}{\mathrm{GL}}
\newcommand{\SO}{\mathrm{SO}}

\newcommand{\Spin}{\mathrm{Spin}}
\newcommand{\spin}{\mathfrak{spin}}
\newcommand{\f}{\mathfrak{f}}
\newcommand{\J}{\mathcal{I}_\mu}
\newcommand{\I}{\mathcal{I}}
\newcommand{\h}{\mathcal{I}_\lambda}

\author{Lisa Knauss}
\email{lisa.knauss@ruhr-uni-bochum.de}

\author{Christian Miebach}
\address{Univ.~Littoral C\^ote d'Opale, EA 2797 - LMPA - Laboratoire de 
math\'ematiques pures et appliqu\'ees Joseph Liouville, F-62228 Calais, 
France}
\email{christian.miebach@univ-littoral.fr}

\title[Spherical algebraic subalgebras]{Classification of spherical algebraic
subalgebras of real simple Lie algebras of rank 1}

\begin{document}

\setlength\arraycolsep{2pt}

\begin{abstract}
We determine all spherical algebraic subalgebras in any simple Lie algebra of
real rank $1$.
\end{abstract}

\maketitle

\section{Introduction}

Let $U^\mbb{C}$ be a complex reductive group with maximal compact subgroup $U$.
It has been proved in~\cite{Br} (see also~\cite{HW}) that smooth compact complex
spherical $U^\mbb{C}$-varieties $Z$ may be characterized by the fact that a
moment map $\mu\colon Z\to\lie{u}^*$ separates the $U$-orbits in $Z$.

In~\cite{HS} it has been shown that the so-called gradient maps are the right
analogue for moment maps when one is interested in actions of a real reductive
group $G=K\exp(\lie{p})$. Spherical gradient manifolds have been introduced
in~\cite{MS} in order to carry over Brion's theorem to the real reductive case.
To be more precise, we call a $G$-gradient manifold $X\subset Z$ with gradient map
$\mu_\lie{p}\colon X\to\lie{p}$ spherical if a minimal parabolic subgroup of $G$
has an open orbit in $X$. If $G$ is connected complex reductive, then a minimal
parabolic subgroup is the same as a Borel subgroup of $G$, so that there is no
ambiguity in this definition. The main result of~\cite{MS} states that $X$ is
spherical if and only if $\mu_\lie{p}$ almost separates the $K$-orbits in $X$.

Recently, real spherical manifolds have attracted attention from the 
representation theoretical view point (see~\cite{KO} and~\cite{KS2}) as well as
from a geometric one (see~\cite{KS} and~\cite{KKS}). In~\cite{Bien} 
and~\cite{KS} the authors have shown that, given a homogeneous real spherical
manifold $X=G/H$, any minimal parabolic subgroup of $G$ has only finitely many
orbits in $X$. Moreover, the paper~\cite{KS} contains the list of all reductive
spherical subalgebras of $\lie{g}=\lie{so}(n,1)$. In~\cite{Ma} the author has
found \emph{all} decompositions of $\lie{so}(n,1)$ as the sum of two
subalgebras. In~\cite{KKPS} the authors classify the reductive spherical
subalgebras of arbitrary simple real Lie algebras.

As the main result of this paper we describe the \emph{non-reductive} spherical 
algebraic subalgebras of $\lie{g}$ where $\lie{g}$ is a simple Lie algebra of 
real rank $1$ by methods in the spirit of~\cite{MS}. We then apply this result 
to classify the reductive spherical subalgebras of $\lie{g}$, thus obtaining a 
second proof of the rank one case in~\cite{KKPS}. Although a subalgebra 
$\lie{h}$ of $\lie{g}$ is spherical whenever $\lie{h}^\mbb{C}$ is a complex 
spherical subalgebra of $\lie{g}^\mbb{C}$, we would like to stress the fact that 
the converse is not true (see the example in Section~3). In particular one 
cannot reduce the question to the complex classification. In contrast to the 
complex case, there are continuous families of spherical subalgebras of 
$\lie{g}$. For $\lie{g}=\lie{su}(n,1)$ the geometry of such a family was studied 
in detail by means of an explicit slice model in~\cite{K}.

Let us outline the main steps of the proof as well as the organization of this
paper. In Section~2 we show that the homogeneous manifold $X=G/H$ admits a
$G$-gradient map if and only if $H$ is an algebraic subgroup of $G$. This is
the reason why we classify spherical \emph{algebraic} subalgebras. In Section~3
we characterize reductive and non-reductive spherical algebraic subgroups $H$ of
$G$ by the fact that a maximal compact subgroup of $H$ acts transitively on the
spheres in a certain representation related to the inclusion $H\hookrightarrow
G$. More precisely, our starting point is the following, see 
Propositions~\ref{Prop:redspherical} and~\ref{Prop:nonredspherical}.

\begin{proposition}\label{intro}
Let $G=K\exp(\p)$ be a connected simple Lie group of real rank $1$ with Iwasawa 
decomposition $G=KAN$. Let $M:=\mathcal{N}_K(\lie{a})$.
\begin{enumerate}[(1)]
\item Let $H=K_H\exp(\p_H)$ be a reductive algebraic subgroup of $G$. Then $H$
is spherical if and only if $K_H$ acts transitively on the connected components
of the spheres in $\p_H^\perp$.
\item Let $H=M_HA_HN_H$ be a non-reductive algebraic subgroup of $G$. If 
$N_H=N$, then $H$ is spherical. If $\dim N/N_H\geq1$, then $H$ is spherical if 
and only if $A_H=A$ and $M_H$ acts transitively on the connected components
of the spheres in $\lie{n}_H^\perp$.  
\end{enumerate}
\end{proposition}

Suppose that $H=M_HA_HN_H$ is a non-reductive spherical algebraic subgroup of 
$G$ with $\dim N/N_H\geq1$. Let us write 
$\lie{n}=\lie{g}_\alpha\oplus\lie{g}_{2\alpha}$ where 
$\lie{g}_\alpha$ and $\lie{g}_{2\alpha}$ are the restricted root spaces with 
respect to the maximal Abelian subspace $\lie{a}$ of $\lie{p}$. Let us fix an 
$M$-invariant scalar product on $\lie{g}_\alpha$. Then the Lie algebra of 
$H=M_HA_HN_H$ is of the form $\lie{h}=\lie{m}_H\oplus\lie{a}\oplus 
W^\perp\oplus\lie{g}_{2\alpha}$ where $W\subset\lie{g}_\alpha$ is an 
$M_H$-stable subspace such that $M_H$ acts transitively on the connected 
components of the spheres in $W$, see 
Corollary~\ref{Cor:nonredfinestructure}. Thus, in a second step we will 
determine the subspaces $W$ of $\lie{g}_\alpha$ such that $\mathcal{N}_M(W)$,
which contains $M_H$, acts irreducibly on $W$. It turns out that in many cases 
this already implies that $\mathcal{N}_M(W)$ acts transitively on the connected 
components of the spheres in $W$. In the final step, we make use of Onishchik's 
classification of transitive actions on spheres in order to find all subgroups 
of $\mathcal{N}_M(W)$ that still act transitively on the connected components 
of the spheres in $W$. More details of this general scheme are given in 
Section~4. In the remaining Sections~5 to~8 we carry out our program case by 
case for $\lie{so}(n,1)$, $\lie{su}(n,1)$, $\lie{sp}(n,1)$ and $\lie{f}_4$. The 
tables containing all spherical algebraic subalgebras are given 
in Theorems~\ref{NonredOrthogonal} and~\ref{ReductiveOrthogonal} for 
$\lie{g}=\lie{so}(n,1)$, in Theorems~\ref{nonredUnitary} 
and~\ref{ReductiveUnitary} for $\lie{g}=\lie{su}(n,1)$, in 
Theorems~\ref{NonredSympl} and~\ref{ReductiveSymplectic} for 
$\lie{g}=\lie{sp}(n,1)$, and in Theorems~\ref{ReductiveExceptional} 
and~\ref{NonredExc} for the exceptional Lie algebra 
$\lie{g}=\lie{f}_4=\lie{f}_{4(-20)}$.

After this paper was finished, we learned that Kimelfeld has considered the
classification problem of algebraic spherical subgroups in real simple Lie groups
$G=K\exp(\lie{p})$ of rank~$1$, too. In~\cite{Ki} he obtains Proposition~\ref{intro}
by differential geometric arguments based on the Karpelevich compactification of the
hyperbolic space $G/K$ and then gives a general description of all spherical
algebraic subgroups of $G$. However, he does not provide an explicit list of all
spherical algebraic subalgebras of $\lie{g}$. Bien obtains an explicit but not 
complete list in \cite{Bien}.

\emph{Acknowledgements.} We would like to thank Peter Heinzner and Valdemar
Tsanov for helpful discussions on the subject presented here. We are much obliged to
Friedrich Knop for informing us about Kimelfeld's paper~\cite{Ki} as well as about an
inaccuracy in the statement of Theorem~\ref{NonredOrthogonal} in an earlier
version of this manuscript. The first author 
gratefully acknowledges the financial support by SPP 1388 ``Darstellungstheorie'' 
and SFB/TR 12 ``Symmetries and Universality in Mesoscopic Systems'' of the DFG.

\section{(Homogeneous) gradient manifolds}

Let $G$ be a connected semisimple Lie group that embeds as a closed subgroup 
into its universal complexification $G^\mbb{C}$. Let $U$ be a compact real form 
of $G^\mbb{C}$ such that $G$ is stable under the corresponding Cartan involution 
of $G^\mbb{C}=U^\mbb{C}$. Then we obtain the Cartan decomposition 
$K\times\lie{p}\to G$, $(k,\xi)\mapsto k\exp(\xi)$, where $K:=G\cap
U$ is a maximal compact subgroup of $G$ and where $\lie{p}:=\lie{g}\cap
i\lie{u}$.

By a \emph{$G$-gradient manifold} we mean the following. Let $Z$ be a K\"ahler
manifold endowed with a holomorphic $G^\mbb{C}$-action and $U$-invariant 
K\"ahler form $\omega$ such that there exists a $U$-equivariant moment map
$\mu\colon Z\to\lie{u}^*$ for the $U$-action on $(Z,\omega)$. We call such a $Z$
a \emph{Hamiltonian $G^\mbb{C}$-manifold}. Basic examples are given by
$G^\mbb{C}$-stable complex submanifolds of some projective space $\mbb{P}(V)$
where $V$ is a finite-dimensional complex $G^\mbb{C}$-representation space. In
particular, homogeneous algebraic $G^\mbb{C}$-varieties are Hamiltonian.
Identifying $\lie{u}^*$ with $i\lie{u}$ and composing $\mu$ with the orthogonal
projection to $\lie{p}\subset i\lie{u}$ with respect to a $U$-invariant inner
product, we obtain the $K$-equivariant \emph{gradient map} $\mu_\lie{p}\colon
Z\to\lie{p}$. Any $G$-stable closed real submanifold $X$ of $Z$ is called a
$G$-gradient manifold with gradient map $\mu_\lie{p}|_X$. For more details and
the basic properties of gradient maps we refer the reader to~\cite{HS}
and~\cite{HSS}.

\begin{proposition}
Let $G$ be a connected semisimple Lie group and let $H$ be a closed subgroup of
$G$. Then $X=G/H$ is a $G$-gradient manifold if and only if $H$ is algebraic,
i.e., if $\ol{H}\cap G=H$ holds for the Zariski closure $\ol{H}$ of $H$ in
$G^\mbb{C}$. 
\end{proposition}

\begin{proof}
Suppose first that $\ol{H}\cap G=H$ holds. It follows that $X=G/H$ is a real
submanifold of the homogeneous space $G^\mbb{C}/\ol{H}$. Since the latter is
quasi-projective, it is in particular K\"ahler, and since $G^\mbb{C}$ is
semisimple, there exists a unique $U$-equivariant moment map on
$G^\mbb{C}/\ol{H}$. Then the construction described above yields a $G$-gradient
map on $X=G/H$.

If $X=G/H$ is a gradient manifold, then by definition there exists a
$G$-equivariant diffeomorphism $G/H\cong G\cdot z\subset Z$ where $Z$ is a
Hamiltonian $G^\mbb{C}$-manifold. From this we obtain $G/H\hookrightarrow
G^\mbb{C}/(G^\mbb{C})_z$, and $G^\mbb{C}/(G^\mbb{C})_z$ is again K\"ahler.
Therefore $(G^\mbb{C})_z$ is an algebraic subgroup of $G^\mbb{C}$,
see~\cite[Corollary~4.12]{GMO}, hence contains $\ol{H}$. This implies
$\ol{H}\cap
G\subset (G^\mbb{C})_z\cap G=G_z=H$, as was to be shown.
\end{proof}

\section{Characterization of spherical homogeneous gradient manifolds}

As in the previous section let $G=K\exp(\lie{p})$ be a connected semisimple Lie
group which embeds into its complexification $G^\mbb{C}$. Let $H\subset G$ be a 
closed subgroup such that $X=G/H$ is a $G$-gradient manifold with gradient map 
$\mu_\lie{p}\colon X\to\lie{p}$. We say that $X=G/H$ is \emph{spherical} if a 
minimal parabolic subgroup of $G$ has an open orbit in $X$. In this case we call 
$H$ a spherical subgroup of $G$ and $\lie{h}$ a spherical subalgebra of 
$\lie{g}$. As shown in~\cite{MS} sphericity of $X$ is equivalent to the fact 
that $\mu_\lie{p}$ almost separates the $K$-orbits in $X$, i.e., that the map 
$X/K\to\lie{p}/K$ induced by $\mu_\lie{p}\colon X\to\lie{p}$ has discrete
fibers.

\begin{ex}
Any symmetric subalgebra of $\lie{g}$ is spherical in $\lie{g}$,
see~\cite[\S6.1]{MS}.
\end{ex}

It is not hard  to see that, if $\lie{h}^\mbb{C}$ is spherical in $\lie{g}^\mbb{C}$ 
in the usual sense, then $\lie{h}$ is spherical in $\lie{g}$. However, as the 
following examples show, the converse does not hold, i.e., there are more 
spherical subalgebras of $\lie{g}$ than just real forms of complex spherical 
subalgebras of $\lie{g}^\mbb{C}$.

\begin{ex}
The unipotent radical $N$ of a minimal parabolic subgroup of $G$ is always
spherical in $G$. However, in general $N^\mbb{C}$ is not a spherical subgroup of
$G^\mbb{C}$. As a concrete example one may take $G={\rm{SO}}^\circ(5,1)$.

A semisimple example is given by the spherical subgroup $H={\rm{Sp}}(1,1)$ of
$G={\rm{Sp}}(2,1)$ (see Theorem~\ref{ReductiveSymplectic}) since $H^\mbb{C}=
{\rm{Sp}}(2,\mbb{C})$ is not spherical in $G^\mbb{C}={\rm{Sp}}(3,\mbb{C})$,
see~\cite{Kr}.
\end{ex}

For the rest of this paper we assume that \emph{$G$ is simple and has real rank
$1$}. Then $\lie{g}$ is isomorphic to either $\lie{so}(n,1)$ ($n\geq3$) or
$\lie{su}(n,1)$ ($n\geq1$) or $\lie{sp}(n,1)$ ($n\geq2$) or the exceptional Lie
algebra $\lie{f}_4=\lie{f}_{4(-20)}$, see~\cite[Chapter~VI.11]{Kn}.

As our main result we will describe the algebraic spherical subalgebras of
$\lie{g}$ up to conjugation by an element of $G$, i.e., those subalgebras
$\lie{h}$ for which there exists a closed subgroup $H\subset G$ having $\lie{h}$
as Lie algebra such that $X=G/H$ is a spherical gradient manifold.

For the classification we distinguish the cases that $H$ is reductive or not. If
$H$ is reductive, then according to~\cite[Theorem~6.3.6]{OV}, after conjugation 
by an element of $G$, we have a Cartan decomposition $H=K_H\exp(\lie{p}_H)$ 
where $K_H:=H\cap K$ is a maximal compact subgroup of $H$ and where 
$\lie{p}_H:=\lie{h}\cap\lie{p}$ is a $K_H$-invariant subspace of $\lie{p}$ with 
$[\lie{p}_H,\lie{p}_H]\subset\lie{k}_H$. We write 
$\lie{p}=\lie{p}_H\oplus\lie{p}_H^\perp$ with respect to the $K$-invariant inner 
product on $\lie{p}$ that comes from the $U$-invariant inner product on 
$i\lie{u}$. In this situation, sphericity of $X=G/H$ has been characterized 
in~\cite[Proposition~6.1]{MS}. Under the additional assumption $\rk_\mbb{R}G=1$ 
we can make the following more precise statement.

\begin{proposition}\label{Prop:redspherical}
Let $G=K\exp(\lie{p})$ be a connected simple Lie group of real rank $1$ and let
$H=K_H\exp(\lie{p}_H)\subset G$ be a closed reductive subgroup. Then $X=G/H$ is
spherical if and only if $K_H$ has an open orbit in every sphere in
$\lie{p}_H^\perp\subset\lie{p}$.
\end{proposition}

\begin{remark}
Except for $\codim\lie{p}_H=1$ Proposition~\ref{Prop:redspherical} says that 
$X=G/H$ is spherical if and only if $K_H$ acts transitively on the spheres in
$\lie{p}_H^\perp$. In particular, $X$ can only be spherical if the
$K_H$-representation on $\lie{p}_H^\perp$ is irreducible.
\end{remark}

\begin{proof}[Proof of Proposition~\ref{Prop:redspherical}]
By~\cite[Theorem~1.1]{MS} the homogeneous gradient manifold $X=G/H$ is spherical
if and only if any gradient map on it almost separates the $K$-orbits. The
Mostow decomposition (see~\cite{HS} for a proof using gradient maps) exhibits 
$X$ as $K$-equivariantly isomorphic to the twisted product 
$K\times_{K_H}\lie{p}_H^\perp$ \footnote{If $H$ is a subgroup of $G$ and $Y$ is
a set on which $H$ acts, then the twisted product $G\times_HY$ is defined as the
quotient of $G\times Y$ with respect to the diagonal $H$-action
$h\cdot(g,y):=(gh^{-1},h\cdot y)$.}. A particular gradient map is given by
$\mu_\lie{p}[k,\xi]=-\Ad(k)\xi$ for $k\in K$ and $\xi\in\lie{p}_H^\perp$. Since
the $K$-orbits in $K\times_{K_H}\lie{p}_H^\perp$ correspond to the $K_H$-orbits
in $\lie{p}_H^\perp$, this gradient map $\mu_\lie{p}$ separates the $K$-orbits 
if and only if the map $\lie{p}_H^\perp/K_H\to\lie{p}/K$, induced by the 
inclusion $\lie{p}_H^\perp\hookrightarrow\lie{p}$, has discrete fibers. Since 
these fibers are precisely the $K_H$-orbits in the compact sets 
$(K\cdot\xi)\cap\lie{p}_H^\perp$ with $\xi\in\lie{p}_H^\perp$, we note in
particular that the fibers have to be finite.

In the case $\rk_\mbb{R}G=1$, the $K$-orbits in $\lie{p}$ are spheres, so their
intersections with any subspace of $\lie{p}$ are again spheres and in particular
connected (unless the subspace is a line). Thus, $X$ is spherical if and only if
$K_H$ has an open orbit in every sphere $(K\cdot\xi)\cap\lie{p}_H^\perp$,
$\xi\in\lie{p}_H^\perp$, hence in any sphere in $\lie{p}_H^\perp$.
\end{proof}

In the rest of this section we give a similar criterion for non-reductive $H$.
For this we fix a minimal parabolic subalgebra $\lie{q}_0=\lie{m}\oplus\lie{a}
\oplus\lie{n}$ of $\lie{g}$ such that $\lie{a}$ is a maximal Abelian subspace
(i.e., a line) of $\lie{p}$. The corresponding group is $Q_0=MAN$ where
$M:=\mathcal{Z}_K(\lie{a})$. Let $\lie{h}$ be a non-reductive algebraic
subalgebra of $\lie{g}$ and let $H$ be a corresponding subgroup of $G$. As is
shown in~\cite{BT} (compare~\cite[Lemma~3.1]{KS}), after conjugation by an 
element of $G$ we may assume that 
$\lie{h}=\lie{l}_H\oplus\lie{n}_H$ where $\lie{l}_H\subset\lie{m}\oplus\lie{a}$ 
is reductive in $\lie{g}$ and $\{0\}\not=\lie{n}_H\subset\lie{n}$ is a nilpotent 
ideal of $\lie{h}$. On the group level we have $H\cong L_H\ltimes N_H$ with a 
reductive group $L_H=M_HA_H\subset MA$ and $N_H\subset N$. Note that $L_H$ acts 
by conjugation on $N$ and stabilizes $N_H$, hence acts on $N/N_H$. On the Lie
algebra level we have the decomposition $\lie{n}=\lie{n}_H\oplus\lie{n}_H^\perp$
of $\lie{n}$ as an $M_H$-module.

The following result is proven in~\cite[Lemma~3.2]{KS}, compare
also~\cite[Proposition~1.1]{Br2}. For the reader's convenience we repeat the
argument here.

\begin{proposition}\label{Prop:nonredspherical}
Let $G$ be a connected simple Lie group of real rank $1$ and let
$X=G/H$ be a $G$-gradient manifold such that $H=M_HA_HN_H$ is non-reductive.
Then $X$ is spherical if and only if either
\begin{enumerate}[(1)]
\item $N_H=N$ and $L_H=M_HA_H$ is arbitrary, or
\item $\dim N/N_H\geq1$ and $A_H=A$ and $M_H\subset\mathcal{N}_M(\lie{n}_H)$ 
acts transitively on the connected components of the spheres in 
$\lie{n}_H^\perp$.
\end{enumerate}
\end{proposition}

\begin{remark}
Every algebraic subgroup $H\subset G$ that contains $N$ is spherical. Therefore, 
we will concentrate on the case $\dim\lie{n}_H^\perp\geq1$.
\end{remark}

\begin{proof}[Proof of Proposition~\ref{Prop:nonredspherical}]
Let $H$ be a non-reductive algebraic subgroup of $G$ of the form $H=M_HA_HN_H
\subset Q_0=MAN$. Since we have the $G$-equivariant fiber bundle $X=G/H\to
G/Q_0$, we see that $X$ is $G$-equivariantly diffeomorphic to the twisted
product $G\times_{Q_0}(Q_0/H)$. By~\cite[Corollary~IX.1.8]{He} the unique open orbit 
of the opposite minimal parabolic subgroup $Q_0^-=MAN^-$ in $G/Q_0$ is 
$Q_0^-\cdot eQ_0$. This implies that $X=G/H$ is spherical if and only if there 
is some $xH\in Q_0/H$ such that $Q_0^-\cdot[e,xH]$ is open in 
$G\times_{Q_0}(Q_0/H)$. The latter is the case if and only if $L:=MA=Q_0^-\cap 
Q_0$ has an open orbit in $Q_0/H$. Using the fact that $N_H$ is normal in $H$, 
we see that $Q_0/H$ is $L$-equivariantly diffeomorphic to $L\times_{L_H}(N/N_H)$ 
where $L_H$ acts by conjugation on $N/N_H$. This proves that $X=G/H$ is 
spherical if and only if $L_H$ has an open orbit in $N/N_H$.

If $A_H\neq A$ then $L_H=M_H$ can only have an open orbit in $N/N_H$ if the
latter is compact forcing $\lie{n}=\lie{n}_H$. In particular if $\dim N/N_H\geq
1$ then $X=G/H$ can only be spherical if $A_H=A$. The $L_H$-equivariant
diffeomorphism $\exp:\lie{n}\to N$ induces via $\xi\mapsto\exp(\xi)N_H$ an
$L_H$-equivariant map from $\lie{n}_H^\perp$ to $N/N_H$. It follows 
from~\cite[Lemma~IV.6.8]{He2} that this map is a diffeomorphism. Therefore, if 
$A_H=A$ holds, sphericity of $X$ is equivalent to the fact that $L_H$ has an 
open orbit in $\lie{n}_H^\perp$.

If $\dim\lie{n}_H^\perp=1$, then $L_H=M_HA$ has an open orbit for any
$M_H\subset\mathcal{N}_M(\lie{n}_H)$. Since every $A$-orbit in $\lie{n}_H^\perp$
intersects any sphere in $\lie{n}_H^\perp$ precisely once, the claim follows.
\end{proof}

As we have seen, in order to classify non-reductive spherical algebraic
subalgebras $\lie{h}\subset\lie{g}$ we may assume without loss of generality
$\dim\lie{n}_H^\perp\geq1$. In particular, $M_H$ must act irreducibly on
$\lie{n}_H^\perp$ if $X=G/H$ is spherical. In closing this section, we state
the following corollary of Proposition~\ref{Prop:nonredspherical} which
exploits this observation a little further.

\begin{corollary}\label{Cor:nonredfinestructure}
Let $H=L_HN_H$ be a spherical non-reductive algebraic subgroup of $G$. Then
$\lie{n}_H=(\lie{n}_H\cap\lie{g}_\alpha)\oplus\lie{g}_{2\alpha}$.

Conversely, for every subspace $W\subset\lie{g}_\alpha$ and for every algebraic 
subalgebra $\lie{m}_H\subset\mathcal{N}_{\lie{m}}(W)$ the direct 
sum $\lie{h}_W:=\lie{m}_H\oplus\lie{a}_H\oplus W^\perp\oplus\lie{g}_{2\alpha}$ 
is an algebraic subalgebra of $\lie{g}$. Moreover, $\lie{h}_W$ is spherical if 
and only if $M_HA_H$ has an open orbit in $W$.
\end{corollary}

\begin{proof}
We have $\lie{n}=\lie{g}_\alpha\oplus\lie{g}_{2\alpha}$ according to the
restricted root space decomposition of $\lie{g}$ with respect to $\lie{a}$ where
$\alpha$ is a simple restricted root. If $2\alpha$ is not a restricted root, we
set $\lie{g}_{2\alpha}=0$. (Note that this is only the case for
$\lie{g}=\lie{so}(n,1)$.) Suppose that $X=G/H$ is spherical with
$\dim\lie{n}_H^\perp\geq1$. Since $A\subset L_H$ acts with two different 
weights on $\lie{g}_\alpha$ and $\lie{g}_{2\alpha}$, we obtain
\begin{equation*}
\lie{n}_H^\perp=(\lie{n}_H^\perp\cap\lie{g}_\alpha)\oplus
(\lie{n}_H^\perp\cap\lie{g}_{2\alpha}).
\end{equation*}
Consequently, $M_H$ acts irreducibly on $\lie{n}_H^\perp$ only if $\lie{n}_H$ 
contains $\lie{g}_\alpha$ or $\lie{g}_{2\alpha}$. Since 
$[\lie{g}_\alpha,\lie{g}_\alpha]=\lie{g}_{2\alpha}$, see~\cite[p.~408]{He}, and 
since $\lie{n}_H$ is a Lie algebra, $\lie{g}_{2\alpha}\subset\lie{n}_H$ must 
hold.
\end{proof}

\section{Strategy of the classification}\label{Section:str}

Let us outline the principal steps that will lead to the classification result.
Recall that $G=K\exp(\lie{p})$ is a connected simple Lie group of real rank $1$
that embeds into $G^\mbb{C}$.

Let $\lie{h}\subset\lie{g}$ be an algebraic subalgebra which we assume first to 
be non-reductive. Motivated by Proposition~\ref{Prop:nonredspherical} 
and Corollary~\ref{Cor:nonredfinestructure} we will first determine the real 
subspaces $W\subset\lie{g}_\alpha$ such that $\mathcal{N}_M(W)$ acts irreducibly 
on $W$. Then we will bring $W$ (by an element in $M$) into a suitable normal 
form and calculate $\mathcal{N}_M(W)$. This step will be carried out 
case-by-case for every simple Lie algebra of real rank one. Let us therefore 
identify the relevant representations in each case.

\begin{rem}\label{Rem:KandMaction}
Suppose first that $G={\rm{SO}}^\circ(n,1)$. The $K$-action on $\lie{p}$ is 
isomorphic to the defining representation of $K\cong{\rm{SO}}(n)$ on $\lie{p} 
\cong\mbb{R}^n$. Moreover, the Lie algebra $\lie{n}=\lie{g}_\alpha$ is Abelian
and the action of $M\cong{\rm{SO}}(n-1)$ on $\lie{n}\cong\mbb{R}^{n-1}$ is 
again isomorphic to the defining representation of ${\rm{SO}}(n-1)$.

If $G={\rm{SU}}(n,1)$, then we can choose
\begin{equation*}
K=\left\{
\begin{pmatrix}
A&0\\0&a
\end{pmatrix};\ 
A\in{\rm{U}}(n), a=\det(A)^{-1}\right\}\cong S(\U(n)\times\U(1)).
\end{equation*}
The group $K\cong S(\U(n)\times\U(1))$ acts on $\lie{p}\cong\mbb{C}^n$ by 
$(A,a)\cdot v=Ava^{-1}$. The group $M$ is a $2$-to-$1$-covering of
$S(\U(n-1)\times\U(1))$. Its action on $\lie{g}_\alpha\cong\mbb{C}^{n-1}$ 
factorizes through the covering map $M\to S(\U(n-1)\times\U(1))$, and
the action of $S(\U(n-1)\times\U(1))$ on $\lie{g}_\alpha$ is given by the
analogous formula.

Suppose now that $G={\rm{Sp}}(n,1)$. The group $K\cong{\rm{Sp}}(n)\times 
{\rm{Sp}}(1)$ acts on $\lie{p}\cong\mbb{H}^n$ by $(A,a)\cdot v:=Ava^{-1}$. The 
group $M\cong{\rm{Sp}}(1)\times{\rm{Sp}}(n-1)$ acts in the same way on 
$\lie{g}_\alpha\cong\mbb{H}^{n-1}$.

If $G=F_4$, then the $K$-action on $\lie{p}$ is isomorphic to the unique 
irreducible representation of $\Spin(9)$ on $\mbb{R}^{16}$. The $M$-action 
on $\lie{g}_\alpha$ is equivalent to the unique irreducible representation of 
${\rm{Spin}}(7)$ on $\mbb{R}^8$, see Lemma~\ref{Lem:so(7)subsetso(8)}.
\end{rem}

In the second step we single out those subgroups of $\mathcal{N}_M(W)$ that do 
indeed act transitively on the spheres in $W$. In order to do so, we use 
results of Montgomery and Samelson as well as of Onishchik which we recall here 
for the reader's convenience.

Montgomery and Samelson considered the case of a connected compact Lie group $L$
which acts transitively and effectively on $S^n$ and obtained the following
result, see~\cite[Theorem~I]{MoS}.

\begin{theorem}\label{Thm:MontgomerySamelson}
Let $L$ be a connected compact Lie group acting transitively and effectively on
$S^n$. If $n$ is even, then $L$ is simple, while for $n$ odd $L$ is either 
simple or finitely covered by $L_1\times L_2$, where $L_2$ is either $\SO(2)$
or $\Sp(1)$ and $L_1$ is simple and acts already transitively on $S^n$.
\end{theorem}

In the same paper Montgomery and Samelson found all simple, compact, connected
groups acting transitively on $S^n$ for almost all $n$ (see 
\cite[Theorem II-IV]{MoS}). In the case that $n$ is even, their result was
sharpened in \cite{Borel}. Given any compact group $G$ acting transitively and
effectively on some homogeneous space, Onishchik found all subgroups that
also act transitively in~\cite[Theorem 4.1]{On}. This enabled him to find all
transitive effective actions of connected compact Lie groups on $S^n$ for all
$n$. (see \cite[Theorem 3 in \S 18.3]{On}).

We recall his result in the following theorem where we consider the defining 
representations of $\OO(n)$ on $\mbb{R}^n$, of $\U(n)$ on $\mbb{C}^n$ and of 
$\Sp(n)$ on $\mbb{H}^n$.

\begin{theorem}\label{Thm2On}
Let $K$ be either ${\rm{O}}(n)$ or ${\rm{U}}(n)$ or ${\rm{Sp}}(n)$ and let $V$
denote the respective defining representation of $K$. The following table lists
all connected proper subgroups $L$ of $K$ that act transitively on the spheres
in $V$, up to conjugation in $K$, where $p\colon\Sp(1)\times\Sp(1)\to\SO(4)$ 
is the universal covering and $L_2\subset\Sp(1)$ is an arbitrary connected subgroup.
\setlength\arraycolsep{7pt}
\begingroup
\renewcommand*{\arraystretch}{1.2}
\begin{equation*}
\begin{array}{|c|c|}
\hline
K&L \\ \hline
\U(2k+1), k\geq1 & \SU(2k+1)\\
\U(2k), k\geq 2&\SU(2k),\ \Sp(k)\times \U(1),\ \Sp(k)\\
\OO(2k+1), k\geq0 & \SO(2k+1)\\
\OO(4k+2), k\geq 1& \SO(4k+2),\ \U(2k+1),\ \SU(2k+1)\\
\OO(4k), k=3, k\geq 5&\SO(4k),\ \U(2k),\ \SU(2k),\ \Sp(k)\times\Sp(1),\
\Sp(k)\times \U(1),\ \Sp(k)\\
\end{array}
\end{equation*}
\begin{equation*}
\begin{array}{|c|c|}
\OO(16)&\SO(16),\ \U(8),\ \SU(8),\ \Sp(4)\times\Sp(1),\ \Sp(4)\times \U(1),\
\Sp(4),\ \Spin(9)\\
\OO(8)&\SO(8),\ \U(4),\ \SU(4),\ \Sp(2)\times\Sp(1),\ \Sp(2)\times \U(1),\
\Sp(2),\ \Spin(7)\\
\OO(7)&\SO(7),\ G_2\\
\OO(4) & \SO(4),\ p\bigl(\Sp(1)\times L_2\bigr),\ p\bigl(L_2\times\Sp(1)\bigr)\\
\OO(2)&\SO(2)\\
\hline
\end{array}
\end{equation*}
\endgroup\setlength\arraycolsep{2pt}

 Note that $K=\Sp(n)$ does not contain such a
subgroup.
\end{theorem}

\begin{proof}
Since a connected subgroup of $\OO(n)$ lies in $\SO(n)$ and since $\SO(n)$
($n\geq3$, $n\not=4$) and $\Sp(n)$ are simple, the result follows directly
from~\cite[Table~8, p.~227]{On} for these groups.

Let us discuss the case $K=\OO(4)$. Identifying $\mbb{R}^4$ with $\mbb{H}$ one
sees that $\lie{so}(4)\cong\lie{sp}(1)\oplus\lie{sp}(1)$. The isotropy algebra
$\lie{so}(4)_{e_1}$ of $e_1\in S^3\subset\mbb{R}^4$ corresponds to the diagonal
$\Delta_{\lie{sp}(1)}$ in $\lie{sp}(1)\oplus\lie{sp}(1)$. Therefore we have to
find all subalgebras $\lie{l}$ of $\lie{sp}(1)\oplus\lie{sp}(1)$ (up to
conjugation) that verify 
\begin{equation}\label{Eqn:sum}
\lie{l}+\Delta_{\lie{sp}(1)}=\lie{sp}(1)\oplus\lie{sp}(1).
\end{equation}
Consider the projections
\begin{equation*}
\xymatrix{
\lie{sp}(1)\oplus\lie{sp}(1)\ar[r]^{\pi_2}\ar[d]_{\pi_1} & \lie{sp}(1)\\
\lie{sp}(1). &
}
\end{equation*}
Let $\lie{l}$ be a subalgebra of $\lie{sp}(1)\oplus\lie{sp}(1)$
verifying~\eqref{Eqn:sum}. This implies
\begin{equation*}
3\leq\dim\lie{l}=\dim\pi_1(\lie{l})+\dim\ker(\pi_1|_\lie{l})
\end{equation*}
where $\ker(\pi_1|_\lie{l})=\lie{l}\cap\bigl(\{0\}\oplus\lie{sp}(1)\bigr)$ is
an ideal in $\lie{l}$.

If $\dim\pi_1(\lie{l})=0$, then $\lie{l}=\{0\}\oplus\lie{sp}(1)$. If
$\dim\pi_1(\lie{l})=1$, then $\pi_1(\lie{l})=:\lie{t}$ is a maximal torus in
$\lie{sp}(1)$ and $\dim\lie{l}\cap\bigl(\{0\}\oplus\lie{sp}(1)\bigr)\geq2$,
hence $\lie{l}=\lie{t}\oplus\lie{sp}(1)$.

Finally, suppose that $\dim\pi_1(\lie{l})=3$ and note that $\dim
\lie{l}\cap\bigl(\{0\}\oplus\lie{sp}(1)\bigr)\in\{0,1,3\}$. If $\dim
\lie{l}\cap\bigl(\{0\}\oplus\lie{sp}(1)\bigr)=3$, then $\lie{l}=\lie{sp}(1)
\oplus\lie{sp}(1)$. If $\dim\lie{l}\cap\bigl(\{0\}\oplus\lie{sp}(1)\bigr)=0$,
then $\dim\lie{l}=3$ and thus $\lie{l}$ is simple. Therefore $\lie{l}\cap
\bigl(\lie{sp}(1)\oplus\{0\}\bigr)$ is either $\lie{sp}(1)\oplus\{0\}$ (and
then $\lie{l}=\lie{sp}(1)\oplus\{0\}$) or $\{0\}$. In the latter case
$\pi_1|_\lie{l}$ and $\pi_2|_\lie{l}$ are isomorphisms onto $\lie{sp}(1)$ and
$\lie{l}$ coincides with the graph of $\varphi:=(\pi_2|_\lie{l})\circ(
\pi_1|_\lie{l})^{-1}\in\Aut\bigl(\lie{sp}(1)\bigr)$. Since we can identify
$\lie{l}\cap\Delta_{\lie{sp}(1)}$ with the space of fixed points
$\lie{sp}(1)^\varphi$ in this case and since $\dim\lie{sp}(1)^\varphi\geq1$, we
obtain $\dim \lie{l}+\Delta_{\lie{sp}(1)}\leq5$, hence $\lie{l}$ cannot
verify~\eqref{Eqn:sum}. Finally consider the case $\dim
\lie{l}\cap\bigl(\{0\}\oplus\lie{sp}(1)\bigr)=1$. Then $\dim\lie{l}=4$ and
$\lie{l}\cap\bigl(\{0\}\oplus\lie{sp}(1)\bigr)$ is a one-dimensional ideal in
$\lie{l}$ which implies $\lie{l}\cap\bigl(\{0\}\oplus\lie{sp}(1)\bigr)=\lie{z}
(\lie{l})$. Thus $\lie{l}=\lie{sp}(1)\oplus\lie{z}(\lie{l})$.

In the unitary case we use~\cite[Theorem~1.5.1]{On} in order to reduce
the classification to ${\rm{SU}}(n)$. We denote the isotropy algebra
$\lie{u}(n)_{e_1}$ of $e_1\in S^{2n-1}\subset\mbb{C}^n$ by $\lie{u}(n-1)$. As
above we look for subalgebras $\lie{l}$ of $\lie{u}(n)$ that verify
$\lie{l}+\lie{u}(n-1)=\lie{u}(n)$. According to~\cite[Theorem~1.5.1]{On} this
holds if and only if $\lie{l}'+\lie{su}(n-1)=\lie{su}(n)$ (where $\lie{l}'$ is
the derived algebra of $\lie{l}$) and $\lie{z}\bigl(\lie{u}(n)\bigr)=\pi\bigl(
\lie{z}(\lie{l})+\lie{z}\bigl(\lie{u}(n-1)\bigr)\bigr)$ where $\pi$ is the 
projection of $\lie{u}(n)$ onto its center with kernel $\lie{su}(n)$. Note that
the second condition is automatically verified since already $\pi\bigl(
\lie{z}(\lie{u}(n-1))\bigr)=\lie{z}\bigl(\lie{u}(n)\bigr)$. Consequently, the
result follows from the classification of subgroups of $\SU(n)$ that act
transitively on $S^{2n-1}\subset\mbb{C}^n$ given in~\cite[Table~8]{On}.
\end{proof}

\begin{rem}\label{sp(1)+l2}
The subalgebras $\lisp(1)\oplus\lie{l}_2$ and $\lie{l}_2\oplus\lisp(1)$ of 
$\so(4)$ are $\OO(4)$-conjugate (see \cite[Remark 8.5]{K}) 
but not $\SO(4)$-conjugate to each other, 
since any inner automorphism of $\SO(4)$ leaves the ideals 
$\lisp(1)\oplus\{0\}$ and $\{0\}\oplus\lisp(1)$ in $\so(4)$ invariant. 
\end{rem}

Since for classical $G$ the $K$-action on $\lie{p}$ and the $M$-action on 
$\lie{g}_\alpha$ are 
essentially the same, we can use the results of the non-reductive case also in 
the reductive one. Moreover, suppose that $\lie{h}=\lie{k}_H\oplus 
\lie{p}_H$ is a reductive spherical subalgebra of $\lie{g}$. Then we have 
\begin{equation}\label{Eqn:Constraint}
[\lie{p}_H,\lie{p}_H]\subset\lie{k}_H\subset\mathcal{N}_\lie{k}(\lie{p}_H).
\end{equation}
Using the normal forms obtained in the first step, it is not hard to 
single out those $\lie{k}_H$ and $\lie{p}_H$ from the lists obtained in the 
non-reductive case that verify~\eqref{Eqn:Constraint}.

\section{$G=\SO^\circ(n,1)$}\label{Section:son}

\subsection{Non-reductive spherical subalgebras}\label{G=SO(n)notreductive}

In this subsection we classify the spherical non-reductive algebraic subalgebras
$\lie{h}$ of $\lie{g}=\lie{so}(n,1)$, where $n\geq 2$. As we have seen we 
may assume $H=M_HA_HN_H$, where $1\leq \dim\lie{n}_H^\perp< n-1$ and $M_H\subset 
\mathcal{N}_M(\mathfrak{n}_H)$ (see Proposition~\ref{Prop:nonredspherical}).
The $M$-action on $\mathfrak{n}$ is isomorphic to the defining representation of
$\SO(n-1)$ on $\RR^{n-1}$ (see Remark \ref{Rem:KandMaction}). Therefore any real subspace 
$\lie{n}_H$ of $\lie{n}$ can, and will, be identified with a real subspace $W$ 
of $\RR^{n-1}$.

Recall that $\RR^{n-1}$ comes equipped with an $\SO(n-1)$-invariant real inner 
product $s$. Note that such an inner product is unique up to a positive factor. 
Let $W$ be a real subspace of $\RR^{n-1}$. Then 
$s|_{W\times W}$ is a real inner product
on $W$ and $\OO(W)$ denotes the group of invertible endomorphisms of $W$ that respect
$s|_{W\times W}$. We write $\OO(W)\times\OO(W^\perp)$ for the subgroup of 
$\OO(\RR^{n-1})$ that stabilizes the decomposition $\RR^{n-1}=W\oplus W^{\perp_s}$.

\begin{lemma}\label{SOn1notred}
Let $W\subset\RR^{n-1}$ be a real subspace. Then 
$\mathcal{N}_{\SO(n-1)}(W)=S(\OO(W)\times\OO(W^{\perp_s}))$ and its action on 
$W$ and $W^{\perp_s}$ is induced by the standard $\OO(W)$-action on $W$ and 
$\OO(W^{\perp_s})$-action on $W^{\perp_s}$ respectively. In particular
$\mathcal{N}_{\SO(n-1)}(W)$ acts transitively on the connected components of 
the spheres in $W$. Moreover there exists a basis $(w_1,\ldots,w_{n-1})$ of $\RR^{n-1}$
such that $W=\RR w_1\oplus\cdots\oplus\RR w_l$ where $l=\dim_\RR W$.
\end{lemma}
\begin{proof}
Any element in $\mathcal{N}_{\SO(n-1)}(W)$ respects $s$, $W$ and $W^{\perp_s}$ and
has determinant $1$, which implies that it lies in 
$S(\OO(W)\times\OO(W)^{\perp_s})$. On the other hand any element in 
$S(\OO(W)\times\OO(W)^{\perp_s})$ leaves $W$ invariant and lies in 
$\SO(n-1)$. Thus $\mathcal{N}_{\SO(n-1)}(W)=S(\OO(W)\times\OO(W^{\perp_s}))$ 
holds. Choosing oriented $s$-orthonormal bases $(w_1,\ldots,w_l)$ of $W$ and 
$(w_{l+1},\ldots,w_{n-1})$ of $W^{\perp_s}$ we obtain an element
$(w_1\cdots w_{n-1})\in\SO(n-1)$ that maps $W_l$ to $W$.
\end{proof}

Combining this observation with Proposition~\ref{Prop:nonredspherical}, 
Theorem~\ref{Thm2On} and Remark~\ref{sp(1)+l2}, we obtain the following
theorem.

\begin{theorem}\label{NonredOrthogonal}
Every spherical non-reductive algebraic subalgebra of $\g=\so(n,1)$, $n\geq2$,
is $G$-conjugate to exactly one in the following list where 
$\lie{c}_l\subset\lie{so}(n-l-1)$ is an arbitrary subalgebra (under the
condition displayed in italic in Remark~\ref{non-reductive,embedding} below)
and where 
\begin{equation*}
\mathfrak{n}_l:=\left\{\left(\begin{smallmatrix}0&v&-v\\-v^t&0&0
\\-v^t&0&0\end{smallmatrix}\right):v\in\{0\}^l\times\RR^{n-1-l}\right\}
\cong W^{\perp_s}.
\end{equation*}
\setlength\arraycolsep{7pt}
\begingroup
\renewcommand*{\arraystretch}{1.2}
\begin{equation*}
\begin{array}{|c|c|}\hline
\lie{l}_H\oplus\lie{n} & \lie{l}_H\subset\lie{m}\oplus\lie{a}\text{ arbitrary}\\
\so(l)\oplus\lie{c}_l\oplus\lie{a}\oplus\lie{n}_l & 1\leq l\leq n-2,\\
\liu(m)\oplus\lie{c}_l\oplus\lie{a}\oplus\lie{n}_l & 1\leq l\leq n-2, l=2m,
m\geq3\\
\su(m)\oplus\lie{c}_l\oplus\lie{a}\oplus\lie{n}_l & 1\leq l\leq n-2, l=2m,
m\geq3\\
\lisp(m)\oplus\lisp(1)\oplus\lie{c}_l\oplus\lie{a}\oplus\lie{n}_l 
& 1\leq l\leq n-2, l=4m, m\geq2\\
\end{array}
\end{equation*}
\begin{equation*}
\begin{array}{|c|c|}
\lisp(m)\oplus\lie{u}(1)\oplus\lie{c}_l\oplus\lie{a}\oplus\lie{n}_l 
& 1\leq l\leq n-2, l=4m, m\geq1\\
\lisp(m)\oplus\lie{c}_l\oplus\lie{a}\oplus\lie{n}_l 
& 1\leq l\leq n-2, l=4m, m\geq1\\
\so(9)\oplus\lie{c}_{16}\oplus\lie{a}\oplus\lie{n}_{16} & \\
\so(7)\oplus\lie{c}_8\oplus\lie{a}\oplus\lie{n}_8  & \\
\lie{g}_2\oplus\lie{c}_7\oplus\lie{a}\oplus\lie{n}_7 & \\
\hline
\end{array}
\end{equation*}
\endgroup\setlength\arraycolsep{2pt}
Note that $\lie{n}_1$ is $1$-codimensional and that 
$\mathcal{N}_{ \lie{m}}(\lie{n}_1)\cong\so(n-2)$.
\end{theorem}

\begin{rem}\label{non-reductive,embedding}
Let us explain the notation in Theorem~\ref{NonredOrthogonal}. Let $\pi\colon
\lie{so}(l)\oplus\lie{so}(n-l)\to\lie{so}(l)$ denote the projection onto the first
factor. Then $\pi$ is a Lie algebra homomorphism and the kernel of its restriction
to an arbitrary subalgebra $\lie{k}\subset\lie{so}(l)\oplus\lie{so}(n-l)$ is an ideal in
$\lie{k}$. Since $\lie{k}$ is compact the orthogonal complement to this ideal with
respect to the Killing form $\kappa$ is also an ideal in $\lie{k}$ and we obtain
$\lie{k}=(\ker\pi|_\lie{k})^{\perp_\kappa}\oplus\ker\pi|_\lie{k}.$
Note that $(\ker\pi|_\lie{k})^{\perp_\kappa}$ is isomorphic to the subalgebra 
$\pi(\lie{k})$ of $\so(l)$.
Let $\Phi:\pi(\lie{k})\to(\ker\pi|_\lie{k})^{\perp_\kappa}$ be a Lie algebra isomorphism. 
Then $\varphi=\pi_{\so(n-l)}\circ\Phi:\pi(\lie{k})\to\so(n-l)$ is a Lie algebra homomorphism,
where $\pi_{\so(n-l)}:\so(l)\oplus\so(n-l)\to\so(n-l)$ is the projection onto the second factor.
Then 
\begin{align*}
\lie{k}=\bigl\{\bigr(\xi,\varphi(\xi)\bigl):\ \xi\in\pi(\lie{k})\bigr\}\oplus
\ker\pi|_\lie{k}.
\end{align*}
In order to simplify the notation, we write here and in the rest of this work
$\lie{k}=\pi(\lie{k})\oplus\lie{c}$ where 
$\lie{c}=\ker\pi|_{\lie{k}}=\lie{k}\cap\bigl(\{0\}\oplus\lie{so}(n-l)\bigr)$ is
the ineffectivity of the 
$\mathcal{N}_{\lie{m}}(W)$-action on $W$, i.e., we omit the representation 
$\varphi$ from the notation.
\begin{quote}
\emph{In particular, $\pi(\lie{k})\oplus\lie{c}$ is a 
subalgebra of $\lie{so}(n,1)$ if and only if $\lie{c}$ is contained in the
centralizer of $\varphi\bigl(\pi(\lie{k})\bigr)$ in $\lie{so}(n-l)$.}
\end{quote}
Note that, if $\pi(\lie{k})$ is simple, then $\varphi$ is either identically
zero or injective.
\end{rem}

\subsection{Reductive spherical subalgebras}\label{orthogonalreductive}

Let $\lie{h}=\lie{k}_H\oplus\lie{p}_H$ be a reductive algebraic 
subalgebra of $\lie{g}=\lie{so}(n,1)$ (where $n\geq 2$) such that
$\lie{k}_H\subset\lie{k}$ and $\lie{p}_H\subset\p$. On the group level we have 
$H=K_H\exp(\lie{p}_H)$.

Since the $K$-representation on $\p$ is essentially the same as the 
$M$-representation on $\lie{n}$ (see Remark \ref{Rem:KandMaction}), we may apply 
Lemma \ref{SOn1notred} to this situation and obtain, after conjugation in $K$, 
that $W=\RR^l\times\{0\}^{n-l}$, where $l=\dim_\RR(W)$. In particular
\begin{equation*}
\lie{p}_H=\lie{p}_{H,l}:=\left\{\left(\begin{smallmatrix}0&x\\x^t&0
\end{smallmatrix}\right): x\in\{0\}^l\times\RR^{n-l}\right\},
\end{equation*}
for some $0\leq l\leq n$. A direct calculation shows that
\begin{align*}
[\lie{p}_{H,l},\lie{p}_{H,l}]&=\left\{\left(\begin{smallmatrix}0&0&0\\0&B&\vdots 
\\
0&\ldots&0\end{smallmatrix}\right)\in\mathfrak{k}: B\in \so(n-l)\right\},\\
\mathcal{N}_\mathfrak{k}(\lie{p}_{H,l})&=\left\{\left(\begin{smallmatrix}
A&0&0\\0&B&0\\0&0&0
\end{smallmatrix}\right)\in\mathfrak{k}: A\in \so(l), B\in \so(n-l)\right\},
\end{align*}
and thus $[\lie{p}_{H,l},\lie{p}_{H,l}]\oplus\lie{p}_{H,l}\cong\lie{so}(n-l,1)$ 
(with $\so(0,1):= \{0\}$). Therefore the condition 
$[\lie{p}_{H},\lie{p}_{H}]\subset\lie{k}_H\subset\mathcal{N}_{\lie{k}}(\lie{
p}_H)$ (see \eqref{Eqn:Constraint}) implies that $\lie{h}$ is of the form
$\lie{h}=\lie{b}\oplus\so(n-l,1)$, where $\lie{b}$ is a subalgebra of $\so(l)$
and $0\leq l\leq n$. Together with Proposition \ref{Prop:redspherical} and
Theorem \ref{Thm2On} we obtain the following theorem.

\begin{theorem}\label{ReductiveOrthogonal}
All spherical reductive algebraic subalgebras $\mathfrak{h}$ of $\so(n,1)$,
$n\geq2$, are (up to conjugation in $G$) one of the following, where
$\so(0,1):=\{0\}$ and $\lie{l}_2\subset\lisp(1)$ is arbitrary.
\setlength\arraycolsep{7pt}
\begingroup
\renewcommand*{\arraystretch}{1.2}
\begin{equation*}
\begin{array}{|c|c|}\hline
\so(l)\oplus\so(n-l,1) & 0\leq l\leq n\\
\liu(m)\oplus\so(n-l,1) &0\leq l\leq n, l=2m, m\geq3\\
\su(m)\oplus\so(n-l,1) &0\leq l\leq n, l=2m, m\geq3\\
\lisp(m)\oplus\lie{l}_2\oplus\so(n-l,1) & 0\leq l\leq n, l=4m, m\geq1\\
\so(9)\oplus\so(n-16,1) & n\geq16\\
\so(7)\oplus\so(n-8,1) & n\geq8 \\
\lie{g}_2\oplus\so(n-7,1) & n\geq7 \\
\lie{l}_2\oplus\lisp(1)\oplus\so(n-4,1)&n\geq4\\
\hline
\end{array}
\end{equation*}
\endgroup\setlength\arraycolsep{2pt}
According to~\cite[Table 2]{Be}, the symmetric subalgebras are 
$\so(n)\oplus\so(1)$ and $\so(l)\oplus \so(n-l,1)$ for $0\leq l<n$.
\end{theorem}

\begin{rem}
Contrary to the non-reductive case, see Remark~\ref{non-reductive,embedding},
the condition 
$[\lie{p}_H,\lie{p}_H]\subset\lie{k}_H\subset\mathcal{N}_{\lie{k}}(\lie{p}_H)$ 
(see \eqref{Eqn:Constraint}) implies that there is no ambiguity in choosing the 
representation $\varphi$. More precisely, we have $\varphi=0$.
\end{rem}

\section{$G=\SU(n,1)$}\label{Section:sun}

\subsection{Notation}

Let $V$ be a complex vector space of complex dimension $n$ and let $h$ be a 
hermitian inner product on $V$. We denote the unitary and special unitary 
groups of $V$ with respect to $h$ by ${\rm{U}}(V)$ and ${\rm{SU}}(V)$,
respectively. Similarly, if $W$ is a real subspace of $V$, we write
${\rm{O}}(W)$ (and ${\rm{SO}}(W)$) for the group of $\mbb{R}$-linear
transformations of $W$ that leave the scalar product $s|_{W\times W}$
invariant (and have determinant $1$) where $s:=\re(h)$. If $W\subset V$ is a
totally real subspace, then we extend the action of ${\rm{O}}(W)$ by
$\mbb{C}$-linearity to $W^\mbb{C}=W\oplus iW$. By abuse of notation, we
consider ${\rm{O}}(W)$ also as a group of $\mbb{C}$-linear maps of $W^\mbb{C}$.

\begin{rem}\label{Rem:irred}
If $W$ is a real form of $V$, i.e., a maximal totally real subspace, then the 
complex ${\rm{O}}(W)$-representation on $V$ is irreducible. Moreover, if 
$n\geq3$, then the complex ${\rm{SO}}(W)$-representation on $V$ is irreducible, 
too, see~\cite[Lemma~6.3]{K}.
\end{rem}

Suppose that $V=V_1\oplus V_2$ is an orthogonal direct sum of two complex
subspaces. As in Section~\ref{Section:son} we write 
${\rm{U}}(V_1)\times{\rm{U}}(V_2)$ for the subgroup of ${\rm{U}}(V)$ that 
stabilizes this decomposition. 

\subsection{Preliminaries}

From now on we fix a real subspace $W$ of $V$ such that a compact subgroup $K$
of $\mathcal{N}_{{\rm{U}}(V)}(W)$ acts irreducibly on $W$. Since $K$ acts
complex linearly on $V$, the maximal complex subspace $W\cap iW$ of $W$ is 
$K$-invariant. Consequently, $W$ is either complex or totally real. We first
treat the complex case.

\begin{lemma}
Let $W$ be any complex subspace of $V$. Then $\mathcal{N}_{{\rm{U}}(V)}(W)$
coincides with ${\rm{U}}(W)\times{\rm{U}}(W^\perp)$ and acts transitively on the
spheres in $W$. Moreover, there is an orthonormal basis $(v_1,\dotsc,v_n)$ of
$V$ such that  $W=\mbb{C}v_1\oplus\dotsb\oplus\mbb{C}v_l$ where
$l=\dim_\mbb{C}W$.
\end{lemma}

Let us now suppose that $W$ is totally real in $V$. Our goal is to understand
how $W$ lies in its complexification $W^\mbb{C}=W\oplus iW\subset V$ with
respect to the symplectic form $\omega:=\im(h)$. For this we define a linear map
$\mathcal{I}\colon W\to W$ by the identity
\begin{equation*}
s(\mathcal{I}w_1,w_2)=\omega(w_1,w_2)
\end{equation*}
for all $w_1,w_2\in W$. One verifies directly that $\mathcal{I}$ is
$\mathcal{N}_{{\rm{U}}(V)}(W)$-equivariant and skew-symmetric with respect to
$s$. Moreover, the kernel of $\mathcal{I}$ is $W\cap W^{\perp_\omega}$, i.e.,
the maximal subspace of $W$ on which $\omega$ is degenerate.

Let $\mathcal{I}^\mbb{C}\in\End(W^\mbb{C})$ be the $\mbb{C}$-linear extension
of $\mathcal{I}$. For the proof of the following proposition we refer the
reader to~\cite[Proposition~6.11]{K}.

\begin{lemma}
The map $\mathcal{I}^\mbb{C}\colon W^\mbb{C}\to W^\mbb{C}$ is $\mathcal{N}_{
{\rm{U}}(V)}(W)$-equivariant and skew-hermitian with respect to $h$. Moreover,
$\mathcal{I}^\mbb{C}$ commutes with the complex conjugation of $W^\mbb{C}$
with respect to $W$, denoted in the following by $\sigma$. 
\end{lemma}

As a consequence, $\mathcal{I}^\mbb{C}$ is diagonalizable having purely
imaginary eigenvalues. Let
\begin{equation*}
W^\mbb{C}=\ker(\mathcal{I}^\mbb{C})\oplus\bigoplus_{\xi\in i\mbb{R}^*}
W^\mbb{C}_\xi
\end{equation*}
be the decomposition of $W^\mbb{C}$ into the eigenspaces of
$\mathcal{I}^\mbb{C}$. This decomposition is
$\mathcal{N}_{{\rm{U}}(V)}(W)$-invariant and $h$-orthogonal, and we have
$\sigma(W^\mbb{C}_\xi)=W^\mbb{C}_{-\xi}$. Hence, we obtain the
$\mathcal{N}_{{\rm{U}}(V)}(W)$-invariant decomposition
\begin{equation*}
W=\ker(\mathcal{I})\oplus\bigoplus_{\xi\in i\mbb{R}^{>0}} (W^\mbb{C}_\xi \oplus
W^\mbb{C}_{-\xi})^\sigma.
\end{equation*}
Since by assumption there exists a compact subgroup $K$ of
$\mathcal{N}_{{\rm{U}}(V)}(W)$ that acts irreducibly on $W$, we have either
$W=\ker(\mathcal{I})$ or $W=(W^\mbb{C}_\xi\oplus W^\mbb{C}_{-\xi})^\sigma$ for
some $\xi\in i\mbb{R}^{>0}$. In the first case $W$ is Lagrangian in $W^\mbb{C}$
with respect to $\omega$, while in the second case $W$ is a symplectic subspace
of $W^\mbb{C}$ with respect to $\omega$.

This discussion proves part of the following proposition.

\begin{proposition}\label{Prop:Un1notredmain}
Let $W$ be a totally real subspace of $(V,h)$ and $W^\mbb{C}=W\oplus iW$ be its
complexi\-fication. If a subgroup $K\subset\mathcal{N}_{{\rm{U}}(V)}(W)$ acts
irreducibly on $W$, then
\begin{enumerate}[(i)]
\item either $W\subset W^\mbb{C}$ is Lagrangian with respect to $\omega$,
$W^\mbb{C}$ is a complex irreducible 
$\mathcal{N}_{{\rm{U}}(V)}(W)$-representation and
$\mathcal{N}_{{\rm{U}}(V)}(W)={\rm{O}}(W)\times{\rm{U}}(W^{\perp_h})$,
\item or $W\subset W^\mbb{C}$ is symplectic with respect to $\omega$, 
$W^\mbb{C}$ decomposes $h$-orthogonally as $W^\mbb{C}=W^\mbb{C}_{\xi}\oplus
W^\mbb{C}_{-\xi}$ into two complex irreducible
$\mathcal{N}_{{\rm{U}}(V)}(W)$-representations that are as real representations
isomorphic to $W$ and $\mathcal{N}_{{\rm{U}}(V)}(W)=
\{(g,\varphi(g))\in{\rm{U}}(W^\mbb{C}_{\xi})\times{\rm{U}}(W^\mbb{C}_{-\xi})\}
\times{\rm{U}}(W^{\perp_h})$, where
$\varphi\colon{\rm{U}}(W^\mbb{C}_{\xi})\to{\rm{U}}(W^\mbb{C}_{-\xi})$ is a Lie
group isomorphism.
\end{enumerate}
In both cases $\mathcal{N}_{{\rm{U}}(V)}(W)$ acts transitively on the spheres in
$W$.
\end{proposition}

\begin{proof}
If $W\subset V$ is Lagrangian, then $W$ and $iW$ are orthogonal with respect to
$s$. Any element in $\mathcal{N}_{\U(V)}(W)$ leaves $W$ and
$W^{\perp_h}=(W^\mbb{C})^{\perp_h}$ as well as $s|_{W\times W}$ invariant, and 
lies therefore in $\OO(W)\times\U(W^{\perp_h})$. Conversely, any element in 
$\OO(W)$ acts $\mbb{C}$-linearly on $W^\mbb{C}$. Again, due to the fact that $W$ 
is Lagrangian in $W^\mbb{C}$, the hermitian inner product $h|_{W^\mbb{C}\times
W^\mbb{C}}$ is $\OO(W)$-invariant, which implies $\OO(W)\times\U(W^{\perp_h})
\subset\mathcal{N}_{\U(V)}(W)$. This shows $\mathcal{N}_{\U(V)}(W)=\OO(W)\times
\U(W^{\perp_q})$. Consequently, Remark~\ref{Rem:irred} implies that
$\mathcal{N}_{\U(V)}(W)$ acts transitively on the spheres in $W$ and complex
irreducibly on $W^\mbb{C}$.

Let us now consider the case that $W^\mbb{C}=W^\mbb{C}_{\xi}\oplus
W^\mbb{C}_{-\xi}$ is the decomposition of $W^\mbb{C}$ into
$\mathcal{I}^\CC$-eigenspaces corresponding to the eigenvalues $\pm\xi\in
i\RR\backslash\{0\}$. Recall that this decomposition is $h$-orthogonal and that
the $\CC$-anti-linear involution $\sigma$ on $W^\mbb{C}$ yields an $\RR$-linear
isomorphism between $W^\mbb{C}_{\xi}$ and $W^\mbb{C}_{-\xi}$. Hence, we obtain
$\RR$-linear, $\mathcal{N}_{\U(V)}(W)$-equivariant isomorphisms $f_{\pm}$
between $W^\mbb{C}_{\pm \xi}$ and $W$ given by $f_{\pm}(w)=\frac{1}{2} 
\bigl(w+\sigma(w)\bigr)$. In particular, the complex subspaces 
$W^\mbb{C}_{\pm\xi}$ are real irreducible, hence complex irreducible,
$\mathcal{N}_{\U(V)}(W)$-representations (isomorphic to $W$). We define another
hermitian inner product on $W^\mbb{C}_{-\xi}$ by 
$h_\sigma(v,u)=\ol{h(\sigma(v),\sigma(u))}$ for all $v,u\in W^\mbb{C}_{-\xi}$.
One verifies directly that $h_\sigma$ and $h|_{W^\mbb{C}_{-\xi}\times
W^\mbb{C}_{-\xi}}$ are both $\mathcal{N}_{\U(V)}(W)$-invariant. Thus there 
exists an $r\in\RR^{>0}$ such that $h_\sigma=r^2h|_{W^\mbb{C}_{-\xi}\times 
W^\mbb{C}_{-\xi}}$.

Suppose for a moment that $r=1$, i.e., that $\sigma$ is an isometry of $s$. 
Consequently, $W$ is orthogonal to $iW$ with respect to $s$, hence a Lagrangian 
subspace of $W^\mbb{C}$ with respect to $\omega$,
contradicting our assumption. Thus we conclude $r\not=1$.

We define the map $\varphi\colon\U(W^\mbb{C}_{\xi})\to\GL(W^\mbb{C}_{-\xi})$ by 
$\varphi(g)v=\sigma(g\cdot\sigma^{-1}(v))$ for all $v\in W^\mbb{C}_{-\xi}$. A
direct calculation shows that $\varphi$ is a Lie group isomorphism from
$\U(W^\mbb{C}_{\xi})$ to $\U(W^\mbb{C}_{-\xi})$. This observation allows us to
show that
\begin{align*}
\mathcal{N}_{\U(V)}(W)
=\{(g,\varphi(g))\in\U(W^\mbb{C}_{\xi})\times\U(W^\mbb{C}_{-\xi})\}\times\U(W^{
\perp_h} ).
\end{align*}
Indeed, any element in $\mathcal{N}_{\U(V)}(W)$ respects the decomposition
$W^\mbb{C}_\xi\oplus W^\mbb{C}_{-\xi}$ and therefore lies in
$\U(W^\mbb{C}_\xi)\times\U(W^\mbb{C}_{-\xi})$. Using the fact that any
element in $\mathcal{N}_{\U(V)}(W)$
stabilizes $W=(W^\mbb{C})^\sigma$ we obtain
$(k_1,k_2)\cdot(v+\sigma(v))=k_1v+k_2\sigma(v)=k_1v+\sigma(\varphi(k_2)v)$ 
has to lie in $W$ for all 
$(k_1,k_2)\in\U(W^\mbb{C}_\xi)\times\U(W^\mbb{C}_{-\xi})$ and for all $v\in 
W^\mbb{C}_\xi$. This implies $k_1v=\varphi(k_2)v$ for all $v\in W^\mbb{C}_\xi$ 
forcing $k_2=\varphi(k_1)$ and thus
\begin{align*}
\mathcal{N}_{\U(V)}(W)
\subset\{(g,\varphi(g))\in\U(W^\mbb{C}_{\xi})\times\U(W^\mbb{C}_{-\xi})\}
\times\U(W^{ \perp_h}).
\end{align*}
The converse inclusion is elementary to check.
\end{proof}

Introducing coordinates we obtain the following normal form of real subspaces 
$W$ of $\mbb{C}^n$ equipped with the standard hermitian inner product, up to the
action of ${\rm{SU}}(n)$.

\begin{corollary}\label{Cor:unitnormform}
Let $W$ be a real subspace of $\mbb{C}^n$. If some compact subgroup of
$\mathcal{N}_{{\rm{U}}(n)}(W)$ acts irreducibly on $W$, then $W$ lies in the
same ${\rm{SU}}(n)$-orbit as one of the following 
\begin{align*}
W_{\CC,l}&:=\CC^l\times\{0\}^{n-1-l},\quad\quad &0\leq l\leq n-1&\quad&(1)\\
W_{\RR,l}&:=\RR^l\times\{0\}^{n-1-l},\quad\quad &0\leq l\leq n-1&\quad&(2)\\
W_{\RR,2l,r}&:=\left\{\left(\begin{smallmatrix}z\\r\ol{z}\\0\end{smallmatrix}
\right)
:z\in\CC^l\right\},\ \ &0\leq l\leq\lfloor\tfrac{n-1}{2}\rfloor&\quad&(3)
\end{align*}
for some $r\in\RR^{>0}\backslash\{1\}$. The corresponding normalizers 
(with respect to the action given in Remark \ref{Rem:KandMaction})
are given by 
\begin{align*} 
\mathcal{N}_{S(\U(n)\times\U(1))}(W_{\CC,l})&=S(\U(l)\times\U(n-1-l)\times\U(1)),\\
\mathcal{N}_{S(\U(n)\times\U(1))}(W_{\RR,l})&\cong 
S(\OO(l)\times\U(n-1-l)\times\U(1)),\\
\mathcal{N}_{S(\U(n)\times\U(1))}(W_{\RR,2l,r})&\cong 
S(\U(l)\times\U(n-1-2l)\times \U(1)).
\end{align*}
In all three cases $\mathcal{N}_{S(\U(n)\times\U(1))}(W)$ acts transitively on the
connected components of the spheres in $W$.
\end{corollary}

\begin{proof}
If $W$ is Lagrangian in $W^\mbb{C}$ with respect to $\omega$, then $W$ is 
orthogonal to $iW$ with respect to $s$. Therefore any $s$-orthonormal basis
over $\RR$ of $W$ is an $h$-orthonormal basis over $\CC$ of $W^\mbb{C}$. In
particular, if we choose an oriented $s$-orthonormal basis $(w_1,\ldots,w_l)$ of
$W$ and complete it to an $h$-orthonormal basis $(w_1,\ldots,w_{n})$ of
$\CC^{n}$ we obtain an element $(w_1\cdots w_{n})\in\SU(n)$ that maps
$W_{\RR,l}$ to $W$.

If $W$ is symplectic in $W^\mbb{C}$ with respect to $\omega$, then we saw in 
the proof of Proposition~\ref{Prop:Un1notredmain} that there exists an 
$r\in\RR^{>0}\backslash\{1\}$ such that 
$\ol{h(\sigma(v),\sigma(u))}= r^2h(v,u)$ holds for all $v,u\in V_{-\xi}$. 
If we choose an $h$-orthonormal basis $(v_1,\ldots,v_{l/2})$ over
$\CC$ of $V_{\xi}$ and set 
$v_{l/2+j}:=\tfrac{1}{r}\sigma(v_j)$ for all $1\leq j\leq l/2$ we obtain the 
$h$-orthonormal basis $(v_{l/2+1},\ldots,v_{l})$ of $V_{-\xi}$. Completing it
to an oriented $h$-orthonormal basis $(v_1,\ldots,v_{n})$ of $\CC^{n}$ the
element $(v_1\cdots v_{n})\in\SU(n)$ maps
$W_{\RR,l,r}:=\left\{\left(\begin{smallmatrix}z\\r\ol{z}\\0
\end{smallmatrix}\right):z\in\CC^{l/2}\right\}$ to $W$.
Therefore $W$ lies in the same $\SU(n)$-orbit as 
$W_{\RR,l,r}$ for some $r\in\RR^{>0}\backslash\{1\}$ in this case. 
Moreover $r$ can be chosen between $0$ and $1$, because
the element $\left(\begin{smallmatrix}0&\id&0\\\id&0&0\\0&0&\id\end{smallmatrix}
\right)\in\U(n)$ maps $W_{\RR,l,r}$ to $W_{\RR,l,1/r}$.
\end{proof}

\begin{rem}
In~\cite[pp.~29--31]{K} the following geometric strengthening of
Corollary~\ref{Cor:unitnormform} is proven. Let us consider the set
\begin{equation*}
\mathcal{M}_m:=\bigl\{ W\subset\mbb{C}^n;\ \dim_\mbb{R}W=m,\ 
\mathcal{N}_{{\rm{U}}(n)}(W)\text{ acts irreducibly on }W\bigr\}.
\end{equation*}
If $m$ is odd, then $\mathcal{M}_m={\rm{U}}(n)\cdot W_{\mbb{R},m}$. If $m$ is
even with $\lfloor\frac{n}{2}\rfloor\leq \frac{m}{2}\leq n$, then 
$\mathcal{M}_m={\rm{U}}(n)\cdot W_{\mbb{C},m/2}$. If $m$ is even with 
$0\leq\frac{m}{2}\leq\lfloor\frac{n}{2}\rfloor$, then the set 
$\mathcal{S}:=\bigl\{ W_{\mbb{R},m,r};\ r\in[0,1]\bigr\}$ is a geometric slice 
for the ${\rm{U}}(n)$-action on $\mathcal{M}_m$ in the following sense. The 
closed subset $\mathcal{S}$ is a real submanifold with boundary which meets each 
${\rm{U}}(n)$-orbit in $\mathcal{M}_m$ exactly once, and $\bigl\{ 
W_{\mbb{R},m,r};\ r\in(0,1)\bigr\}$ intersects each ${\rm{U}}(n)$-orbit 
transversally.
\end{rem}

\subsection{Non-reductive spherical subalgebras}

In this subsection we describe all non-reduc\-tive spherical algebraic
subalgebras of $\lie{su}(n,1)$, up to the action of $M$. Their unipotent
radicals are of the form $W^\perp\oplus\lie{g}_{2\alpha}$ where $W\subset
\lie{g}_\alpha$ is one of the subspaces described in
Corollary~\ref{Cor:unitnormform}. In order to find all possibilities for their
maximal compact subalgebras we apply Onishchik's theorem to their normalizers 
and thus obtain the following. 

\begin{theorem}\label{nonredUnitary}
Every spherical non-reductive algebraic subalgebra of $\g=\su(n,1)$, $n\geq2$,
is $G$-conjugate to one in the following list where
$\lie{b}_j\subset\liu(n-1-j)$ and $\lie{c}\subset\liu(1)$
are arbitrary (under the condition displayed in italic in
Remark~\ref{non-reductive,embedding} adapted to this situation).
\setlength\arraycolsep{7pt}
\begingroup
\renewcommand*{\arraystretch}{1.2}
\begin{equation*}
\begin{array}{|c|c|}\hline
\lie{l}_H\oplus\lie{n} &\lie{l}_H\subset\lie{m}\oplus\lie{a}\text{ arbitrary}\\
\hline
\lie{s}(\liu(l)\oplus\lie{b}_l\oplus\lie{c})\oplus\lie{a}\oplus\lie{n}_{\CC,l} 
&1\leq l\leq n-1,\\
\lie{s}(\su(l)\oplus\lie{b}_l\oplus\lie{c})\oplus\lie{a}\oplus\lie{n}_{\CC,l} 
&2\leq l\leq n-1\\
\lie{s}(\lisp(m)\oplus\lie{b}_l\oplus\lie{c})\oplus\lie{a}\oplus\lie{n}_{\CC,l} 
&1\leq l\leq n-1, l=2m, m\geq2\\
\lie{s}(\lisp(m)\oplus\lie{u}(1)\oplus\lie{b}_l\oplus\lie{c})\oplus\lie{a}\oplus
\lie{n}_{\CC,l} & 1\leq l\leq n-1, l=2m, m\geq2\\
\lie{s}(\lie{c}\oplus\lie{b}_1\oplus\liu(1))\oplus\lie{a}\oplus\lie{n}_{\CC,1} &
\\
\hline
\lie{s}(\so(l)\oplus\lie{b}_l\oplus\lie{c})\oplus\lie{a}\oplus\lie{n}_{\RR,l} 
&1\leq l\leq n-1,\\
\lie{s}(\liu(m)\oplus\lie{b}_l\oplus\lie{c})\oplus\lie{a}\oplus\lie{n}_{\RR,l} 
&1\leq l\leq n-1, l=2m, m\geq3\\
\lie{s}(\su(m)\oplus\lie{b}_l\oplus\lie{c})\oplus\lie{a}\oplus\lie{n}_{\RR,l} 
& 1\leq l\leq n-1, l=2m, m\geq3\\
\lie{s}(\lisp(m)\oplus\lisp(1)\oplus\lie{b}_l\oplus\lie{c})\oplus\lie{a}\oplus
\lie{n}_{\RR,l} &1\leq l\leq n-1, l=4m, m\geq2\\
\lie{s}(\lisp(m)\oplus\mathfrak{u}(1)\oplus\lie{b}_l\oplus\lie{c})\oplus\lie{a}
\oplus
\lie{n}_{\RR,l} &1\leq l\leq n-1, l=4m, m\geq1\\
\lie{s}(\lisp(m)\oplus\lie{b}_l\oplus\lie{c})\oplus\lie{a}\oplus\lie{n}_{\RR,l} 
& 1\leq l\leq n-1, l=4m, m\geq1\\
\lie{s}(\so(9)\oplus\lie{b}_{16}\oplus\lie{c})\oplus\lie{a}\oplus\lie{n}_{\RR,16
} & \\
\lie{s}(\so(7)\oplus\lie{b}_8\oplus\lie{c})\oplus\lie{a}\oplus\lie{n}_{\RR,8} 
&\\
\lie{s}(\lie{g}_2\oplus\lie{b}_7\oplus\lie{c})\oplus\lie{a}\oplus\lie{n}_{\RR,7}
&\\
\hline
\lie{s}(\liu(l)\oplus\lie{b}_{2l}\oplus\lie{c})\oplus\lie{a}\oplus\lie{n}_{\RR,
2l,r} 
&1\leq 2l\leq n-1, 0<r\neq 1\\
\lie{s}(\su(l)\oplus\lie{b}_{2l}\oplus\lie{c})\oplus\lie{a}\oplus\lie{n}_{\RR,2l
,r} 
&2\leq 2l\leq n-1, 0<r\neq 1\\
\lie{s}(\lisp(m)\oplus\mathfrak{u}(1)\oplus\lie{b}_{2l}\oplus\lie{c})\oplus
\lie{a}\oplus\lie{n}_{\RR,2l,r} &1\leq 2l\leq n-1, l=2m, m\geq2,0<r\neq 1\\
\lie{s}(\lisp(m)\oplus\lie{b}_{2l}\oplus\lie{c})\oplus\lie{a}\oplus\lie{n}_{\RR,
2l,r} 
&2\leq 2l\leq n-1, l=2m, m\geq2, 0<r\neq 1\\
\hline
\end{array}
\end{equation*}
\endgroup\setlength\arraycolsep{2pt}
Here we have written $\lie{n}_{\mbb{C},l}$ to denote $W_{\mbb{C},l}^\perp\oplus 
\lie{g}_{2\alpha}$ and so on. Note that $\lie{n}_{\RR,1}$ is $1$-codimensional 
and that 
$\mathcal{N}_{\lie{m}}(\lie{n}_{\RR,1})=\lie{s}(\liu(n-2)\oplus\liu(1))$.
\end{theorem}

\subsection{Reductive spherical subalgebras}

In this subsection we classify the reductive spherical subalgebras $\lie{h}=
\lie{k}_H\oplus\lie{p}_H$ of $\lie{su}(n,1)$. Since the $K$-action on $\lie{p}$ 
is essentially the same as the $M$-action on $\lie{g}_\alpha$, we may apply 
again Corollary~\ref{Cor:unitnormform} in order to obtain all the candidates for 
the subspace $\lie{p}_H\subset\lie{p}$. In the reductive case we have the 
additional restriction that 
$[\lie{p}_H,\lie{p}_H]\subset\lie{k}_H\subset\mathcal{N}_{\lie{k}}
(\lie{p}_H)$ must hold.

To be precise, let $\lie{p}_H^\perp= W\subset\CC^n$ be a real subspace. If 
$\mathcal{N}_K(W)$ acts irreducibly on $W$ then $\lie{p}_H$ fulfills 
$[\lie{p}_H,\lie{p}_H]\subset\mathcal{N}_{\lie{k}}(\lie{p}_H)$ only if $W$ lies 
in the same $\SU(n)$-orbit as $W_{\CC,l}$ (in which case 
$[\lie{p}_H,\lie{p}_H]\oplus\lie{p}_H\cong\su(n-l,1)$) for some $0\leq l\leq n$ 
or $W_{\RR,n}$ (in which case 
$[\lie{p}_H,\lie{p}_H]\oplus\lie{p}_H\cong\so(n,1)$), see~\cite[Lemma~6.21]{K}.

Consequently, we obtain the following list of reductive spherical subalgebras
of $\lie{su}(n,1)$.

\begin{theorem}\label{ReductiveUnitary}
All spherical, reductive subalgebras $\mathfrak{h}$ of $\su(n,1)$,
where $n>1$ are (up to
conjugation
in $G$) 
one of the following
\setlength\arraycolsep{7pt}
\begingroup
\renewcommand*{\arraystretch}{1.2}
\begin{equation*}
\begin{array}{|c|c|}\hline
\liu(n)&\\
\su(n)&n>1\\
\lisp(m)&n=2m\text{ and }m\geq 2\\
\lisp(m)\oplus\lie{u}(1)&n=2m\text{ and }m\geq 2\\
\mathfrak{s}(\liu(l)\oplus\liu(n-l,1))&  0\leq l< n\\
\su(l)\oplus\su(n-l,1)&0< l< n\\
\mathfrak{s}(\lisp(m)\oplus\liu(n-l,1))&0\leq l< n, l=2m, m\geq 2\\
\mathfrak{s}(\lisp(m)\oplus \mathfrak{u}(1)\oplus\liu(n-l,1))&0\leq l< n,
l=2m,m\geq 2\\
\so(n,1)&\\
\hline
\end{array}
\end{equation*}
\endgroup\setlength\arraycolsep{2pt}
According to \cite[Table 2]{Be}, the symmetric subalgebras are
$\lie{s}(\liu(l)\oplus\liu(n-l,1))$ for $0\leq l<n$,
$\lie{s}(\liu(n)\oplus\liu(1))$ and $\so(n,1)$. 
\end{theorem}

\begin{proof}
Due to~\cite[Lemma~6.21]{K} we may assume that $\lie{p}_H\cong W_{\CC,l}^\perp$ 
for some $0\leq l\leq n$ or $\lie{p}_H\cong W_{\RR,n}^\perp$. Note that 
$W_{\CC,n}^\perp=\{0\}$. Therefore $\lie{h}$ is compact in that case.
If $\lie{p}_H\cong W_{\CC,l}^\perp$ for some $0\leq l< n$ then 
$[\lie{p}_H,\lie{p}_H]\subset\lie{k}_H\subset\mathcal{N}_{\lie{k}}(\lie{p}_H)$ 
together with $[\lie{p}_H,\lie{p}_H]\oplus\lie{p}_H\cong\su(n-l,1)$
implies that $\lie{h}=\lie{s}(\lie{b}\oplus\liu(n-l,1))$ where $\lie{b}$ is some
subalgebra of $\liu(l)$ that acts transitively on the spheres in $\CC^l$. 
If $\lie{p}_H\cong W_{\RR,n}^\perp$ then 
$\so(n)=[\lie{p}_H,\lie{p}_H]\subset\lie{k}_H\subset\mathcal{N}_{\lie{k}}
(\lie{p}_H)=\so(n)$ implies that $\lie{h}=\so(n,1)$.
\end{proof}

\section{$G=\Sp(n,1)$}\label{symplectic}

\subsection{Notation}
Let $U$ be a quaternionic vector space (with scalar multiplication from the right) and 
let $q:U\times U\to\HH$ be a quaternionic inner product on $U$ (i.e., $q$ is conjugate 
symmetric, $\HH$-linear in the second argument and positive definite).
We denote the group of invertible, $\HH$-linear transformations on $U$ by 
$\GL_\HH(U)$ and define 
\begin{align*}
\Sp(U):=\Sp(U,q)=\{A\in\GL_\HH(U):\ \forall\ v,w\in U, q(Av,Aw)=q(v,w)\}.
\end{align*}
For $\HH^n$ equipped with the standard quaternionic inner product 
(given by $q(v,w)=\ol{v}^tw$ for all $v,w\in\HH^n$) we set
$\Sp(n):=\Sp(\HH^n)$.
The standard quaternionic inner product on $\HH$ is denoted by $q_\HH$.

Let $\lambda\in\Sp(1)\cap\im(\HH)$. Then $\lambda^2=-1$ and 
$F_\lambda:=\RR\oplus\RR\lambda$ is a subfield of $\HH$ isomorphic to
$\CC$. If $\mu\in\Sp(1)\cap\im(\HH)$ has the property 
$\re(q_\HH(\lambda,\mu))=0$, then $\lambda\mu$ lies in 
$\Sp(1)\cap\im(\HH)$ and is $\re q_\HH$-orthogonal to $\lambda$ and 
$\mu$.

For the proof of the following result we refer the reader
to~\cite[Lemma~7.6]{K}.

\begin{lemma}\label{quaternionictohermitesch}
Let $(1,\lambda,\mu,\lambda\mu)$ be a $\re(q_\HH)$-orthonormal basis over 
$\RR$ of $\HH$ and let $U$ be a quaternionic vector space with quaternionic 
inner product $q$. The projection $h_\lambda:=\pi_{F_\lambda}
\circ q$ of $q$ onto $F_\lambda$ is an Hermitian form on $U$ and 
$q=h_\lambda+\mu\eta_\lambda$, where $\eta_\lambda:U\times U\to F_\lambda$ is a 
symplectic form on $U$. We define the 2-forms $s,\omega_\lambda,\omega_\mu,
\omega_{\mu\lambda}:U\times U\to\RR$ by $s:=\re h_\lambda$, 
$\omega_\lambda:=\re (-\lambda h_\lambda)$, $\omega_\mu:=\re \eta_\lambda$ and 
$\omega_{\mu\lambda}:=\re (-\lambda \eta_\lambda)$ and obtain that all of them are 
$\Sp(U,q)$-invariant and fulfill :
\begin{align*}
q=s+\lambda\omega_\lambda+\mu(\omega_\mu+\lambda\omega_{\mu\lambda})
&:U\times U\to\HH\text{ is a quaternionic inner product},\\
h_\lambda=s+\lambda\omega_\lambda&:U\times U\to F_\lambda\text{ is Hermitian},\\
\eta_\lambda=\omega_\mu+\lambda\omega_{\mu\lambda}
&: U\times U\to F_\lambda\text{ is symplectic},\\
\omega_\lambda,\omega_\mu,\omega_{\mu\lambda}
&:U\times U\to\RR\text{ is symplectic},\\
s&:U\times U\to\RR\text{ is a real inner product}.
\end{align*}
Since we choose $q$ to be $\HH$-linear in the second and $\HH$-anti-linear in the 
first entry we have
\begin{align}
\begin{split}\label{connectionbetweenquaternionicforms}
&s(v\alpha,w\beta)+\lambda\omega_\lambda(v\alpha,w\beta)+\mu\omega_\mu(v
\alpha,w\beta)+\mu\lambda\omega_{\mu\lambda}(v\alpha,w\beta)=q(v\alpha,w\beta)\\
&=\ol{\alpha}q(v,w)\beta=\ol{\alpha}s(v,w)\beta+\ol{\alpha}\lambda\omega_\lambda
(v,w)\beta+\ol{\alpha}\mu\omega_\mu(v,w)\beta+\ol{\alpha}\mu\lambda
\omega_{\mu\lambda}(v,w)\beta
\end{split}
\end{align}
for all $v,w\in U$, $\alpha,\beta\in\HH$. The same is true if one interchanges $\lambda$
and $\mu$.
\end{lemma}

\begin{rem}\label{U(W)irredonV}
If $V$ is an $F_\lambda$-vector space whose quaternionification is $U=V\oplus 
V\mu$, then the quaternionic $\U(V)$-representation on $U$ is irreducible. 
Furthermore, if $W$ is a real form of $U$, i.e., if $U=W\oplus W\lambda\oplus 
W\mu\oplus W\lambda\mu$ holds, then the $\OO(W)$-representation on $U$ is 
irreducible, too, see \cite[Lemma 7.11]{K}.
\end{rem}

Suppose that $U=U_1\oplus U_2$ is a $q$-orthogonal direct sum of two 
quaternionic subspaces. As in Sections \ref{Section:son} and \ref{Section:sun} 
we write $\Sp(U_1)\times\Sp(U_2)$ for the subgroup of $\Sp(U)$ that stabilizes 
this decomposition.

\subsection{Strategy of classification}

We consider the $\Sp(n)\times\Sp(1)$-action on $\HH^n$ given by $(A,a)\cdot 
v=Ava^{-1}$ for all $(A,a)\in\Sp(n)\times\Sp(1)$ and $v\in\HH^n$. Recall that
we want to find all real subspaces $W$ of $\mbb{H}^n$ and all connected
subgroups $K$ of $\mathcal{N}_{\Sp(n)\times\Sp(1)}(W)$ which act transitively
on the connected components of the spheres in $W$.

Given such a $W$, the group $\mathcal{N}_{\Sp(n)\times\Sp(1)}(W)$ acts
necessarily irreducibly on $W$. Since the $\Sp(n)\times\Sp(1)$-action on
$\mbb{H}^n$ maps quaternionic subspaces to quaternionic subspaces, the maximal
quaternionic subspace $W_\HH=W\cap Wi\cap Wj\cap Wk$ of $W$ is
$\mathcal{N}_{\Sp(n)\times\Sp(1)}(W)$-invariant. Consequently, $W$ is either
quaternionic or its maximal quaternionic subspace $W_\HH$ is zero.
Lemma~\ref{Lem:Wquaternionic} gives a normal form for quaternionic $W$ which
allows to complete the classification in this case.

Therefore we are left to deal with subspaces $W$ with $W_\mbb{H}={0}$ on which
$\mathcal{N}_{\Sp(n)\times\Sp(1)}(W)$ acts irreducibly. This case splits into
two subcases: Either there exists a $\lambda\in\Sp(1)\cap\im(\mbb{H})$ such
that the subfield $F_\lambda\cong\mbb{C}$ acts by right multiplication on $W$
or $W$ is totally real in the sense that there exists no $\lambda\in\Sp(1)
\setminus\{\pm1\}$ with $W=W\lambda$. It turns out that the first of these
subcases can be treated by the same arguments as the ones developed in
Section~6, while the totally real case requires most of the work.

In closing we describe how these three cases influence the structure of
$\mathcal{N}_{\Sp(n)\times\Sp(1)}(W)$. The normalizer
$\mathcal{N}_{\Sp(n)\times\Sp(1)}(W)$ is the direct product of 
$\mathcal{N}_{\Sp(n)}(W):=\mathcal{N}_{\Sp(n)\times\{e\}}(W)$ and a normal
subgroup $L$ of $\mathcal{N}_{\Sp(n)\times\Sp(1)}(W)$ that is isomorphic to the
image of $\mathcal{N}_{\Sp(n)\times\Sp(1)}(W)$ under the projection 
$\pi_2\colon\Sp(n)\times\Sp(1)\to\Sp(1)$ given by $(g,h)\mapsto h$, see 
\cite[Lemma 7.13]{K}. Note that $L$ contains $\mathcal{N}_{\Sp(1)}(W):=
\mathcal{N}_{\{e\}\times\Sp(1)}(W)$. Moreover, if $\lie{l}\cong\lie{sp}(1)$,
then $\mathcal{N}_{\Sp(1)}(W)$ is either trivial or coincides with $L$, see
\cite[Lemma 7.13]{K}.

The group $\mathcal{N}_{\Sp(1)}(W)$ is of real dimension $3$ if and only if $W$
is quaternionic. It has real dimension $1$ if and only if $W$ is not
quaternionic but invariant under the subfield $F_\lambda=\RR\oplus\RR\lambda$ of
$\HH$ for some $\lambda\in\Sp(1)\cap\im(\HH)$. The group
$\mathcal{N}_{\Sp(1)}(W)$ is zero-dimensional if and only if $W$ is
totally real, see \cite[Lemma 7.16]{K}.

\subsection{The quaternionic case}

The following lemma gives a normal form for a quaternionic subspace $W\subset
\mbb{H}^n$ and its normalizer. For a proof see~\cite[Lemma~7.17]{K}.

\begin{lemma}\label{Lem:Wquaternionic}
Let $W$ be any quaternionic subspace of $\HH^{n}$. Then 
$\mathcal{N}_{\Sp(n)\times\Sp(1)}(W)$ coincides with
$(\Sp(W)\times\Sp(W^\perp))\times\Sp(1)$ and acts transitively on the 
spheres in $W$. Moreover, there is an orthonormal basis $(v_1,\ldots,v_n)$
of $\HH^n$ such that $W=v_1\cdot\HH\oplus\cdots\oplus v_l\cdot\HH$
where $l=\dim_\HH W$.
\end{lemma}

\subsection{The $F_\lambda$-complex case}

Let us now suppose that $W$ is not quaternionic but that there exists a 
$\lambda\in\Sp(1)\cap\im(\HH)$ such that $W$ is invariant under the
subfield $F_\lambda$. In that case the assumption that 
\begin{align*}
\mathcal{N}_{\Sp(n)\times\Sp(1)}=
\mathcal{N}_{\Sp(n)}(W)\times \underbrace{(\Sp(1)\cap F_\lambda)}_{=L\cong S^1}
\end{align*}
acts transitively on the spheres in $W$ implies that $W$ is an 
$F_\lambda$-irreducible representation of $\mathcal{N}_{\Sp(n)}(W)$, 
see~\cite[Lemma 7.18]{K}.
We complete $(1,\lambda)$ to a $\re q_\HH$-orthonormal basis 
$(1,\lambda,\mu,\lambda\mu)$ over $\RR$ of $\HH$. Then 
$V:=W\oplus W\mu\subset\HH^{n}$ is the quaternionification of 
$W$ on which $\mathcal{N}_{\Sp(n)}(W)\subset\Sp(n)$ acts 
$\HH$-linearly. The idea is to forget the $F_\lambda$-structure on $W$
and $\HH^n$, i.e., to consider $W$ as maximal totally real subspace of 
the complex space $V$ (with respect to $F_\mu$), and then to use the same 
ideas as in Section \ref{Section:sun}.

In particular we consider the Hermitian form $h_\mu$ given by projecting $q$ to 
$F_\mu$, i.e., $h_\mu=s+\mu\omega_\mu$, where $s=\re q$ is a real
inner product on $\HH^n$ and $\omega_\mu=\re(-\mu\cdot q)$ is a
symplectic form on $\HH^n$ (see Lemma \ref{quaternionictohermitesch}). 
We then define the $\RR$-linear map $\J :W\to W$ by the identity
\begin{align}\label{defmathcalJ}
s(\J  w_1,w_2):=\omega_\mu(w_1,w_2)\quad \forall\ w_1,w_2\in W,
\end{align}
and extend it $F_\mu$-linearly to $V=W\oplus W\mu$.

Since $\mathcal{N}_{\Sp(n)}(W)$ is a subgroup of 
$\mathcal{N}_{\U(F_\mu^{n},h_\mu)}(W)$ we obtain as in Section 
\ref{Section:sun} that $V$ decomposes as
\begin{align*}
V=\ker\J\oplus\bigoplus_{\xi\in\mu\RR^{*}}V_\xi
\end{align*}
into $\mathcal{N}_{\Sp(n)}(W)$-invariant and $h_\mu$-orthogonal
eigenspaces of $\J$ to eigenvalues in $\mu\RR$ such that  
$\sigma_\mu(V_\xi)=V_{-\xi}$ holds where $\sigma_\mu$ denotes complex 
conjugation of $V$ with respect to $W$. Furthermore, the eigenspaces are 
quaternionic subspaces that are $q$-orthogonal (see \cite[Lemma 7.25]{K}).
Thus we obtain the $\mathcal{N}_{\Sp(n)}(W)$-invariant decomposition
\begin{align*}
W&=\ker\J\oplus\bigoplus_{\xi\in \mu\RR^{>0}}
(V_\xi\oplus \sigma_\mu(V_\xi))^{\sigma_\mu}.
\end{align*}
Since the group $\mathcal{N}_{\Sp(n)}(W)$ acts irreducibly on $W$ by 
assumption, we have either $W=\ker(\J)$ or 
$W=(V_\xi\oplus\sigma_\mu(V_\xi))^{\sigma_\mu}$ for some 
$\xi\in\mu\RR^{>0}$. This discussion proves part of the following proposition.

\begin{proposition}\label{Prop:SPn1notredmain}
Let $W$ be an $F_\lambda$-invariant subspace of $\HH^{n}$
with quaternionification $V=W\oplus W\mu$. If $K\subset\mathcal{N}_{\Sp(n)}(W)$
acts irreducibly on $W$, then
\begin{itemize}
\item[(i)] either $W\subset V$ is Lagrangian with respect to $\eta_\lambda$,
$V$ is a quaternionic irreducible $\mathcal{N}_{\Sp(n)}(W)$-representation and
$\mathcal{N}_{\Sp(n)}(W)=\U_{h_\lambda}(W)\times\Sp(W^{\perp_q})$.
\item[(ii)] or $W\subset V$ is symplectic with respect to $\eta_\lambda$,
$V$ decomposes as $V=V_{\xi}\oplus V_{-\xi}$ into two, 
$q$-orthogonal, quaternionic irreducible 
$\mathcal{N}_{\Sp(n)}(W)$-representations which are 
as $F_\lambda$-representations isomorphic to $W$ and
\begin{align*}
\mathcal{N}_{\Sp(n)}(W)
&=\bigl\{(g,\psi(g))\in\Sp(V_{\xi})\times\Sp(V_{-\xi})\bigr\}\times
\Sp(W^{\perp_q}),
\end{align*}
where $\psi:\Sp(V_{\xi})\to\Sp(V_{-\xi})$ is a Lie group isomorphism. 
\end{itemize}
In both cases $\mathcal{N}_{\Sp(n)}(W)$ acts transitively on the spheres in $W$.
\end{proposition}

\begin{proof}
If $W\subset V$ is Lagrangian with respect to $\eta_\lambda$, then $W$ and 
$W\mu$ are orthogonal with respect to $h_\lambda$.
We now want to show that 
$\mathcal{N}_{\Sp(n)}(W)=\U_{h_\lambda}(W)\times\Sp(W^{\perp_q})$
holds, where $
\U_{h_\lambda}(W):=\left\{g\in\GL_{F_\lambda}(W):h_\lambda(gv,gw)
=h_\lambda(v,w)\ \forall\ v,w\in W\right\}.$
Any element in $\mathcal{N}_{\Sp(n)}(W)$ leaves $W$, 
$W^{\perp_q}$ and $h_\lambda|_{W\times W}$ invariant and lies 
therefore in $\U_{h_\lambda}(W)\times\Sp(W^{\perp_q})$. Conversely 
any element in $\U_{h_\lambda}(W)$ acts $\HH$-linearly on $V$. Again, 
due to the fact that $W$ is Lagrangian in $V$, the quaternionic inner product 
$q$ is $\U_{h_\lambda}(W)$-invariant. Hence 
$\U_{h_\lambda}(W)\times\Sp(W^{\perp_q})\subset
\mathcal{N}_{\Sp(n)}(W)$ and therefore 
$\mathcal{N}_{\Sp(n)}(W)=\U_{h_\lambda}(W)
\times\Sp(W^{\perp_q})$ holds. The group 
$\mathcal{N}_{\Sp(n)}(W)=\U_{h_\lambda}(W)\times
\Sp(W^{\perp_q})$ acts transitively on the spheres in $W$ and quaternionic 
irreducibly on $V$ (see Remark \ref{U(W)irredonV}).

Let us now consider the case that $V=V_\xi\oplus V_{-\xi}$ is the 
decomposition of $V$ into $\J$-eigenspaces corresponding to the 
eigenvalues $\pm\xi\in \mu\RR\backslash\{0\}$. Recall that this 
decomposition is $q$-orthogonal and that the $F_\mu$-anti-linear 
involution $\sigma_\mu$ on $V$ yields an isomorphism between 
$V_{\xi}$ and $V_{-\xi}$. Hence, we obtain $F_\lambda$-linear, 
$\mathcal{N}_{\Sp(n)}(W)$-equivariant isomorphisms 
$g_{\pm}$ between $V_{\pm\xi}$ and $W$ given by
$g_{\pm}(v)=\tfrac{1}{2}\bigl(v+\sigma_\mu(v)\bigr)$. In particular, the 
quaternionic spaces  $V_{\pm\xi}$ are $F_\lambda$-irreducible 
$\mathcal{N}_{\Sp(n)}(W)$-representations isomorphic to $W$.
We define another quaternionic inner product $q_{\sigma_\mu}$ on $V_{-\xi}$ by 
\begin{align*}
q_{\sigma_\mu}(v,u)=\sigma_{\HH,\mu}\bigl(q(\sigma_\mu(v),
\sigma_\mu(u))\bigr)\quad\forall\ v,u\in V_{-\xi}=\sigma_\mu(V_\xi).
\end{align*}
One verifies directly that $q_{\sigma_\mu}$ and $q|_{V_{-\xi}\times V_{ -\xi}}$ are both 
$\mathcal{N}_{\Sp(n)}(W)$-invariant. Thus there exists a $t\in\RR^{>0}$ 
such that $q_{\sigma_\mu}=t^2 q|_{V_{-\xi}\times V_{-\xi}}$. 
Suppose for a moment that $t=1$, i.e., that $\sigma_\mu$ is an isometry of 
$h_\lambda$. Consequently, $W$ is orthogonal to $W\mu$ with respect to 
$h_\lambda$, hence a Lagrangian subspace of $V$ with respect to $\eta_\lambda$, 
contradicting our assumption. Thus we conclude $t\neq 1$.

We define the map $\psi:\Sp(V_{\xi})\to\GL(V_{-\xi})$ by $\psi(g)v
=\sigma_\mu\bigl(g\cdot\sigma_\mu^{-1}(v)\bigr)$ for all $v\in V_{-\xi}$. A 
direct calculation shows that $\psi$ is a Lie group isomorphism from 
$\Sp(V_{\xi})$ to $\Sp(V_{-\xi})$. This observation allows us to show that
\begin{align*}
\mathcal{N}_{\Sp(n)}(W)=\bigl\{(g,\psi(g))\in\Sp(V_{\xi})\times\Sp(V_{-\xi}) 
\bigr\} \times\Sp(W^{\perp_q}).
\end{align*}
Indeed any element in $\mathcal{N}_{\Sp(n)}(W)$ acts as element of 
$\Sp(V_{\xi})\times\Sp(V_{-\xi})$ on $V_\xi\times V_{-\xi}$. 
Using the fact that any element
in $\mathcal{N}_{\Sp(n)}(W)$ stabilizes $W=V^{\sigma_\mu}$ we obtain
\begin{align*}
(k_1,k_2)\cdot\bigl(v+\sigma_\mu(v)\bigr)= 
k_1v+k_2\sigma_\mu(v)=k_1v+\sigma_\mu\bigl(\psi(k_2)v\bigr)
\end{align*}
has to lie in $W$ for all $(k_1,k_2)\in\Sp(V_\xi)\times\Sp(V_{-\xi})$ 
and for all $v\in V_\xi$. This implies
$k_1v=\psi(k_2)v$ for all $v\in V_\xi$ forcing $k_2=\psi(k_1)$ and thus
\begin{align*}
\mathcal{N}_{\Sp(n)}(W)\subset\bigl\{(g,\psi(g))\in\Sp(V_{\xi})\times 
\Sp(V_{-\xi})\bigr\} \times\Sp(W^{ \perp_q } ).
\end{align*}
The converse conclusion is elementary to check.
\end{proof}

Introducing coordinates we obtain the following normal form of these real subspaces 
$W$ of $(\HH^n,q)$, up to the action of $\Sp(n)$.

\begin{corollary}\label{SPn1cplxnormalform}
Let $W$ be a real subspace of $\HH^n$, that has no quaternionic subspace but is invariant under $F_\lambda$ for some $\lambda\in\Sp(1)\cap\im(\HH)$. If some compact subgroup of $\mathcal{N}_{\Sp(n)}(W)$ acts irreducibly on $W$, then $W$ lies in the same $\Sp(n)$-orbit as one of the following
\begin{align*}
W_{F_\lambda,l}&:=F_\lambda^l\times\{0\}^{n-l},&0\leq l\leq n\\
W_{F_\lambda,l,t}&:=\left\{\left(\begin{smallmatrix}z+\mu w\\t(z-\mu w)\\0
\end{smallmatrix}\right):z,w\in F_\lambda^{l}\right\}, &0\leq l\leq \lfloor\tfrac{n}{2}\rfloor 
\end{align*}
for some $t\in (0,1)$.
The corresponding normalizers are given by
\begin{align*}
\mathcal{N}_{\Sp(n)\times\Sp(1)}(W_{F_\lambda,l})&:= 
\bigl(\U_{h_\lambda}(l)\times\Sp(n-l)\bigr)\times (\Sp(1)\cap F_\lambda),\\
\mathcal{N}_{\Sp(n)\times\Sp(1)}(W_{F_\lambda,l,t})&:=
\left\{\left(\left(\begin{smallmatrix}A&&\\&\psi(A)&\\&&B\end{smallmatrix}
\right),a\right)\in \Sp(n)\times\Sp(1):\ A\in\Sp(l), B\in\Sp(n-2l), a\in
F_\lambda\right\}\\
&\cong (\Sp(l)\times\Sp(n-2l))\times S^1.
\end{align*}
In both cases $\mathcal{N}_{\Sp(n)}(W)$ acts transitively on the connected components of the spheres in $W$.
\end{corollary}

\begin{proof}
If $W$ is $F_\lambda$-invariant and Lagrangian in $V=W\oplus W\mu$ with respect 
to $\eta_\lambda$, then $W$ is orthogonal to $W\mu$ with respect to $h_\lambda$. 
Therefore any $h_\lambda$-orthonormal basis over $F_\lambda$ of $W$ is a 
$q$-orthonormal basis over $\HH$ of $V$. In particular, if we choose an 
$h_\lambda$-orthonormal basis $(w_1,\ldots,w_l)$ of $W$ and complete
it to a $q$-orthonormal basis $(w_1,\ldots,w_{n})$ of $\HH^{n}$, we obtain an 
element $(w_1\cdots w_{n})\in\Sp(n)$ that maps $W_{F_\lambda,l}$ to $W$.

If $W$ is symplectic with respect to $\eta_\lambda$, then we saw in the proof 
of Proposition~\ref{Prop:SPn1notredmain} that there exists a 
$t\in\RR^{>0}\backslash\{1\}$ such that
$\sigma_{\HH,\mu}\bigl(q(\sigma_\mu(v),\sigma_\mu(u))\bigr)=t^2q(v,u)$ 
for all $v,u\in V_{-\xi}$. If we choose a $q$-orthonormal basis 
$(v_1,\ldots,v_{l/2})$ over $\HH$ of $V_{\xi}$ and set 
$v_{l/2+j}:=\tfrac{1}{t}\sigma_\mu(v_j)$ for all $1\leq j\leq l/2$ we obtain the 
$q$-orthonormal basis $(v_{1},\ldots,v_l)$ of $V=V_{\xi}\oplus V_{-\xi}$. 
Completing it to a $q$-orthonormal basis $(v_1,\ldots,v_{n})$ of $\HH^{n}$ the 
element $(v_1\cdots v_{n})\in\Sp(n)$ maps 
$W_{F_\lambda,l,t}$ to $W$. Note that $t$ may be chosen between $0$ and $1$, because 
the element
$\left(\begin{smallmatrix}0&\id&0\\\id&0&0\\0&0&\id\end{smallmatrix}\right)\in\Sp(n)$
maps $W_{F_\lambda,l,t}$ to $W_{F_\lambda,l,1/t}$.
\end{proof}

\subsection{The totally real case}

Let us now consider the case that $W$ is a totally real subspace of $\HH^n$, 
i.e., that there exists no $\lambda\in\Sp(1)\backslash\{\pm1\}$ such that 
$W=W\lambda$.

Again we want to determine those real subspaces $W\subset(\HH^{n},q)$ such that 
$\mathcal{N}_{\Sp(n)\times\Sp(1)}(W)$ acts transitively on the spheres in $W$.
Recall that $\mathcal{N}_{\Sp(n)\times\Sp(1)}(W)=\mathcal{N}_{\Sp(n)}(W)\times 
L$.

Theorem~\ref{Thm:MontgomerySamelson} implies that if $\dim_\RR(W)>4$ then
already $\mathcal{N}_{\Sp(n)}(W)$ has to act transitively on the spheres in $W$.
We therefore consider first totally real subspaces $W$ of arbitrary dimension
which satisfy the condition that $\mathcal{N}_{\Sp(n)}(W)$ acts irreducibly on
$W$. The only spherical subalgebras that we are possibly missing are those for
wich $\dim_\RR(W)\leq4$. We will consider these cases at the end of this
subsection.

We first give a normal form of totally real $W$ up to the action of $\GL(
\mbb{H}^n)$.

\begin{lemma}\label{lactstrivially}
Let $W$ be a real subspace of $\HH^{n}$ such that $\langle W\rangle_\HH=W\oplus 
Wi\oplus Wj\oplus Wk$. Then $W$ lies in the same $\GL(\HH^{n})$-orbit as 
$W_{\RR,l}:=\RR^l\times\{0\}^{n-l}$, where $l=\dim_\RR(W)$ and $\lie{l}$ acts 
trivially on $W$.
\end{lemma}

\begin{proof}
Let $(v_1,\ldots,v_l)$ be an $s$-orthonormal real basis of $W$. Then
$\langle W\rangle_\HH=\langle v_1,\ldots,v_l\rangle_\HH$. In particular
$(v_1,\ldots,v_l)$ is a quaternionic basis of $\langle W\rangle_\HH$
that we complete to a quaternionic basis $(v_1,\ldots,v_{n})$ of 
$\HH^{n}$. Now $g=(v_1\cdots v_{n})$ is an element in 
$\GL(\HH^{n})$ that maps $W_{\RR,l}$ to $W$.
This implies that
\begin{align*}
\mathcal{N}_{\lisp(n)\oplus\lisp(1)}(W)&\subset\Ad(g)(\mathcal{N}_{\lisp(n)\oplus\lisp(1)}(W_{\RR,l}))\\ 
\mathcal{N}_{\lisp(n)\oplus\lisp(1)}(W_{\RR,l})&=\left\{\left(A,\xi\right)\in\lisp(n)\oplus\lisp(1):Aw-w\xi\in W_0\ 
\forall\ w\in W_0\right\}\\
&=\left\{\left(A,\xi\right)\in\lisp(n)\oplus\lisp(1):A=\left(\begin{smallmatrix}A_1&0\\0& *
\end{smallmatrix}\right), A_1-\xi\id\in\RR^{l\times l}\right\}\\
&=\underbrace{\left\{\left(A,0\right)\in\lisp(n)\oplus\lisp(1):A=\left(\begin{smallmatrix}A_1&0\\
0& *\end{smallmatrix}\right), 
A_1\in\RR^{l\times l}\right\}}_{\mathcal{N}_{\lisp(n)}(W_{\RR,l})}\\
&\quad\oplus\underbrace{\left\{\left(A,\xi\right)\in\lisp(n)\oplus\lisp(1): A=\left(\begin{smallmatrix}\xi\id&0\\
0&0\end{smallmatrix}\right)\right\}}_{\tilde{\lie{l}}}.
\end{align*}
Therefore, the action of $\tilde{\lie{l}}$ on $W_{\RR,l}$ and thus on $W$ is trivial.
\end{proof}

In order to give a normal form of $W$ up to the action of $\Sp(n)$, we will
analyze the position of $W$ in $\mbb{H}^n$ with respect to the various
symplectic forms listed in Lemma~\ref{quaternionictohermitesch} which are
invariant under $\Sp(n)$. We will see that under the condition that
$\mathcal{N}_{\Sp(n)}(W)$ acts irreducibly on $W$, these symplectic forms
contain information about the division algebra
$\End_{\mathcal{N}_{\Sp(n)}(W)}(W)$. Hence, we will now divide such totally
real $W$ into three subcases, corresponding to the three possibilities
$\mbb{R}$, $\mbb{C}$ or $\mbb{H}$ for $\End_{\mathcal{N}_{\Sp(n)}(W)}(W)$.

We define the $\RR$-linear maps $\I_i,\I_j,\I_k:W\to W$ by the identity
\begin{align}\label{defIl}
s(\I_m w_1,w_2)=\omega_m(w_1,w_2)\quad\forall\ w_1,w_2\in W,\ \forall\ m\in\{i,j,k\}.
\end{align}
Let us consider the map 
\begin{align*}
\chi:\HH&\to\End_{\mathcal{N}_{\Sp(n)}(W)}(W)\\
a+bi+cj+dk&\mapsto a\id+b\I_i+c\I_j+d\I_k.
\end{align*}
Note that the kernel of $\chi$ lies in $\im(\HH)$, i.e., that $\chi|_\RR$ is injective and that
\begin{align*}
s(\chi(z)w_1,w_2)
&=s(w_1z,w_2)
\end{align*}
holds for all $z\in\HH$ and $w_1,w_2\in W$. In particular, if $\lambda\in\im(\HH)\cap\Sp(1)$, then
$s(\chi(\lambda)w_1,w_2)=s(w_1\lambda,w_2)=\omega_\lambda(w_1,w_2)$ (see
Lemma~\ref{quaternionictohermitesch}), i.e., $\chi(\lambda)=\I_\lambda$, where
$\I_\lambda$ is defined by the equation
\begin{equation}\label{Ilambda}
s(\I_\lambda w_1,w_2)=\omega_\lambda(w_1,w_2)
\end{equation}
for all $w_1,w_2\in W$.
The $\mathcal{N}_{\Sp(n)}(W)$-equivariant $\RR$-linear endomorphisms
of $W$ form a division algebra over $\RR$ that is isomorphic to $\RR$, $\CC$ or
$\HH$. We are going to successively consider the cases that $\dim_\RR(\chi(\HH))$ equals $1$ 
(see Lemma~\ref{Lem:spn1notredrealfirstcase}), $2$ (see Lemma
\ref{Lem:spn1notredrealmiddletwocases}) or is bigger than or equal to $3$ (see
Lemma~\ref{Lem:spn1notredreallastcase}). We will see in Remark \ref{BtDcases}
that these cases do indeed correspond to the cases that
$\End_{\mathcal{N}_{\Sp(n)}(W)}(W)$ is isomorphic to $\RR$, $\CC$ or $\HH$
respectively.

Let us start with the case that $\dim_\RR(\chi(\HH))=1$. Since we already noted 
that the kernel of $\chi$ lies in $\im(\HH)$, we obtain that $\omega_i|_{W\times 
W}=\omega_j|_{W\times W}=\omega_k|_{W\times W}=0$ in this case.

\begin{lemma}\label{Lem:spn1notredrealfirstcase}
Let $W$ be a totally real subspace of $\HH^{n}$ of real dimension $l$ with 
quaternionification $U:=W\oplus Wi\oplus Wj\oplus Wk$. If 
$W\subset U$ is isotropic with respect to $\omega_i$, $\omega_j$ and $\omega_k$ and 
$\mathcal{N}_{\Sp(n)}(W)$ acts irreducibly on $W$, then $U$ is a quaternionic 
irreducible $\mathcal{N}_{\Sp(n)}(W)$-representation, $\mathcal{N}_{\Sp(n)}
(W)=\OO(W)\times\Sp(W^{\perp_q})$ acts transitively on the spheres in $W$, 
$\lie{l}\cong\lisp(1)$ acts trivially on $W$, and 
$W$ lies in the same $\Sp(n)$-orbit as $W_{\RR,l}:=\RR^l\times\{0\}^{n-l}$.
\end{lemma}

\begin{proof}
Since $W\subset U$ is isotropic with respect to $\omega_i$ the spaces $W$ 
and $W\cdot i$ are $s$-orthogonal. Any 
element in $\mathcal{N}_{\Sp(n)}(W)$ respects $W$, $W^{\perp_q}$ and leaves 
$s|_{W\times W}$ invariant and lies therefore in 
$\OO(W)\times\Sp(W^{\perp_q})$. Conversely, any element in $\OO(W)$ acts 
$\HH$-linearly on $U$. Again, due to the fact that $W$ is isotropic with respect to
$\omega_i$, $\omega_j$ and $\omega_k$, the quaternionic inner product 
$q|_{U\times U}$ is $\OO(W)$-invariant, which implies 
$\OO(W)\times\Sp(W^{\perp_q})\subset\mathcal{N}_{\Sp(n)}(W)$. This shows
$\mathcal{N}_{\Sp(n)}(W)=\OO(W)\times\Sp(W^{\perp_q})$. Consequently,
Remark \ref{U(W)irredonV} implies that $\mathcal{N}_{\Sp(n)}(W)$ acts transitively on 
the spheres in $W$ and quaternionic irreducibly on $U$.

If we choose an $s$-orthonormal basis $(w_1,\ldots,w_l)$ over $\RR$ of $W$ then 
$(w_1,\ldots,w_l)$ is a $q$-orthonormal basis over $\HH$ of $U$. Completing it to a 
$q$-orthonormal basis $(w_1,\ldots,w_{n})$ of $\HH^{n}$ we obtain the element 
$(w_1\cdots w_{n})\in\Sp(n)$ that maps $W_{\RR,l}$ to $W$. A direct calculation 
yields
\begin{align*}
\mathcal{N}_{\lisp(n)\oplus\lisp(1)}(W_{\RR,l})&=\left\{\left(A,\xi\right)\in\lisp(n)\oplus\lisp(1):
Aw-w\xi\in W_{\RR,l}\ \forall\ w\in W_{\RR,l}\right\}\\
&=\left\{\left(A,\xi\right)\in\lisp(n)\oplus\lisp(1):A=\left(\begin{smallmatrix}A_1&0\\0&*
\end{smallmatrix}\right),A_1-\xi\id\in\RR^{l\times l}\right\}\\
&=\underbrace{\left\{\left(A,0\right)\in\lisp(n)\oplus\lisp(1):A=\left(\begin{smallmatrix}A_1&0\\0
&*\end{smallmatrix}\right),A_1\in\RR^{l\times l}\right\}}_{=\so(l)\oplus
\lisp(n-l)}\\
&\quad\oplus\underbrace{\left\{\left(A,\xi\right)\in\lisp(n)\oplus\lisp(1):A=\left(
\begin{smallmatrix}\xi\id&0\\0&*\end{smallmatrix}\right)\right\}}_{\cong\lisp(1)},
\end{align*}
i.e., $\lie{l}\cong\lisp(1)$ in this case and $\lie{l}$ acts 
trivially on $W_{\RR,l}$ (see Lemma \ref{lactstrivially}).
\end{proof}

Let us now assume that $\dim_\RR(\chi(\HH))=2$, i.e., the kernel of $\chi$ is a 
$2$-dimensional subspace of $\im(\HH)$. Its orthogonal complement in $\HH$ 
is also $2$-dimensional and contains $\RR$. Let us assume that it is equal to 
$F_\lambda=\RR\oplus\RR\lambda$ for some $\lambda\in\im(\HH)\cap\Sp(1)$, 
i.e., $\chi|_{F_\lambda}:F_\lambda\to\RR\id\oplus\RR\I_\lambda$ is an 
isomorphism. We choose the $\re q_\HH$-orthonormal basis 
$(1,\lambda,\mu,\lambda\mu)$ of $\HH$. Then 
the kernel of $\chi$ is equal 
to $\RR\mu\oplus\RR\lambda\mu=F_\lambda\mu$.
We set $U=W\oplus W\lambda\oplus W\mu\oplus W\lambda\mu$ and
denote by $\sigma_\lambda:U\to U$  the $F_\lambda$-anti-linear and $F_\mu$-linear involution
that fixes $W$ pointwise. Analogously we let $\sigma_\mu:U\to U$ denote the $F_\mu$-anti-linear 
and $F_\lambda$-linear involution that fixes $W$ pointwise. The two
involutions $\sigma_\lambda$ and $\sigma_\mu$ are 
$\mathcal{N}_{\Sp(n)}(W)$-equivariant and commute.
Furthermore $W$ is the joint set of fixed points 
of $\sigma_\lambda$ and $\sigma_\mu$ in $U$, i.e., 
$W=(U^{\sigma_\lambda})^{\sigma_\mu}
=(U^{\sigma_\mu})^{\sigma_\lambda}$.
The map $f_{\sigma_\mu}:U\to U^{\sigma_\mu}$ defined by $u\mapsto\tfrac{1}{2}(u+\sigma_\mu(u))$ is 
$F_\lambda$-linear, $\mathcal{N}_{\Sp(n)}(W)$-equivariant, surjective 
and has kernel $W\mu\oplus W\lambda\mu$.

\begin{lemma}\label{Lem:spn1notredrealmiddletwocases}
Let $W$ be a totally real subspace of $\HH^{n}$ of real dimension $l$ with 
quaternionification $U:=W\oplus W\lambda\oplus W\mu\oplus W\lambda\mu$ 
such that $\mathcal{N}_{\Sp(n)}(W)$ acts irreducibly on $W$. We assume 
furthermore that $\I_\lambda\neq 0$ on $W$ and 
$\omega_\mu|_{W\times W}=\omega_{\mu\lambda}|_{W\times W}=0$, where
$\I_\lambda$ is defined by~\eqref{Ilambda}. Then
$U$ decomposes $q$-orthogonally as
$U=U_{\h,\xi}\oplus U_{\h,-\xi}$, where $U_{\h,\xi}$ and 
$U_{\h,-\xi}$ are quaternionic irreducible 
$\mathcal{N}_{\Sp(n)}(W)$-representations while $f_{\sigma_\mu}(U_{\h,\xi})$ and 
$f_{\sigma_\mu}(U_{\h,-\xi})$ are $F_\lambda$-irreducible 
$\mathcal{N}_{\Sp(n)}(W)$-representations that are as real 
representations isomorphic to $W$. Furthermore 
\begin{align*}
\mathcal{N}_{\Sp(n)}(W)&=\{(g,\varphi(g))\in\U
(f_{\sigma_\mu}(U_{\h,\xi}),h_\lambda)\times\U
(f_{\sigma_\mu}(U_{\h,-\xi}),h_\lambda)\}\times\Sp(W^{\perp_q}),\\
\lie{l}&\cong\lie{u}(1),
\end{align*}
where $\varphi:\U(f_{\sigma_\mu}(U_{\h,\xi}),h_\lambda)\to\U(f_{\sigma_\mu}(U_{\h,-\xi}),h_\lambda)$ 
defined by $\varphi(g)(v)=\sigma_\lambda(g\cdot\sigma_\lambda(v))$ for all 
$v\in f_{\sigma_\mu}(U_{\h,-\xi})$ is a Lie group isomorphism 
and $W$ lies in the same $\Sp(n)$-orbit as 
$W_{\RR,l,r,F_\lambda}:=\left\{\left(\begin{smallmatrix}x+\lambda y\\r(x-\lambda y)\\0
\end{smallmatrix}\right):x,y\in \RR^{l/2}\right\}$ for some $r\in (0,1)$.
Furthermore $\mathcal{N}_{\Sp(n)}(W)$ acts transitively on the spheres in $W$
while $\lie{l}$ acts trivially on $W$.
\end{lemma}
\begin{proof}
We denote by $V_\lambda$ the complexification $W\oplus W\lambda$ of $W$. Note 
that $V_\lambda=U^{\sigma_\mu}$. The assumption $\omega_\mu|_{W\times 
W}=\omega_{\mu\lambda}|_{W\times W}=0$ implies $\omega_\mu|_{V_\lambda\times 
V_\lambda}=\omega_{\mu\lambda}|_{V_\lambda\times V_\lambda}=0$. This gives in 
turn that $V_\lambda$ is $h_\lambda$-orthogonal to $V_\lambda\cdot\mu$. 
Therefore any $h_\lambda$-orthonormal basis over $F_\lambda$ of $V_\lambda$ is a 
$q$-ortho\-normal basis over $\HH$ of $U$. In particular, if we choose an 
$h_\lambda$-orthonormal basis of $V_\lambda$ and complete it to a 
$q$-orthonormal basis of $\HH^{n}$, we obtain an element in $\Sp(n)$ that maps 
$V_{F_\lambda,l}:=F_\lambda^l\times\{0\}^{n-l}$ to $V_\lambda$. This implies that
$\mathcal{N}_{\Sp(n)}(V_\lambda)=\U(V_\lambda,h_\lambda)\times\Sp(U^{\perp_q})$
acts $F_\lambda$-irreducibly on $V_\lambda$ and quaternionic irreducibly on $U$
(see Remark \ref{U(W)irredonV}). Since $\Sp(n)$ acts $\HH$-linear on 
$\HH^{n}$ any element in $\mathcal{N}_{\Sp(n)}(W)$ leaves $V_\lambda$
invariant. As a result we obtain
\begin{align*}
\mathcal{N}_{\Sp(n)}(W)&\subset\mathcal{N}_{\Sp(n)}(V_\lambda)
=\U(V_\lambda,h_\lambda)\times\Sp(U^{\perp_q})
\end{align*}
and the action of $\mathcal{N}_{\Sp(n)}(W)$ on $W$ is orbit equivalent to
the action of the first factor, i.e., the action of 
$\mathcal{N}_{\U(V_\lambda,h_\lambda)}(W)$. Therefore 
$\mathcal{N}_{\U(V_\lambda,h_\lambda)}(W)$ acts irreducibly on $W$.

Since $\I_\lambda$ is defined as in Section \ref{Section:sun} we obtain that the 
kernel of $\I_\lambda:W\to W$ is equal to the maximal subspace of $W$ on which 
$\omega_\lambda$ is degenerate. Since $W$ is a maximal totally real subspace of 
the Hermitian vector space $(V_\lambda,h_\lambda)$ and $\Sp(n)$ is a subgroup 
of $\U(F_\lambda^{2n},h_\lambda)$, the requirements of 
Proposition~\ref{Prop:Un1notredmain} are met and we conclude (since
$\I_\lambda\neq 0$) 
that $W\subset V_\lambda$ is symplectic with respect to $\omega_\lambda$, 
$V_\lambda$ decomposes $h_\lambda$-orthogonally as 
$V_\lambda=V_{\lambda,\xi}\oplus V_{\lambda,-\xi}$ into two 
$F_\lambda$-irreducible 
$\mathcal{N}_{\U(F_\lambda^{2n},h_\lambda)}(W)$-representations that are as real 
representations isomorphic to $W$ and which are eigenspaces of the 
$F_\lambda$-linear continuation of $\I_\lambda:W\to W$ to $V_\lambda$ for the
eigenvalues $\pm\xi\in\lambda\RR\backslash\{0\}$. The involution
$\sigma_\lambda$ induces an $\mathcal{N}_{\Sp(n)}(W)$-equivariant and 
$F_\lambda$-anti-linear isomorphism between $V_{\lambda,\xi}$ and 
$V_{\lambda,-\xi}$ and
\begin{align*}
\mathcal{N}_{\U(V_\lambda,h_\lambda)}(W)
=\{(g,\varphi(g))\in\U(V_{\lambda,\xi})\times\U(V_{\lambda,-\xi})\},
\end{align*}
where the Lie group isomorphism $\varphi:\U(V_{\lambda,\xi})\to\U(V_{\lambda,-\xi})$ is given by
$\varphi(g)(v)=\sigma_\lambda(g\sigma_\lambda(v))$ for all $g\in\U(V_{\lambda,\xi})$ and $v\in V_{\lambda,-\xi}$.

We already stated that the action of 
$\mathcal{N}_{\Sp(n)}(W)$ on $W$ and $V_\lambda$ is orbit equivalent
to the action of $\mathcal{N}_{\U(V_\lambda,h_\lambda)}(W)$ on $W$ and $V_\lambda$.
As a result
the subspaces $V_{\lambda,\pm\xi}$ are two $F_\lambda$-irreducible 
$\mathcal{N}_{\Sp(n)}(W)$-representations that
are as real representations isomorphic to $W$ and 
\begin{align*}
\mathcal{N}_{\Sp(n)}(W)
=\{(g,\varphi(g))\in\U(V_{\lambda,\xi},h_\lambda)\times
\U(V_{\lambda,-\xi},h_\lambda)\}\times\Sp(U^{\perp_q}).
\end{align*}

Note that due to Remark~\ref{U(W)irredonV} the induced representations of 
$\mathcal{N}_{\Sp(n)}(W)$ on 
$U_{\I_\lambda,\pm\xi}:=\langle V_{\lambda,\pm \xi}\rangle_\HH$
are as quaternionic representations irreducible.
We conclude that $\mathcal{N}_{\Sp(n)}(W)$
acts transitively on the spheres in $W$, $F_\lambda$-irreducible on 
$V_{\lambda,\pm\xi}$ and
quaternionic irreducibly on $U_{\h,\pm\xi}$.

If we map $V_\lambda$ to $V_{F_\lambda,l}$ by an element in
$\Sp(n)$ as described above, Corollary \ref{Cor:unitnormform} yields
that $W$ lies in the same $\U(V_\lambda,h_\lambda)$-orbit as 
$W_{\RR,l,r,F_\lambda}$ for some $r\in (0,1)$. A direct 
calculation shows $\lie{l}\cong\lie{u}(1)$. The action of $\lie{l}$ on $W$ 
is trivial (see Lemma \ref{lactstrivially}). 
\end{proof}

Let us now assume $\dim_\RR(\chi(\HH))\geq 3$. In particular
$\End_{\mathcal{N}_{\Sp(n)}(W)}(W)\cong\HH$. In order to simplify
this case we are not only going to assume that $\mathcal{N}_{\Sp(n)}(W)$
acts irreducibly on $W$ but we are going to use the stronger condition that
$\mathcal{N}_{\Sp(n)}(W)$ acts transitively on the spheres in $W$. 
As we noted at the begin of this section this is always the case if 
$\mathcal{N}_M(W)$ acts transitively on the spheres 
in $W$ and $\dim_\RR(W)> 4$. The case $\dim_\RR(W)\leq 4$ will be 
considered later.

\begin{lemma}\label{Lem:spn1notredreallastcase}
Let $W$ be a totally real subspace of $\HH^{n}$ of real dimension $l$ with 
quaternionification $U:=W\oplus W\lambda\oplus W\mu\oplus W\lambda\mu$ 
such that $\mathcal{N}_{\Sp(n)}(W)$ acts transitively on the spheres in $W$. 
We assume 
$\End_{\mathcal{N}_{\Sp(n)}(W)}(W)\cong\HH$.
Then $U=\langle W\rangle_\HH$ decomposes into four isomorphic
quaternionic irreducible 
$\mathcal{N}_{\Sp(n)}(W)$-representations
that are as real $\mathcal{N}_{\Sp(n)}(W)$-representations 
isomorphic to $W$, $\mathcal{N}_{\Sp(n)}(W)$ is isomorphic to
$\Sp(W)\times\Sp(W^{\perp_q})$ and its action on $W$ is given by the
standard action of $\Sp(W)$ on $W$.
In this case $W$ lies in the same $\Sp(n)$-orbit as
\begin{align*}
W_{\RR,l,z}:=\left\{\left(\begin{smallmatrix}u\\uz_2\\uz_3\\uz_4\\0\end{smallmatrix}\right):u\in \HH^{l/4}\right\}
\end{align*}
for some $z\in\mathcal{P}$, where
\begin{align*}
\mathcal{P}&=\left\{(z_2,z_3,z_4)^t\in(\HH^*)^3:\langle 1,z_2,z_3,z_4\rangle_\RR=\HH,\right.\\
&\left.\quad\quad \dim_\RR(\im(\langle \xi+\ol{z_2}\xi z_2+\ol{z_3}\xi z_3+\ol{z_4}\xi z_4:\xi\in\HH\rangle_\RR))\geq 2\right\}. 
\end{align*}
Furthermore $\mathcal{N}_{\Sp(n)\times\Sp(1)}(W_{\RR,l,z})=\mathcal{N}_{\Sp(n)}(W_{\RR,l,z})$ for all $z\in\mathcal{P}$.
\end{lemma}
\begin{proof}
Since $\End_{\mathcal{N}_{\Sp(n)}(W)}(W)$ is assumed to be isomorphic 
to $\HH$, Theorem \cite[Theorem II.6.7]{BtD} implies that 
$W\in\Irr(\mathcal{N}_{\Sp(n)}(W),\RR)_\HH$. This in turn implies
that the complexification $V:=\CC\otimes_\RR W$ decomposes as
$\CC\otimes_\RR W=V_1\oplus V_2$ into
two complex irreducible $\mathcal{N}_{\Sp(n)}(W)$-representations
$V_1,V_2$, which are as complex $\mathcal{N}_{\Sp(n)}(W)$-representations
isomorphic and as real $\mathcal{N}_{\Sp(n)}(W)$-representations 
isomorphic to $W$ (see Proposition \cite[Proposition II.6.6 (iii)]{BtD}).
Moreover, for $m\in\{1,2\}$, the quaternionification $\HH\otimes_\CC V_m$ of 
$V_m$
decomposes as $\HH\otimes_\CC V_m=U_{m,1}\oplus U_{m,2}$
into two quaternionic irreducible 
$\mathcal{N}_{\Sp(n)}(W)$-representations, which are as quaternionic
$\mathcal{N}_{\Sp(n)}(W)$-representations
isomorphic and as real $\mathcal{N}_{\Sp(n)}(W)$-representations 
isomorphic to $W$ (see Proposition \cite[Proposition II.6.6 (ix)]{BtD}).
Overall this shows that the quaternionification 
$U:=\HH\otimes_\RR W=\HH\otimes_\CC V$
decomposes as $U_{1,1}\oplus U_{1,2}\oplus U_{2,1}\oplus U_{2,2}$ into
four quaternionic irreducible 
$\mathcal{N}_{\Sp(n)}(W)$-representations, which are as quaternionic
$\mathcal{N}_{\Sp(n)}(W)$-representations
isomorphic and as real $\mathcal{N}_{\Sp(n)}(W)$-representations 
isomorphic to $W$.

Let $g:W\to U_{1,1}$ be an $\mathcal{N}_{\Sp(n)}(W)$-equivariant $\RR$-linear 
isomorphism. Since we assumed that $\mathcal{N}_{\Sp(n)}(W)$ acts transitively 
on the spheres in $W$, it also acts transitively on the spheres in $U_{1,1}$. 
Furthermore $\mathcal{N}_{\Sp(n)}(W)$ respect the restriction $q|_{U_{1,1}\times U_{1,1}}$
and therefore acts as a subgroup of $\Sp(U_{1,1})$ on $U_{1,1}$.
Due to Onishchik's classification (see Theorem \ref{Thm2On})
there exists no proper subgroup of $\Sp(U_{1,1})$ that acts transitively on the spheres
in $U_{1,1}$. Therefore the action of $\mathcal{N}_{\Sp(n)}(W)$ on $U_{1,1}$ is the action of 
$\Sp(U_{1,1})$ on $U_{1,1}$. If we pull back the quaternionic structure on $U_{1,1}$ via $g:W\to U_{1,1}$ 
to $W$ we obtain a quaternionic structure on $W$ that is respected by $\mathcal{N}_{\Sp(n)}(W)$.
Moreover the action of $\mathcal{N}_{\Sp(n)}(W)$ on $W$ is the action of 
$\Sp(W)$ on $W$. Since $\mathcal{N}_{\Sp(n)}(W)$ leaves $W^{\perp_q}$ invariant we conclude
$\mathcal{N}_{\Sp(n)}(W)=\Sp(W)\times \Sp(W^{\perp_q})$.

Since the four quaternionic irreducible 
$\mathcal{N}_{\Sp(n)}(W)$-representations $U_{1,1},U_{1,2},U_{2,1},U_{2,2}$ are 
isomorphic, we may assume them to be $q$-orthogonal and thus $W$ lies in the 
same $\Sp(n)$-orbit as $W_\psi=\left\{\left(\begin{smallmatrix} 
u\\\psi_2(u)\\\psi_3(u)\\\psi_4(u)\\0\end{smallmatrix}\right):u\in\HH^{l/4} 
\right\}$, for some 
$\psi_2,\psi_3,\psi_4\in\End_{\Sp(l/4)}(\HH^{l/4})\backslash\{0\}\cong\HH^*$. A 
direct calculation then shows that the isomorphisms have to fulfill the 
constraints given by the set $\mathcal{P}$ in order for $W_\psi$ to have the 
assumed properties. For the technical details of the proof of the normal form we 
refer the reader to \cite[Proposition 7.42]{K}.
\end{proof}

\begin{rem}\label{BtDcases}
Proposition 6.6 and Theorem 6.7 in \cite{BtD} imply that
Lemmas~\ref{Lem:spn1notredrealfirstcase},
\ref{Lem:spn1notredrealmiddletwocases} and
\ref{Lem:spn1notredreallastcase} 
cover the cases that the division algebra
$\End_{\mathcal{N}_{\Sp(n)}(W)}(W)$ is isomorphic 
to $\RR$, $\CC$ or $\HH$ respectively.
\end{rem}

The following well-known Lemma \ref{Lem:automorphismofH} implies that 
the element $\lambda\in\im(\HH)\cap\Sp(1)$ in Proposition
\ref{Prop:SPn1notredmain}, 
Lemmas \ref{SPn1cplxnormalform},  \ref{Lem:spn1notredrealmiddletwocases} 
and \ref{Lem:spn1notredreallastcase} may be chosen to be $i$
(while $\mu$ may be chosen as $j$).

For the proof of the following result we refer the reader
to~\cite[Lemma~7.45]{K}.

\begin{lemma}\label{Lem:automorphismofH}
Any ring automorphism $\psi$ of $\HH$ is inner, i.e., for any $\psi$ there exists a 
$q_\psi\in\HH\backslash\{0\}$ such that $\psi(z)=q_\psi zq_\psi^{-1}$. Furthermore 
$q_\psi$ can be chosen in $\Sp(1)$.

Furthermore the action of any $p\in\Sp(1)$ on $\im(\HH)\cong
\RR^3$ given by $pvp^{-1}$ is a special orthogonal transformation, i.e., the map 
$\Sp(1)\to\SO(3)$, $p\mapsto (v\mapsto pvp^{-1})$ for all $v\in\im(\HH)$ is a group 
homomorphism with kernel $\{\pm 1\}$. 
\end{lemma}
At this point we have found all quaternionic and complex subspaces $W$
such that the normalizer $\mathcal{N}_{\Sp(n)\times\Sp(1)}(W)$ acts irreducibly on $W$ and all
real subspaces $W$ of $\HH^{n}$ such that 
$\mathcal{N}_{\Sp(n)}(W)$ acts transitively on the spheres in $W$. Furthermore we
have shown that for all such $W$ the normalizer 
$\mathcal{N}_{\Sp(n)}(W)$ (and therefore also 
$\mathcal{N}_{\Sp(n)\times\Sp(1)}(W)$) acts transitively on the connected components 
of the spheres in $W$.

The goal of this section is to find all real subspaces $W$ of $\HH^{n}$
such that $\mathcal{N}_{\Sp(n)\times\Sp(1)}(W)$ acts transitively on the connected 
components of the spheres in $W$. If $\mathcal{N}_{\Sp(n)}(W)$
does not act irreducibly on $W$ (in particular if $\dim_\RR(W)>1$)
then $L$ has to act transitively on the spheres in $W$ forcing 
$\dim_\RR(W)\leq 4$. 
This is the reason why we will consider the cases $2\leq \dim_\RR(W)\leq 4$
from now on.
Even though we could assume that $L$ acts transitively on the spheres 
in $W$ we rather just assume $\dim_\RR(W)\leq 4$ and see what this 
implies for the $L$-action on $W$. Note that we still assume that 
$W\subset\HH^{n}$ is a real subspace such that 
$W\neq W\lambda$ for all $\lambda\in\HH\backslash\RR$ and that
\begin{align*}
\lie{l}=\{(\phi(\alpha),\alpha)\in\lisp(n)\times\lisp(1):
\alpha\in\pi_2(\mathcal{N}_{\lisp(n)\oplus\lisp(1)}(W))\}
\end{align*}
for a Lie algebra homomorphism $\phi:\pi_2(\mathcal{N}_{\lisp(n)\oplus\lisp(1)}(W))\to\lisp(n)$,
which is injective by the assumption 
$\mathcal{N}_{\{0\}\oplus\lisp(1)}(W)=\{0\}$. If $\pi_2(\mathcal{N}_{\lisp(n)\oplus\lisp(1)}(W))=
\lisp(1)$ then the simply connectedness of $\Sp(1)$ implies that $\phi$ lifts to a Lie 
group homomorphism.

Note that $\dim_\RR(\langle W\rangle_\HH)\leq 4 \dim_\RR(W)$ and that 
$\dim_\RR(\langle W\rangle_\HH)$ is divisible by four. Therefore we have the following cases
\begin{center}
\begin{tabular}{cl}
$\dim_\RR(W)$&$\dim_\HH(\langle W\rangle_\HH)$\\
\hline
4&$1,\ 2^\bigstar,\ 3^\bigstar,\ 4^*$\\
3&$1,\ 2^\bigstar,\ 3^*$\\
2&$1,\ 2^*$,
\end{tabular}
\end{center}
where the case $\dim_\RR(W)=4$, $\dim_\HH(\langle W\rangle_\HH)=1$ and the case 
$\dim_\RR(W)=2$, $\dim_\HH(\langle W\rangle_\HH)=1$ are excluded by the assumption 
$W\neq W\lambda$ for all $\lambda\in\HH\backslash\RR$. 

The cases $^*$ are excluded by Lemma \ref{lactstrivially}, while the cases with
$^\bigstar$ are excluded by \cite[Lemma 7.48, Lemma 7.49]{K}).

Hence we only need to consider the case where $\dim_\RR(W)=3$ and 
$\dim_\HH(\langle W\rangle_\HH)=1$. 

\begin{lemma}\label{3dimrealinH}
Let $\dim_\RR(W)=3$ and $\dim_\HH(\langle W\rangle_\HH)=1$. Then $W$ lies in 
the same $\Sp(n)$-orbit as $W_{\RR,3,1}:=\im(\HH)\times\{0\}^{n-1}$ and the 
action of $L\cong\Sp(1)$ on $W$ is isomorphic to the defining representation of 
$\SO(3)$ on $\RR^3$ and 
$\mathcal{N}_{\lisp(n)}(W)=\{0\}\times\lisp(W^{\perp_q})$ acts trivially on 
$W$.
\end{lemma}
\begin{proof}
Since $L$ respects the real inner product $s$, the real $1$-dimensional subspace 
$W^{\perp_s}\cap \langle W\rangle_\HH$ is invariant under the action of $L$. Let us choose an 
element $v_1\in W^{\perp_s}\cap \langle W\rangle_\HH$ with $q(v_1,v_1)=1$. Then we can 
complete $v_1$ to a quaternionic basis of $\HH^{n}$ and obtain an element in 
$\Sp(n)$ that maps $W_{\RR,3,1}$ to $W$. Then
\begin{align*}
\mathcal{N}_{\lisp(n)\oplus\lisp(1)}(W_{\RR,3,1})
&=\left\{\left(A,\xi\right)\in\lisp(n)\oplus\lisp(1):A=\left(\begin{smallmatrix}
\xi&0&\ldots&0\\0&&&
\\\vdots&&*&\\0&&&\end{smallmatrix}\right)\right\}\\
\mathcal{N}_{\lisp(n)}(W_{\RR,3,1})
&=\left\{A\in\lisp(n):A=\left(\begin{smallmatrix}0&\ldots&0\\\vdots&*&\\0&&
\end{smallmatrix}\right)\right\}
\end{align*}
\begin{align*}
\lie{l}
&=\left\{\left(A,\xi\right)\in\lisp(n)\oplus\lisp(1):A=\left(\begin{smallmatrix}\xi&0&\ldots&0\\
0&&&\\\vdots&&0&\\0&&&\end{smallmatrix}\right)\right\}\cong\lisp(1).
\end{align*}
Now the proof follows from Lemma~\ref{Lem:automorphismofH}.
\end{proof}

We summarize the work of this section in the following theorem.

\begin{theorem}\label{spn1notredsummary}
Let $W$ be a real subspace of $\HH^{n}$. If $\mathcal{N}_{\Sp(n)\times\Sp(1)}(W)$ acts 
irreducibly on $W$ then $W$ lies in the same $\Sp(n)\times\Sp(1)$-orbit as one of the following
\begin{align*}
W_{\HH,l}&=\HH^l\times\{0\}^{n-l},&0\leq l\leq n\\
W_{\CC,l}&=\CC^l\times\{0\}^{n-l},&0< l\leq n\\
W_{\CC,2l,t}&=\left\{\left(\begin{smallmatrix}z+j w\\t(z-j w)\\0
\end{smallmatrix}\right):z,w\in \CC^l\right\},&0< l\leq \lfloor\tfrac{n}{2}\rfloor\\
W_{\RR,l}&=\RR^l\times\{0\}^{n-l},&0< l\leq n\\
W_{\RR,2l,t,\CC}&=\left\{\left(\begin{smallmatrix}z\\t\ol{z}\\0\end{smallmatrix}
\right):z\in \CC^l\right\},&0< l\leq \lfloor\tfrac{n}{2}\rfloor\\
W_{\RR,4l,z}&=\left\{\left(\begin{smallmatrix}v\\vz_2\\vz_3\\vz_4\\0\end{smallmatrix}
\right):v\in \HH^l\right\},&0< l\leq \lfloor\tfrac{n}{4}\rfloor\\
W_{\RR,3,1}&=\im(\HH)\times\{0\}^{n-1},&
\end{align*}
for some $t\in\RR^{>0}\backslash\{1\}$, $z\in\mathcal{P}$. Furthermore 
\begin{align*}
\mathcal{N}_{\lisp(n)\oplus\lisp(1)}(W_{\HH,l})&
=(\lisp(l)\oplus\lisp(n-l))\oplus\lisp(1),\\
\mathcal{N}_{\lisp(n)\oplus\lisp(1)}(W_{\CC,l})&
=(\liu(l)\oplus\lisp(n-l))\oplus \lie{u}(1),\\
\mathcal{N}_{\lisp(n)\oplus\lisp(1)}(W_{\CC,2l,t})&
\cong(\lisp(l)\oplus\lisp(n-2l))\oplus \lie{u}(1),\\
\mathcal{N}_{\lisp(n)\oplus\lisp(1)}(W_{\RR,l})&
\cong(\so(l)\oplus\lisp(n-l))\oplus\lisp(1),\\
\mathcal{N}_{\lisp(n)\oplus\lisp(1)}(W_{\RR,2l,t,\CC})&
\cong(\liu(l)\oplus\lisp(n-2l))\oplus\lie{u}(1),\\
\mathcal{N}_{\lisp(n)\oplus\lisp(1)}(W_{\RR,4l,z})&
\cong(\lisp(l)\oplus\lisp(n-4l))\oplus\{0\},\\
\mathcal{N}_{\lisp(n)\oplus\lisp(1)}(W_{\RR,3,1})&
\cong\lisp(n-1)\oplus\lisp(1).
\end{align*}
\end{theorem}
\begin{proof}
The proof for $W_{\HH,l}$ follows directly from Lemma~\ref{Lem:Wquaternionic}.
Using Lemma~\ref{Lem:automorphismofH}
the proof for $W_{\CC,l}$ and $W_{\CC,2l,t}$ follows from Corollary \ref{SPn1cplxnormalform}, while the proof for
$W_{\RR,l}$, $W_{\RR,2l,t,\CC}$ and $W_{\RR,4l,z}$ follow from Lemmas
\ref{Lem:spn1notredrealfirstcase},
\ref{Lem:spn1notredrealmiddletwocases} and
\ref{Lem:spn1notredreallastcase} respectively.
The proof for $W_{\RR,3,1}$ follows from Lemma \ref{3dimrealinH}.
\end{proof}

\subsection{Non-reductive spherical subalgebras}
In this subsection we describe all non-re\-duc\-tive spherical algebraic 
subalgebras of $\lisp(n,1)$, up to the action of $M$. Their unipotent radicals 
are of the form $W^\perp\oplus\g_{2\alpha}$, where $W\subset\g_\alpha$ is one of 
the subspaces described in Theorem \ref{spn1notredsummary}. In order to find all 
possibilities for their maximal compact subalgebras we apply Onishchik's Theorem 
\ref{Thm2On} to their normalizers and thus obtain (with Remark \ref{sp(1)+l2}) 
the following.

\begin{theorem}\label{NonredSympl}
Every spherical non-reductive algebraic subalgebra of $\g=\lisp(n,1)$,
$n\geq2$, is $G$-conjugate to one in the following list, where
$\lie{b}_j\subset\lisp(n-1-j)$, $\lie{c}\subset\lisp(1)$ and
$\lie{d}\subset\lie{u}(1)$ are arbitrary (under the condition displayed in
italic in Remark~\ref{non-reductive,embedding} adapted to this situation).
\setlength\arraycolsep{7pt}
\begingroup
\renewcommand*{\arraystretch}{1.2}
\begin{equation*}
\begin{array}{|c|c|}\hline
\lie{l}_H\oplus\lie{n} & \lie{l}_H\subset\lie{m}\oplus\lie{a}\text{ arbitrary}\\
\hline
\lisp(k)\oplus\lie{b}_l\oplus\lie{c}\oplus\lie{a}\oplus\lie{n}_{\HH,l} &1\leq l\leq n-1,\\
\lie{c}\oplus\lie{b}_1\oplus\lisp(1)\oplus\lie{a}\oplus\lie{n}_{\HH,1} &\\
\hline
\liu(l)\oplus\lie{b}_l\oplus\lie{d}\oplus\lie{a}\oplus\lie{n}_{\CC,l} & 1\leq l\leq n-1\\
\su(l)\oplus\lie{b}_l\oplus\lie{d}\oplus\lie{a}\oplus\lie{n}_{\CC,l} & 2\leq l\leq n-1\\
\lisp(m)\oplus\lie{u}(1)\oplus\lie{b}_l\oplus\lie{d}\oplus\lie{a}\oplus
\lie{n}_{\CC,l} 
& 1\leq l\leq n-1, l=2m, m\geq 2\\
\lisp(m)\oplus\lie{b}_l\oplus\lie{d}\oplus\lie{a}\oplus\lie{n}_{\CC,l} 
& 1\leq l\leq n-1, l=2m, m\geq2\\
\hline
\lisp(l)\oplus\lie{b}_{2l}\oplus\lie{d}\oplus\lie{a}\oplus\lie{n}_{\CC,2l,t}
&1\leq 2l\leq n-1, l=2m, t\in\RR^{>0}\backslash\{1\}\\
\hline
\so(l)\oplus\lie{b}_l\oplus\lie{c}\oplus\lie{a}\oplus\lie{n}_{\RR,l}&1\leq l\leq n-1\\
\liu(m)\oplus\lie{b}_l\oplus\lie{c}\oplus\lie{a}\oplus\lie{n}_{\RR,l}
&1\leq l\leq n-1, l=2m, m\geq 3\\
\su(m)\oplus\lie{b}_l\oplus\lie{c}\oplus\lie{a}\oplus\lie{n}_{\RR,l}
&1\leq l\leq n-1, l=2m, m\geq 3\\
\lisp(m)\oplus\lisp(1)\oplus\lie{b}_l\oplus\lie{c}\oplus\lie{a}\oplus\lie{n}_{\RR,l}
&1\leq l\leq n-1, l=4m, m\geq 2\\
\lisp(m)\oplus\lie{u}(1)\oplus\lie{b}_l\oplus\lie{c}\oplus\lie{a}\oplus
\lie{n}_{\RR,l}
&1\leq l\leq n-1, l=4m, m\geq 1\\
\lisp(m)\oplus\lie{b}_l\oplus\lie{c}\oplus\lie{a}\oplus\lie{n}_{\RR,l}
&1\leq l\leq n-1, l=4m, m\geq 1\\
\so(9)\oplus\lie{b}_{16}\oplus\lie{c}\oplus\lie{a}\oplus\lie{n}_{\RR,16}&\\
\so(7)\oplus\lie{b}_8\oplus\lie{c}\oplus\lie{a}\oplus\lie{n}_{\RR,8}&\\
\lie{g}_2\oplus\lie{b}_7\oplus\lie{c}\oplus\lie{a}\oplus\lie{n}_{\RR,7}&\\
\hline
\liu(l)\oplus\lie{b}_{2l}\oplus\lie{d}\oplus\lie{a}\oplus\lie{n}_{\RR,2l,t,\CC} 
& 2\leq 2l\leq n-1, t\in\RR^{>0}\backslash\{1\}\\
\su(l)\oplus\lie{b}_{2l}\oplus\lie{d}\oplus\lie{a}\oplus\lie{n}_{\RR,2l,t,\CC} 
& 4\leq 2l\leq n-1,  t\in\RR^{>0}\backslash\{1\}\\
\lisp(m)\oplus\lie{u}(1)\oplus\lie{b}_{2l}\oplus\lie{d}\oplus\lie{a}\oplus
\lie{n}_{\RR,2l,t,\CC} & 2\leq 2l\leq n-1, l=2m, m\geq2,  t\in\RR^{>0}\backslash\{1\}\\
\lisp(m)\oplus\lie{b}_{2l}\oplus\lie{d}\oplus\lie{a}\oplus\lie{n}_{\RR,2l,t,\CC} 
& 2\leq 2l\leq n-1, l=2m,m\geq 2,  t\in\RR^{>0}\backslash\{1\}\\
\hline
\lisp(l)\oplus\lie{b}_{4l}\oplus\lie{a}\oplus\lie{n}_{\RR,4l,z}
&2\leq 4l\leq n-1,z\in\mathcal{P}\\
\hline
\lisp(1)\oplus\lie{b}_1\oplus\lie{a}\oplus\lie{n}_{\RR,3,1}&\\
\hline
\end{array}
\end{equation*}
\endgroup\setlength\arraycolsep{2pt}
\end{theorem}
\begin{proof}
Here we have written $\lie{n}_{\HH,l}$ to denote
$W_{\HH,l}^\perp\oplus\g_{2\alpha}$ and so on.
Note that $\lie{n}_{\RR,1}$ is $1$-co\-di\-men\-sional and that 
$\mathcal{N}_{\lie{m}}(\lie{n}_{\RR,1})\cong\lisp(n-2)\oplus\lisp(1)$.
\end{proof}

\subsection{Reductive spherical subalgebras}

In this subsection we classify the reductive spherical subalgebras 
$\lie{h}=\lie{k}_H\oplus\p_H$ of $\lisp(n,1)$ (for $n>1$, since 
$\lisp(1,1)\cong\so(4,1)$). Since the $K$-action on $\lie{p}$ is 
essentially the same as the $M$-action on $\g_\alpha$, we may apply
again Theorem~\ref{spn1notredsummary} in order to obtain all the 
candidates for the subspace $\lie{p}_H^\perp\subset\p$. 
We use the fact, that the Lie bracket on $\p$ is given by
\begin{align*}
\left[\left[\left(\begin{smallmatrix}0&z\\\overline{z}^t&0\end{smallmatrix}\right),
\left(\begin{smallmatrix}0&w\\\overline{w}^t&0\end{smallmatrix}\right)\right],\left(
\begin{smallmatrix}0&v\\\overline{v}^t&0\end{smallmatrix}\right)\right]
&=\left(\begin{smallmatrix}0&(z\overline{w}^t-w\overline{z}^t)v-v(\overline{z}
^tw-
\overline{w}^tz)\\(\overline{z}^tw-\overline{w}^tz)\overline{v}^t-\overline{v}^t(z
\overline{w}^t-w\overline{z}^t)&0\end{smallmatrix}\right)
\end{align*}
in order to single out those Lie algebras that fulfill the additional 
restriction that 
$[\p_H,\p_H]\subset\lie{k}_H\subset\mathcal{N}_{\lie{k}}(\p_H)$ 
(see \eqref{Eqn:Constraint}) has to hold.

If $\p_H^\perp$ is equal to $W_{\CC,2l,t}$ or $W_{\RR,2l,t,\CC}$ (for some $1\leq l\leq 
\lfloor\tfrac{n}{2}\rfloor$ and some $t\in\RR^{>0}\backslash\{1\}$)
or $W_{\RR,4l,z}$ (for some  $1\leq l\leq \lfloor\tfrac{n}{4}\rfloor$, $z\in \mathcal{P}$) 
or $W_{\RR,l}$ for some $0< l\leq n$ or $W_{\RR,3,1}$
then $\lie{h}=\lie{k}_H\oplus\p_H$ cannot be spherical in $\lisp(n,1)$ if $n>1$, 
see \cite[Lemma 7.55-7.57]{K}.

If $\p_H^\perp=W_{\CC,l}$ for some $0< l\leq n$, then $\lie{h}=\lie{k}_H\oplus\p_H$ can only 
be spherical in $\lisp(n,1)$ for $n>1$ if $l=n$. In that case $\lie{h}$ is $K$-conjugate to 
$\su(n,1)$ or $\liu(n,1)$, see \cite[Lemma 7.54]{K}.

If $\p_H^\perp=W_{\HH,l}$ for some $0\leq l\leq n$, then it is shown in 
\cite[Lemma 7.53]{K} that the only reductive
spherical Lie algebras $\lie{h}=\lie{k}_H\oplus\lie{p}_H$ of $\lisp(n,1)$ for 
$n>1$ are the following
\setlength\arraycolsep{7pt}
\begingroup \renewcommand*{\arraystretch}{1.2}
\begin{equation*}
\begin{array}{cc}
\lisp(n)\oplus\lisp(1),\quad
\lisp(n)\oplus\lie{u}(1),\quad
\lisp(n)\oplus\{0\},&l=n\\
\lisp(l)\oplus\lisp(n-l,1),&  0\leq l< n\\
\lie{u}(1)\oplus\lisp(n-1,1),\quad
\lisp(n-1,1).&\\
\end{array}
\end{equation*}
\endgroup\setlength\arraycolsep{2pt}

Hence, combining these results with Proposition~\ref{Prop:redspherical} and
Theorem~\ref{Thm2On}
 we obtain the following list of all spherical reductive algebraic
subalgebras of
$\lisp(n,1)$ up to conjugacy in $G$.

\begin{theorem}\label{ReductiveSymplectic}
All spherical, reductive subalgebras $\mathfrak{h}$ of $\lisp(n,1)$
where $n>1$ are (up to
conjugation
in $G$) 
one of the following
\setlength\arraycolsep{7pt}
\begingroup
\renewcommand*{\arraystretch}{1.2}
\begin{equation*}
\begin{array}{|c|c|}\hline
\lisp(n)\oplus\lisp(1)&\\
\lisp(n)\oplus\lie{u}(1)&\\
\lisp(n)\oplus\{0\}&\\
\lisp(l)\oplus\lisp(n-l,1)&  0\leq l< n\\
\lie{u}(1)\oplus\lisp(n-1,1)&\\
\lisp(n-1,1)&\\
\hline
\su(n,1)&\\
\liu(n,1)&\\
\hline
\end{array}
\end{equation*}
\endgroup\setlength\arraycolsep{2pt}
According to \cite[Table 2]{Be}, the symmetric subalgebras are
$\lisp(l)\oplus\lisp(n-l,1)$ for $0\leq l<n$, 
$\lisp(n)\oplus\lisp(1)$ and
$\liu(n,1)$.
\end{theorem}

\section{The exceptional group $G=F_4$}\label{exceptional}

\subsection{Reductive spherical subalgebras}

We want to classify the spherical reductive algebraic subalgebras 
$\lie{h}=\lie{k}_H\oplus\lie{p}_H$ of $\lie{g}=\lie{f}_4$. Using the classification 
of maximal reductive subalgebras of $\lie{g}$, Kr\"otz and Schlichtkrull have 
shown that a spherical reductive subalgebra of $\lie{g}$ must be contained in a 
symmetric one, see~\cite[Lemma 6.2]{KS}. According to~\cite[Table 2]{Be}, the 
symmetric proper subalgebras of $\f_4$ are $\so(9)$, $\so(8,1)$ and
$\lisp(2,1)\oplus\lisp(1)$. 

We first determine $\p_H$.

\begin{lemma}\label{q=pperp}
Let $G^\sigma=K^\sigma\exp(\p^\sigma)$ be a symmetric subgroup of $G$ (defined
as the fixed point set of an involution $\sigma$ that commutes with the Cartan
involution) and $H=K_H\exp(\p_H)$ a subgroup of $G^\sigma$. The adjoint 
$K_H$-action on $\p$ can only be transitive on the spheres in $\p_H^\perp\subset\p$ if 
$\p_H=\p^\sigma$.
\end{lemma}
\begin{proof}
Since $\mathfrak{h}$ is a subalgebra of $\g^\sigma$, $\p_H$ lies in $\p^\sigma$
which implies that $\p_H^\perp\subset\p=\p^\sigma\oplus\p^{-\sigma}$ is the direct 
sum of $\p^{-\sigma}$ and $\p_H^\perp\cap\p^\sigma$. Since the adjoint 
$\mathfrak{k}^\sigma$ action on $\p$ stabilizes $\p^\sigma$ and $\p^{-\sigma}$, 
$K_H$ can only act irreducibly on $\p_H^\perp\subset\p$ if $\p_H=\p^\sigma$.
\end{proof}

If $\g^\sigma$ is simple and non-compact, then we have 
$\lie{k}^\sigma=[\lie{p}^\sigma,\lie{p}^\sigma]$. We are now in the position to 
prove the main statement of this section.

\begin{theorem}\label{ReductiveExceptional}
Up to conjugation by an element of $G$, the following table exhausts the
spherical, reductive subalgebras $\mathfrak{h}=\mathfrak{k}_H\oplus\p_H$ of $\f_4$.
\setlength\arraycolsep{7pt}
\begingroup
\renewcommand*{\arraystretch}{1.2}
\begin{equation*}
\begin{array}{|c|}\hline
\f_4\\
\so(9)\\
\so(8,1)\\
\lisp(2,1)\\
\lisp(2,1)\oplus\lie{u}(1)\\
\lisp(2,1)\oplus\lisp(1)\\
\hline
\end{array}
\end{equation*}
\endgroup\setlength\arraycolsep{2pt}
According to~\cite[Table~2]{Be}, the symmetric ones are $\f_4$, $\so(9)$,
$\so(8,1)$ and
$\lisp(2,1)\oplus\lisp(1)$.
\end{theorem}

\begin{proof}
Let us first consider the case that $\g^\sigma$ equals $\f_4$ or $\so(8,1)$. 
Since $\p_H=\p^\sigma$ by Lemma~\ref{q=pperp} and $\g^\sigma$ is simple 
non-compact, we conclude that $\lie{h}$ has to contain 
$\lie{k}^\sigma=[\p^\sigma,\p^\sigma]$. Hence $\mathfrak{h}=\g^\sigma$ is the 
only spherical subalgebra of $\g$ that lies in $\g^\sigma$ in these two cases.

Let us now consider the case that $\mathfrak{h}\subset\so(9)$. In this case
$\lie{h}=\mathfrak{k}_H$ and $\p_H^\perp=\p\cong\RR^{16}$. We are therefore looking
for connected subgroups of $\Spin(9)$ that act transitively on the spheres in
$\RR^{16}\cong\p$ with respect to the adjoint action. Due to Theorem 
\ref{Thm2On} there are none except for $\Spin(9)$ itself. This gives 
$\mathfrak{h}=\so(9)$.

If $\g^\sigma=\lisp(2,1)\oplus\lisp(1)$ then 
$\mathfrak{k}^\sigma=\lisp(2)\oplus\lisp(1)\oplus\lisp(1)$ and 
$\dim_\RR(\p^\sigma)=8$. Since $\p_H=\p^\sigma$ by Lemma~\ref{q=pperp} and 
$\lisp(2,1)$ is simple non-compact, $\mathfrak{h}\subset\g^\sigma$ has to 
contain $\lisp(2,1)$. Hence $\mathfrak{h}=\lisp(2,1)\oplus\mathfrak{b}$, where 
$\mathfrak{b}$ is a subalgebra of $\lisp(1)$. Note that we have
\begin{equation*}
\lisp(2)\oplus\lisp(1)\subset\mathfrak{k}_H
\subset\lisp(2)\oplus\lisp(1)\oplus\lisp(1)=
\mathfrak{k}^\sigma.
\end{equation*}
The condition for $H$ to be spherical in $G$ is that $K_H$ acts transitively on
the spheres in $\p_H^\perp\cong\p^{-\sigma}\cong\RR^8$.

Let $\theta$ denote a Cartan involution on $\g$ that commutes with $\sigma$.
Then $\theta\sigma$ is another involution on $\g$ defining the Lie algebra
$\g^{\theta\sigma}=\mathfrak{k}^\sigma\oplus\p^{-\sigma}$ as fixed point set. 
Note that $\g^{\theta\sigma}$ has real rank $1$ since $\g$ has real rank $1$. 
Therefore $(K^\sigma)^\circ=\Sp(2)\times\Sp(1)\times\Sp(1)$ acts transitively
on the spheres in $\p^{-\sigma}$. In particular the adjoint 
$K^\sigma$-representation on $\p^{-\sigma}$ is irreducible. It is well known 
that then $\p^{-\sigma}$ decomposes into a tensor product of three $D$-vector 
spaces that are irreducible representations for $\Sp(2)$, $\Sp(1)$ and $\Sp(1)$ 
respectively, where $D\in\{\RR,\CC,\HH\}$. Since the adjoint 
$K^\sigma$-representation on $\p^{-\sigma}$ is not only irreducible but 
transitive on the spheres this tensor product has two factors that are equal to 
$D$. Using again the fact that $K^\sigma$ acts transitively on the sphere 
$S^7\in\RR^8\cong\p^{-\sigma}$, and the fact that 
$\dim_\RR\Sp(1)=3<7=\dim_\RR S^7$ we obtain the following three possibilities 
to decompose $\p^{-\sigma}$
\begin{equation*}
\p^{-\sigma}\cong\begin{cases}\RR^8\otimes\RR\otimes\RR,&\\
\CC^4\otimes\CC\otimes\CC,&\\
\HH^2\otimes\HH\otimes\HH.\end{cases}
\end{equation*}
Note that in each case $\Sp(2)$ acts by the standard action on 
$\p^{-\sigma}\cong\RR^8\cong\CC^4\cong\HH^2$. Hence, already 
$\Sp(2)\times\{e\}\times\{e\}$ acts transitively on the spheres in $\p_H^\perp$.
This implies that $\mathfrak{b}\subset\lisp(1)$ is arbitrary.
\end{proof}

\subsection{Non-reductive spherical subalgebras}

We start by analyzing the structure of a minimal parabolic subgroup $Q_0=MAN$ of
$G=F_4$. According to~\cite[p. 32-33]{Ar} we have that 
$\lie{n}=\lie{g}_\alpha\oplus\lie{g}_{2\alpha}$ is the restricted root
space decomposition, where $\dim\lie{g}_\alpha=8$ and 
$\dim\lie{g}_{2\alpha}=7$. Due to Corollary~\ref{Cor:nonredfinestructure} we 
have $\lie{n}_H=W^\perp\oplus\lie{g}_{2\alpha}$ where $W\subset\lie{g}_\alpha$ 
is a real subspace such that $M_H$ acts transitively on the spheres in $W$.

First we are going to identify the $M$-representation on $\g_\alpha$. Note that 
it has to be irreducible due to~\cite[Chapter~8.13]{Wo}.

\begin{lemma}\label{Lem:so(7)subsetso(8)}
The $M$-representation on $\lie{g}_\alpha$ is induced by the embedding
$\varphi:\so(7)\hookrightarrow\so(8)$ given by
\begin{multline*}
\left(\begin{smallmatrix}
0&-a&-b&-c&-d&-e&-f\\a&0&-g&-h&-i&-j&-k\\b&g&0&-l&-m&-n&-p\\
c&h&l&0&-q&-r&-s\\d&i&m&q&0&-t&-u\\e&j&n&r&t&0&-v\\
f&k&p&s&u&v&0\end{smallmatrix}\right)\mapsto\\
\left(\begin{smallmatrix}
0&-a+s-t&-b-r-u&-c-k+n&-d+j+p&-e-i-l&-f+h-m&g+q-v\\
a-s+t&0&-g+q-v&f-h-m&-e-i+l&d-j+p&-c-k-n&-b+r+u\\
b+r+u&g-q+v&0&-e+i-l&-f-h-m&c-k-n&d+j-p&a+s-t\\
c+k-n&-f+h+m&e-i+l&0&g-q-v&-b-r+u&a-s-t&-d-j-p\\
d-j-p&e+i-l&f+h+m&-g+q+v&0&-a-s-t&-b+r-u&c-k+n\\
e+i+l&-d+j-p&-c+k+n&b+r-u&a+s+t&0&-g-q-v&f+h-m\\
f-h+m&c+k+n&-d-j+p&-a+s+t&b-r+u&g+q+v&0&-e+i+l\\
-g-q+v&b-r-u&-a-s+t&d+j+p&-c+k-n&-f-h+m&e-i-l&0\end{smallmatrix}\right).
\end{multline*}
\end{lemma}

\begin{proof}
The embedding from $\Spin(7)$ into $\SO(8)$ is described in 
\cite[Chapter~1.5.3]{On} as follows. If $1\leq i<j\leq 7$, then 
$(-E_{ij}+E_{ji})$ forms a basis of $\so(7)$, where $E_{ij}$ is the matrix with 
a $1$ in line $i$ and column $j$ and zeros everywhere else.

We take the standard orthogonal basis $\{e_0,e_1,\ldots,e_7\}$ of octonions over 
$\RR$ with the following multiplication table.
\setlength\arraycolsep{2pt}\begingroup
\renewcommand*{\arraystretch}{1}
\begin{equation*}\begin{array}{|r|r|r|r|r|r|r|r|r|}
	\hline
	\times&\ e_0&\ e_1&\ e_2&\ e_3&\ e_4&\ e_5&\ e_6&\ e_7\\
  \hline
  \ e_0&\ e_0&\ e_1&\ e_2&\ e_3&\ e_4&\ e_5&\ e_6&\ e_7\\
  \hline
  \ e_1&\ e_1&-e_0&\ e_3&-e_2&\ e_5&-e_4&-e_7&\ e_6\\
  \hline
  \ e_2&\ e_2&-e_3&-e_0&\ e_1&\ e_6&\ e_7&-e_4&-e_5\\
  \hline
  \ e_3&\ e_3&\ e_2&-e_1&-e_0&\ e_7&-e_6&\ e_5&-e_4\\
  \hline
  \ e_4&\ e_4&-e_5&-e_6&-e_7&-e_0&\ e_1&\ e_2&\ e_3\\
  \hline
  \ e_5&\ e_5&\ e_4&-e_7&\ e_6&-e_1&-e_0&-e_3&\ e_2\\
  \hline
  \ e_6&\ e_6&\ e_7&\ e_4&-e_5&-e_2&\ e_3&-e_0&-e_1\\
  \hline
  \ e_7&\ e_7&-e_6&\ e_5&\ e_4&-e_3&-e_2&\ e_1&-e_0\\
	\hline
\end{array}\end{equation*}\endgroup
Note that the space of pure imaginary octonions is spanned by 
$\{e_1,\ldots,e_7\}$. The Lie algebra $\spin(7)$ is spanned by elements $e_ie_j$ 
with $1\leq i<j\leq 7$. The map $2(-E_{ij}+E_{ji})\mapsto e_ie_j$ for all $1\leq 
i<j\leq 7$ is an isomorphism from $\so(7)$ to $\spin(7)$. The map 
$\lambda:e_ie_j\mapsto (x\mapsto e_i(e_jx))$ is an embedding of $\spin(7)$ into 
$\so(8)$, where we choose the basis $\{e_1,\ldots,e_7,e_0\}$ of $\RR^8$. Direct 
calculation with the multiplication table shows that the map 
$\lambda:\spin(7)=\so(7)\to\so(8)$ is equal to the map $\varphi$ from the 
statement of this proposition. 

Note that the induced action of $\Spin(7)$ on $\RR^8$ is transitive on the
spheres. This follows since $\lambda(\spin(7))(e_0)=T_{e_0}(\Spin(7)\cdot e_0)$ 
coincides (by the multiplication table) with the purely imaginary octonions, 
which equals $T_{e_0}(S^7)$. Therefore the orbit of $\Spin(7)$ is open in $S^7$. 
Since $\Spin(7)$ and hence also its orbits are compact, it coincides with $S^7$. 
In particular, this representation is irreducible.

In order to finish the proof we note that a direct application of the Weyl 
Dimension formula (see e.g.~\cite[Chapter~V.6]{Kn}) shows that the irreducible 
$\so(7)$-representation in dimension $8$ is unique up to isomorphism. For the 
complete argument we refer the reader to~\cite[p.~78]{K}.
\end{proof}

Now that we identified the $M$-representation on $\g_\alpha$ (see 
Lemma~\ref{Lem:so(7)subsetso(8)}) we want to give a normal form for the 
subspace $\lie{n}_H\subset\lie{n}$. Before doing so we state the following
remark.
 
\begin{rem}\label{so4eqso3pluso3}
Recall that $\so(4)=\so(3)\oplus\so(3)$. Moreover, both $\lie{so}(3)$ factors 
are conjugate under an element in $\OO(4)$. The maps
\begin{align*}
\left(\begin{smallmatrix}
0&-a&-b&c\\
a&0&-c&-b\\
b&c&0&a\\
-c&b&-a&0\\
\end{smallmatrix}\right)\mapsto
\left(\begin{smallmatrix}
ia&-b-ic\\
b-ic&-ia\\
\end{smallmatrix}\right)\mapsto
ia+j(b-ic).
\end{align*}
define the Lie algebra isomorphisms from 
the ideal $\so(3)$ in $\so(4)$
to $\su(2)$ and $\lisp(1)$.
\end{rem}

\begin{lemma}\label{oEsubspaceofR8}
Let $\Spin(7)$ act irreducibly on $\RR^8$ and $W$ be an $l$-dimensional 
subspace. Then $W$ lies in the same $\Spin(7)$-orbit as $\RR^l\times\{0\}^{8-l}$ 
or as the real span of $e_1,e_2,e_3,xe_4+e_8$ for some $x\in\RR$, where 
$(e_1,\ldots,e_8)$ is the standard basis of $\RR^8$. The latter case can only 
occur if $l=4$.
\end{lemma}

\begin{proof}
The proof of Lemma~\ref{Lem:so(7)subsetso(8)} shows that $\Spin(7)$ acts 
transitively on the spheres in $\RR^8$. In particular all its stabilizers are 
conjugate to each other. The stabilizer of $\Spin(7)$ in $e_8:=e_0$ is equal to 
$G_2$ (\cite[Chapter 1 \S5]{On}). Due to Onishchik's classification (\cite[Table 
8]{On}) the group $G_2$ acts transitively on the spheres in $\RR^7$ and has 
stabilizer equal to $\SU(3)$. The group $\SU(3)$ acts transitively on the sphere 
$S^5\subset\RR^6\cong\CC^3$ but its isotropy in one point (being isomorphic to 
$\SU(2)$) acts reducibly on $\RR^5\cong (i\RR)\times\CC^2$.

Now let $(v_1,\ldots,v_l)$ be an orthonormal basis of $W$. If $l\leq 3$ we can 
map $v_1$ to $e_1$ with $\Spin(7)$, $v_2$ to $e_2$ with $G_2$ and $v_3$ to $e_3$ 
with $\SU(3)$.

If $l>4$ the orthogonal complement of $W$ has dimension less or equal to $3$.
By the previous argument we can map the orthogonal complement of $W$ to 
$\{0\}^l\times\RR^{8-l}$, thus mapping $W$ to the span of $e_1,\ldots,e_l$.

Now let $l=4$. As before we use $\Spin(7)$, $G_2$ and $\SU(3)$ to map $v_1$ to 
$e_1$, $v_2$ to $e_2$ and $v_3$ to $e_3$. Recall that the action of $\Spin(7)$ 
on $\RR^8$ is induced by the map $\varphi:\so(7)\to\so(8)$. A direct calculation 
shows that
\begin{align*}
\left(\left(\varphi(\so(7))_{e_1}\right)_{e_2}\right)_{e_3}=
\left\{\left(\begin{smallmatrix}
0&0&0&0&0&0&0&0\\
0&0&0&0&0&0&0&0\\
0&0&0&0&0&0&0&0\\
0&0&0&0&-v&u&-t&0\\
0&0&0&v&0&-t&-u&0\\
0&0&0&-u&t&0&-v&0\\
0&0&0&t&u&v&0&0\\
0&0&0&0&0&0&0&0
\end{smallmatrix}\right):t,u,v\in\RR\right\}
\stackrel{R. \ref{so4eqso3pluso3}}\cong
\so(3).
\end{align*}
Therefore $\left(\left(\varphi(\so(7))_{e_1}\right)_{e_2}\right)_{e_3}
\cong\so(3)\cong\su(2)$ acts trivially on $\RR e_8$ and transitively on the 
spheres in $\{0\}^3\times\RR^4\times\{0\}$.

If $v_4$ is an element of $\{0\}^3\times\RR^4\times\{0\}$ then $W$ lies in the 
same $\Spin(7)$-orbit as the span of $e_1,\ldots,e_4$. 

If $v_4$ is in $\{0\}^3\times\RR^5\backslash (\{0\}^3\times\RR^4\times\{0\})$ we 
can use the described $\SU(2)$-action to map it to $ae_4+be_8$, where $a\in\RR$ 
and $b\in\RR\backslash\{0\}$. The assumption follows since 
$\tfrac{1}{b}(ae_4+be_8)=\tfrac{a}{b}e_4+e_8=:xe_4+e_8$ holds for some 
$x\in\RR$.
\end{proof}

Now we arrive at the main result of this section. We assume 
$\lie{n}_H=W\oplus\g_{2\alpha}$, where $W$ is a real subspace of $\g_\alpha$.
Due to Lemma~\ref{oEsubspaceofR8} we may assume (after conjugation in $M$) 
that $\lie{n}_H=\lie{n}_l:=(\RR^l\times\{0\}^{8-l})\oplus\g_{2\alpha}$ or
$\lie{n}_H=\lie{n}_{4,x}:=\langle e_1,e_2,e_3, 
xe_4+e_8\rangle_\RR\oplus\g_{2\alpha}$ for some $x\in\RR$, where the latter case 
can only occur if $\dim_\RR(W)=4$. We consider the map 
$\varphi:\so(7)\hookrightarrow\so(8)$ defined in 
Lemma~\ref{Lem:so(7)subsetso(8)}. Its restriction to 
$\mathcal{N}_{\lie{m}}(W)\subset\lie{m}=\so(7)$ yields a map 
$\varphi|_{\mathcal{N}_{\lie{m}}(W)}:\mathcal{N}_{\lie{m}}(W) 
\to\mathcal{N}_{\so(8)}(W)$ and the action of $\mathcal{N}_M(W)$ on $W^\perp$ is 
induced by the standard action of $\varphi(\mathcal{N}_{\lie{m}}(W))\subset
\mathcal{N}_{\so(8)}(W)$ on $W^\perp$.

\begin{theorem}\label{NonredExc}
Every non-reductive spherical algebraic subalgebra of $\f_4$ is $G$-conjugate 
to one in the following table where $\lie{l}_1\subset\lisp(1)$ and 
$\lie{l}_2\subsetneq\lisp(1)$ are arbitrary  (under the condition that the 
maximal compact subalgebra is a Lie algebra, see 
Remark~\ref{non-reductive,embedding}).
\setlength\arraycolsep{7pt}
\begingroup
\renewcommand*{\arraystretch}{1.2}
\begin{equation*}
\begin{array}{|c|c|}\hline
\lie{l}_H\oplus\lie{n} & \lie{l}_H\subset\lie{m}\oplus\lie{a}\text{ arbitrary}\\
\hline
\so(7)\oplus\lie{a}\oplus\lie{n}_0  & \\
\hline
\lie{g}_2\oplus\lie{a}\oplus\lie{n}_1 & \\
\hline
\liu(3)\oplus\lie{a}\oplus\lie{n}_2 & \\
\su(3)\oplus\lie{a}\oplus\lie{n}_2 & \\
\hline
\so(4)\oplus\lie{a}\oplus\lie{n}_l & l=4,5\\
\lisp(1)\oplus\lie{l}_2\oplus\lie{a}\oplus\lie{n}_l &l=4,5 \\
\lie{l}_2\oplus\lisp(1)\oplus\lie{a}\oplus\lie{n}_l & l=4\\
\hline
\lie{l}_1\oplus\so(4)\oplus\lie{a}\oplus\lie{n}_{4,0} & \\
\lie{l}_1\oplus\lisp(1)\oplus\lie{l}_2\oplus\lie{a}\oplus\lie{n}_{4,0} &\\
\lie{l}_1\oplus\lie{l}_2\oplus\lisp(1)\oplus\lie{a}\oplus\lie{n}_{4,0} &\\
\hline
\so(4)\oplus\lie{a}\oplus\lie{n}_{4,x} &x\neq 0\\
\lisp(1)\oplus\lie{l}_2\oplus\lie{a}\oplus\lie{n}_{4,x} &x\neq 0\\
\lie{l}_2\oplus\lisp(1)\oplus\lie{a}\oplus\lie{n}_{4,x} &x\neq 0\\
\hline
\lie{u}(1)\oplus\lie{k}_2\oplus\lie{a}\oplus\lie{n}_6 &\\
\hline
\lie{m}'\oplus\lie{a}\oplus\lie{n}_7 & 
\lie{m}'\subset\g_2=\mathcal{N}_{\lie{m}}(\lie{n}_7)\text{ arbitrary}\\
\hline
\end{array}
\end{equation*}
\endgroup\setlength\arraycolsep{2pt}
where $\lie{k}_2$ is isomorphic to $\{0\}$, $\lie{u}(1)^k$ for some $1\leq 
k\leq 2$, $\su(2)\oplus\lie{l}_1$ or $\su(3)$.
\end{theorem}

\begin{proof}
We assume $\mathfrak{n}_H=W\oplus\g_{2\alpha}$. Lemma~\ref{oEsubspaceofR8} shows 
that $W$ may (after conjugation in $M$) be chosen as $\RR^l\times\{0\}^{8-l}$ or 
as the real span of $e_1,e_2,e_3, xe_4+e_8$ if $\dim_\RR(W)=4$, where $x\in\RR$. 
The case that $W$ is the real span of $e_1,e_2,e_3,xe_4+e_8$ will be considered 
at the end of this proof. Due to Proposition~\ref{Prop:nonredspherical} it 
suffices to consider the case $\dim W=l<8$.

If $W=\RR^l\times\{0\}^{8-l}$ we have
\begin{equation*}
\lie{n}_H=\mathfrak{n}_l:=(\RR^l\times\{0\}^{8-l})\oplus\g_{2\alpha},
\end{equation*}
where $0\leq l\leq 7$ is the dimension of $\mathfrak{n}_H\cap\g_{\alpha}$. 
A direct calculation yields
\begin{align*}
\mathcal{N}_M(\mathfrak{n}_l)
&=\left\{A\in\Spin(7):\varphi(A)=\left(\begin{smallmatrix}A_1&0\\0&A_2
\end{smallmatrix}
\right), A_1\in\RR^{l\times l}, A_2\in\RR^{(8-l)\times (8-l)}\right\}
\end{align*}
and the action of $\mathcal{N}_M(\mathfrak{n}_l)$ on $\mathfrak{n}_l^\perp=
\{0\}^l\times\mbb{R}^{8-l}$ is given by 
\begin{equation*}
\varphi(A)\cdot\left(\begin{smallmatrix}0\\v\end{smallmatrix}
\right)=\left(\begin{smallmatrix}A_1&0\\0&A_2\end{smallmatrix}\right)\cdot\left(
\begin{smallmatrix}0\\v\end{smallmatrix}\right)=\left(\begin{smallmatrix}0\\A_2v
\end{smallmatrix}\right).
\end{equation*}
Onishchik's classification Theorem~\ref{Thm2On} shows that the subgroup $M_H$ of
$\mathcal{N}_M(\mathfrak{n}_l)$ can only act transitively on the spheres in
$\mathfrak{n}_l^\perp$ if the projection of 
$\varphi(\lie{m}_H)\subset\so(l)\oplus\so(8-l)$ onto $\so(8-l)$ is 
one of the following Lie algebras, where $\lie{l}_2\subsetneq\lisp(1)$ is 
arbitrary.
\setlength\arraycolsep{5pt}\begingroup
\renewcommand*{\arraystretch}{1}
\begin{equation*}\begin{array}{|c|c|}
	\hline
	\ l\ & \\
	\hline
	0&\so(8),\so(7),\liu(4),\su(4),\lisp(2)\oplus\mathfrak{u}(1),\lisp(2)\\
	1&\so(7),\g_2\\
	2&\so(6),\liu(3),\su(3)\\
	3&\so(5)\\
	4&\so(4),\lisp(1)\oplus\lie{l}_2,\lie{l}_2\oplus\lisp(1)\\
	5&\so(3)\\
	6&\so(2)\\
	\hline
\end{array}
\end{equation*}
\endgroup\setlength\arraycolsep{2pt}
Since $\varphi(\lie{m}_H)\subset\varphi(\mathcal{N}_\lie{m}(\mathfrak{n}_l))$ 
holds, our next step is to take each of the Lie algebras in the table above and 
check if it is contained in the projection onto $\so(8-l)$ of 
\begin{equation*}
\varphi(\mathcal{N}_\lie{m}(\mathfrak{n}_l))=\left\{\left(
\begin{smallmatrix}
A_1&0\\0&A_2\end{smallmatrix}\right)\in\varphi(\so(7)): A_1\in\RR^{l\times l},
A_2\in\RR^{(8-l)\times (8-l)}\right\}.
\end{equation*}
We start with $l=0$. Note that this implies $\lie{n}_H=\g_{2\alpha}$ and
$\mathcal{N}_\lie{m}(\mathfrak{n}_0)=\lie{m}=\so(7)$. It is clear that $\so(8)$ 
is not contained in $\varphi(\so(7))\cong\so(7)$, since the dimension of 
$\so(7)$ is smaller than the dimension of $\so(8)$. Due to Onishchik's 
classification Theorem~\ref{Thm2On} there exists no subgroup of $\Spin(7)$ that 
acts transitively on the spheres in $\RR^8$. In particular $\Spin(7)$ is the 
only subgroup of $\varphi(\mathcal{N}_M(\mathfrak{n}_0))\cong\Spin(7)$ that acts 
transitively on the spheres in $\mathfrak{n}$ (with respect to the action 
defined by $\varphi$).

Now let $l=1$. Since $G_2=(\Spin(7))_{e_0}$ (with respect to the action defined
by $\varphi$, i.e. $e_0=e_8$) it is clear that $\lie{g}_2$ is contained in
$\varphi(\mathcal{N}_{\lie{m}}(\mathfrak{n}_1))\cong\g_2$. The dimension 
$\dim(\varphi(\mathcal{N}_\lie{m}(\mathfrak{n}_1)))=\dim(\g_2)=14$
is smaller than the dimension of $\so(7)$. Therefore the projection of 
$\varphi(\mathcal{N}_\mathfrak{m}(\mathfrak{n}_1))$ onto $\so(7)$ can not
contain $\so(7)$.

If $l=2$ then 
\begin{equation*}
\varphi(\mathcal{N}_\mathfrak{m}(\mathfrak{n}_2))=\left\{
\left(\begin{smallmatrix}
0&a-s+t&0&0&0&0&0&0\\
-a+s-t&0&0&0&0&0&0&0\\
0&0&0&i&-h&-k&j&a\\
0&0&-i&0&-v&u&s&-j\\
0&0&h&v&0&t&-u&-k\\
0&0&k&-u&-t&0&-v&h\\
0&0&-j&-s&u&v&0&i\\
0&0&-a&j&k&-h&-i&0
\end{smallmatrix}\right):
a,h,i,j,k,s,t,u,v\in\RR\right\}.
\end{equation*}
In particular its projection onto $\so(6)$ cannot contain $\so(6)$ for 
dimensional reasons. If we conjugate its projection onto $\so(6)$ with
the matrix 
$\left(\begin{smallmatrix}
1&0&0&0&0&0\\
0&0&0&0&0&1\\
0&0&0&-1&0&0\\
0&0&1&0&0&0\\
0&0&0&0&1&0\\
0&1&0&0&0&0
\end{smallmatrix}\right)$ in $\OO(6)$ we obtain
\begin{equation*}
\left\{\left(\begin{smallmatrix}
0&a&k&-h&j&i\\
-a&0&h&k&-i&j\\
-k&-h&0&t&v&u\\
h&-k&-t&0&-u&v\\
-j&i&-v&u&0&-s\\
-i&-j&-u&-v&s&0
\end{smallmatrix}\right):
a,h,i,j,k,s,t,u,v\in\RR\right\},
\end{equation*}
which is the natural embedding of $\liu(3)$ into $\so(6)$. Note that this 
implies $\mathcal{N}_{\lie{m}}(\lie{n}_2)\cong\liu(3)$. Since $\su(3)$ is a 
subalgebra of $\liu(3)$ this shows that the projection of
$\varphi(\mathcal{N}_\mathfrak{m}(\mathfrak{n}_2))$ onto $\so(6)$ 
contains $\liu(3)$ and $\su(3)$.

If $l=3$ the projection of
\begin{equation*}
\varphi(\mathcal{N}_\mathfrak{m}(\mathfrak{n}_3))=\left\{
\left(\begin{smallmatrix}
0&s-t&-r-u&0&0&0&0&0\\
-s+t&0&q-v&0&0&0&0&0\\
r+u&-q+v&0&0&0&0&0&0\\
0&0&0&0&-q&-r&-s&0\\
0&0&0&q&0&-t&-u&0\\
0&0&0&r&t&0&-v&0\\
0&0&0&s&u&v&0&0\\
0&0&0&0&0&0&0&0
\end{smallmatrix}\right):
q,r,s,t,u,v\in\RR\right\}.
\end{equation*}
onto $\so(5)$ does not contain $\so(5)$.

If $l=4$ the projection of
\begin{equation*}
\varphi(\mathcal{N}_\mathfrak{m}(\mathfrak{n}_4))=\left\{
\left(
\begin{smallmatrix}
0&-t&-u&n&0&0&0&0\\
t&0&-v&-m&0&0&0&0\\u&v&0&i&0&0&0&0\\-n&m&-i&0&0&0&0&0\\
0&0&0&0&0&-t&-u&n\\0&0&0&
0&t&0&-v&-m\\0&0&0&0&u&v&0&i\\0&0&0&0&-n&m&-i&0
\end{smallmatrix}\right):
i,m,n,t,u,v\in\RR\right\}\cong\so(4)
\end{equation*}
onto $\so(4)$ is a Lie algebra isomorphism. This shows furthermore that 
the projection of $\varphi(\mathcal{N}_\mathfrak{m}(\mathfrak{n}_4))$
onto $\so(4)$ also contains $\lisp(1)\oplus\lie{l}_2$ and 
$\lie{l}_2\oplus\lisp(1)$, where $\lie{l}_2\subsetneq\lisp(1)$ is arbitrary.

If $l=5$ the projection of 
\begin{equation*}\varphi(\mathcal{N}_\mathfrak{m}(\mathfrak{n}_5))=
\left\{\left(\begin{smallmatrix}
0&0&0&0&0&0&0&0\\
0&0&-g&-m&l&0&0&0\\
0&g&0&i&-h&0&0&0\\
0&m&-i&0&-q&0&0&0\\
0&-l&h&q&0&0&0&0\\
0&0&0&0&0&0&-g-q&h-m\\
0&0&0&0&0&g+q&0&i+l\\
0&0&0&0&0&-h+m&-i-l&0\end{smallmatrix}\right):
g,h,i,l,m,q\in\RR\right\}\cong\so(4)\end{equation*}
onto $\so(3)$ is surjective. The projection of $f_1(\so(3))\subset\so(4)\cong
\varphi(\mathcal{N}_\mathfrak{m}(\mathfrak{n}_5))$ (i.e., $m=-h$, $i=l$, $g=q$)
onto $\so(3)$ is still surjective, while the projection of
$f_2(\so(3))\subset\so(4)\cong 
\varphi(\mathcal{N}_\mathfrak{m}(\mathfrak{n}_5))$ (i.e., $m=h$, $i=-l$, $g=-q$)
onto $\so(3)$ equals $\{0\}$. Therefore $\lie{m}_H$ is isomorphic to
$\so(4)$ or $\lisp(1)\oplus\lie{l}_2$ for arbitrary 
$\lie{l}_2\subsetneq\lisp(1)$ in this case.

If $l=6$ the projection of 
\begin{equation*}\varphi(\mathcal{N}_\mathfrak{m}(\mathfrak{n}_6))=
\left\{\left(\begin{smallmatrix}
0&-a&-b&-c&-d&-e&0&0\\
a&0&-g&-h&-w&d&0&0\\
b&g&0&-l&-h&c&0&0\\
c&h&l&0&g&-b&0&0\\
d&w&h&-g&0&-a&0&0\\
e&-d&-c&b&a&0&0&0\\
0&0&0&0&0&0&0&e-w-l\\
0&0&0&0&0&0&-e+w+l&0\end{smallmatrix}\right):
a,b,c,d,e,g,h,l,w\in\RR\right\}\end{equation*}
onto $\so(2)$ is surjective. A direct calculation shows furthermore that
\begin{align*}
\Ad\left(\begin{smallmatrix}1&&&&&&&\\&&&&&1&&\\
&&&-1&&&&\\&&-1&&&&&\\&&&&1&&&\\&1&&&&&&\\
&&&&&&1&\\&&&&&&&1\end{smallmatrix}\right)
\left(\varphi(\mathcal{N}_\mathfrak{m}(\mathfrak{n}_6))\right)
&=\left\{\left(\begin{smallmatrix}
0&-e&c&b&-d&-a&0&0\\
e&0&-b&c&a&-d&0&0\\
-c&b&0&l&-g&-h&0&0\\
-b&-c&-l&0&h&-g&0&0\\
d&-a&g&-h&0&w&0&0\\
a&d&h&g&-w&0&0&0\\
0&0&0&0&0&0&0&e-w-l\\
0&0&0&0&0&0&-e+w+l&0\end{smallmatrix}\right)\in\so(8)\right\}\\
&\cong\left\{\left(\begin{smallmatrix}
ie&c-ib&-d+ia&0\\
-c-ib&-il&-g+ih&0\\
d+ia&g+ih&-iw&0\\
0&0&0&-ie+il+iw
\end{smallmatrix}\right)\in\liu(4)\right\}\\
&=\left\{\left(\begin{smallmatrix}
A&0\\
0&-\tr(A)
\end{smallmatrix}\right)\in\liu(4): A\in\liu(3)\right\}\cong\liu(3)
\end{align*}
holds. Therefore $\lie{m}_H$ is any subalgebra of $\liu(3)$ that does
not lie in $\su(3)$. If we denote the projection 
$\liu(3)\to\mathcal{Z}(\liu(3))=\lie{u}(1)$ by $\pi$ then its restriction
to $\lie{m}_H$ is surjective. The kernel $\ker \pi|_{\lie{m}_H}=:\lie{k}_2$
is an ideal in $\lie{m}_H$ and equals $\lie{m}_H\cap\su(3)$. We obtain that
the reductive subalgebra $\lie{m}_H$ of $\liu(3)$ decomposes as
$\lie{m}_H=\lie{k}_1\oplus\lie{k}_2$, where $\lie{k}_1\cong\lie{u}(1)$. Any
subalgebra of $\su(3)$ has at most dimension $8$ and is at most
of rank $2$. The simple algebras of rank $1$ and $2$ are 
$\su(2)\cong\so(3)=\lisp(1)$, $\su(3)$, $\so(5)$ and $\g_2$. Since 
$\dim_\RR(\so(5))=10$ and $\dim_\RR(\g_2)=14$ they can not be subalgebras of 
$\lie{k}_2$. We therefore obtain the following list of possibilities for 
$\lie{m}_H=\lie{k}_1\oplus\lie{k}_2$ up to isomorphism:
\begin{gather*}
\lie{u}(1),\quad\lie{u}(1)\oplus\lie{u}(1),\quad\lie{u}(1)\oplus\lie{u}(1)\oplus
\lie{u}(1)\\
\lie{u}(1)\oplus\su(2),\quad\lie{u}(1)\oplus\lie{u}(1)\oplus\su(2),\quad
\lie{u}(1)\oplus\su(2)\oplus\su(2),\\
\lie{u}(1)\oplus\su(3),
\end{gather*}
i.e., $\lie{m}_H$ is isomorphic to $\lie{u}(1)^k$, $\liu(2)\oplus\lie{l}_1$ or 
$\liu(3)$, where $1\leq k\leq 3$ and $\lie{l}_1\subset\su(2)\cong\lisp(1)$ is 
arbitrary.

If $l=7$ then $\lie{n}_7$ has codimension $1$ in $\lie{n}$. Since 
$\mathcal{N}_{\lie{m}}(\lie{n}_7)=\mathcal{N}_{\lie{m}}(\lie{n}_7^\perp)$ and
$\lie{n}_7^\perp=\RR e_8$ we obtain $\mathcal{N}_{\lie{m}}(\lie{n}_7)
=\mathcal{N}_{\lie{m}}(\RR e_8)=\lie{g}_2$. Due to Proposition 
\ref{Prop:nonredspherical} the subalgebra 
$\lie{m}_H\subset\mathcal{N}_{\lie{m}}(\lie{n}_7)=\g_2$ is arbitrary in this 
case.

To summarize, we have seen that if $\lie{h}$ is a spherical algebraic subalgebra 
of $\g$ then $\lie{m}_H$ has to be conjugate to one of the following, where 
$\lie{l}_1\subset\lisp(1)$ and $\lie{l}_2\subsetneq\lisp(1)$ 
are arbitrary.
\begingroup
\renewcommand*{\arraystretch}{1}
\begin{equation*}\begin{array}{|c|c|}
	\hline
	\,\ l\,\ &\lie{m}_H \\
	\hline
	0&\so(7)\\
	1&\g_2,\\
	2&\liu(3),\ \su(3)\\
	4&\so(4),\ \lisp(1)\oplus\lie{l}_2,\ \lie{l}_2\subset\lisp(1)\\
	5&\so(4),\ \lisp(1)\oplus\lie{l}_2\\
	6&\lie{u}(1)\oplus\lie{k}_2\\
	7&\text{ any subalgebra of }\g_2\\
	\hline
\end{array}\end{equation*}
\endgroup
and $\lie{k}_2$ is isomorphic to $\{0\}$, $\lie{u}(1)^k$ where $1\leq k\leq2$, 
$\su(2)\oplus\lie{l}_1$ or $\su(3)$. Now let us consider the case that 
$\lie{n}_H=\lie{n}_{4,x}$, i.e., $\lie{n}_H=W\oplus\g_{2\alpha}$ and 
$W\subset\g_\alpha=\RR^8$ is the real span of $e_1,e_2,e_3,xe_4+e_8$ for some 
$x\in\RR$. Then $\lie{n}_{4,x}^\perp\subset\g_\alpha$ is the real span of 
$e_5,e_6,e_7,e_4-xe_8$ and
\begin{align*}
&\mathcal{N}_{\so(8)}(W)=\left\{\left(\begin{smallmatrix} 
0&*&*&-v_1&0&0&0&-w_1\\
*&0&*&-v_2&0&0&0&-w_2\\
*&*&0&-v_3&0&0&0&-w_3\\ 
v_1&v_2&v_3&0&z_1&z_2&z_3&0\\
0&0&0&-z_1&0&*&*&-y_1\\
0&0&0&-z_2&*&0&*&-y_2\\ 
0&0&0&-z_3&*&*&0&-y_3\\
w_1&w_2&w_3&0&y_1&y_2&y_3&0\end{smallmatrix}\right)
\in\so(8),
v_i=xw_i, -xz_i=y_i\ \forall\ 1\leq i\leq 3\right\}
\end{align*}
holds. A direct calculation shows
\begin{align*}
\mathcal{N}_\mathfrak{m}(\mathfrak{n}_{4,0})=
\left\{\left(\begin{smallmatrix}
0&-a&-b&0&0&0&0\\a&0&-g&0&0&0&0\\b&g&0&0&0&0&0\\
0&0&0&0&-q&-r&-s\\0&0&0&q&0&-t&-u\\0&0&0&r&t&0&-v\\
0&0&0&s&u&v&0\end{smallmatrix}\right):a,b,g,q,r,s,t,u,v\in\RR\right\}
\cong\so(3)\oplus\so(4)
\end{align*}
and
\begin{align*}
\mathcal{N}_\mathfrak{m}(\mathfrak{n}_{4,x\neq 0})=
\left\{\left(\begin{smallmatrix}
0&-a&-b&0&0&\tfrac{x}{2}(a-t)&\tfrac{x}{2}(b-u)\\a&0&-g&0&-\tfrac{x}{2}
(a-t)&0&\tfrac{x}{2}(g-v)\\b&g&0&0&-\tfrac{x}{2}(b-u)&-\tfrac{x}{2}(g-v)&0\\
0&0&0&0&0&0&0\\0&\tfrac{x}{2}(a-t)&\tfrac{x}{2}(b-u)&0&0&-t&-u\\-\tfrac{x}{2}
(a-t)&0&\tfrac{x}{2}(g-v)&0&t&0&-v\\
-\tfrac{x}{2}(b-u)&\tfrac{x}{2}(g-v)&0&0&u&v&0\end{smallmatrix}\right):a,b,g,t,u
,v\in\RR\right\}
\end{align*}
in the case that $x\neq 0$.
In order to understand the Lie algebra structure of 
$\mathcal{N}_\mathfrak{m}(\mathfrak{n}_{4,x\neq0})$ and the action of 
$\mathcal{N}_\mathfrak{m}(\mathfrak{n}_{4,x})$ on 
$\lie{n}_{4,x}^\perp=W^\perp\subset\RR^8$ for arbitrary $x$ we calculate the 
image of $\mathcal{N}_\mathfrak{m}(\mathfrak{n}_{4,x})$ under $\varphi$.
\begin{align*}
\varphi(\mathcal{N}_\mathfrak{m}(\mathfrak{n}_{4,x}))
&=\begin{cases}
\left\{\left(\begin{smallmatrix}
0&a&b&0&0&0&0&c+d-i\\
-a&0&c&0&0&0&0&-b+e+h\\
-b&-c&0&0&0&0&0&a+f-g\\
0&0&0&0&d&e&f&0\\
0&0&0&-d&0&g&h&0\\
0&0&0&-e&-g&0&i&0\\
0&0&0&-f&-h&-i&0&0\\
-c-d+i&b-e-h&-a-f+g&0&0&0&0&0
\end{smallmatrix}\right)\in\so(8)\right\},&x=0,\vspace{0,2cm}\\
\left\{\left(\begin{smallmatrix}
0&d&e&xa&0&0&0&a\\
-d&0&f&xb&0&0&0&b\\
-e&-f&0&xc&0&0&0&c\\
-xa&-xb&-xc&0&a&b&c&0\\
0&0&0&-a&0&d&e&xa\\
0&0&0&-b&-d&0&f&xb\\
0&0&0&-c&-e&-f&0&xc\\
-a&-b&-c&0&-xa&-xb&-xc&0
\end{smallmatrix}\right):a,b,c,d,e,f\in\RR\right\},&x\neq 0.
\end{cases}
\end{align*}
Let $\pi$ denote the projection of $\mathcal{N}_{\so(8)}(W)$ to 
$\so(W^\perp)\cong\so(4)$ given by
\begin{equation*}
\left(\begin{smallmatrix}
0&*&*&-xg&0&0&0&-g\\
*&0&*&-xh&0&0&0&-h\\
*&*&0&-xj&0&0&0&-j\\
xg&xh&xj&0&a&b&c&0\\
0&0&0&-a&0&d&e&ax\\
0&0&0&-b&-d&0&f&bx\\
0&0&0&-c&-e&-f&0&cx\\
g&h&j&0&-ax&-bx&-cx&0
\end{smallmatrix}\right)
\mapsto
\left(\begin{smallmatrix}
0&0&0&0&0&0&0&0\\
0&0&0&0&0&0&0&0\\
0&0&0&0&0&0&0&0\\
0&0&0&0&a&b&c&0\\
0&0&0&-a&0&d&e&ax\\
0&0&0&-b&-d&0&f&bx\\
0&0&0&-c&-e&-f&0&cx\\
0&0&0&0&-ax&-bx&-cx&0
\end{smallmatrix}\right).
\end{equation*}
Then its restriction $\pi|_{\varphi(\mathcal{N}_{\lie{m}}(\lie{n}_{4,x}))}:
\varphi(\mathcal{N}_{\lie{m}}(\lie{n}_{4,x}))\to\so(W^\perp)\cong\so(4)$ is a 
surjective Lie algebra homomorphism which is an isomorphism if $x\neq 0$. 
This implies that the action of $\mathcal{N}_\lie{m}(\lie{n}_{4,x})$ on 
$W^\perp\cong\RR^4$ is the standard action of $\so(4)$ on $\RR^4$. Hence 
$\mathcal{N}_{M}(\lie{n}_{4,x})$ acts transitively on the spheres in $W^\perp$ 
and the only proper subgroups of $N_M(\lie{n}_{4,x})$ that act transitively on 
the spheres in $W^\perp$ have Lie algebras that lie in the preimage
\begin{align*}
\left(\pi|_{\varphi(\mathcal{N}_\lie{m}(\lie{n}_{4,x}))}\right)^{-1}(\underbrace
{\lisp(1)\oplus\lie{l}_2}_{\subset\so(4)})\ 
\text{  or  }
\left(\pi|_{\varphi(\mathcal{N}_\lie{m}(\lie{n}_{4,x}))}\right)^{-1}(\underbrace
{\lie{l}_2\oplus\lisp(1)}_{\subset\so(4)}),
\end{align*}
where $\lie{l}_2\subset\lisp(1)$ is arbitrary.
\end{proof}

\end{document}